\documentclass[10pt]{amsart}

\usepackage{amsmath, amssymb, graphicx}
\usepackage{pst-all}
\usepackage{xspace}

\usepackage{color}

\allowdisplaybreaks

\theoremstyle{theorem}

\newtheorem{coro}[equation]{Corollary}
\newtheorem{fact}[equation]{Fact}
\newtheorem{lemm}[equation]{Lemma}
\newtheorem{prop}[equation]{Proposition}

\newtheorem*{conjA}{Conjecture~$\ConjA$}

\newtheorem*{conjB}{Conjecture~$\ConjB$}
\newtheorem*{conjC}{Conjecture~$\ConjC$}
\newtheorem*{conjCunif}{Conjecture~$\ConjCunif$}

\theoremstyle{definition}
\newtheorem{defi}[equation]{Definition}
\newtheorem{exam}[equation]{Example}
\newtheorem{nota}[equation]{Notation}
\newtheorem{ques}[equation]{Question}
\newtheorem{rema}[equation]{Remark}

\psset{unit=1mm}
\psset{fillcolor=white}
\psset{dotsep=1.5pt}
\psset{dash=1.5pt 1.5pt}
\psset{linewidth=0.8pt}
\psset{arrowsize=3.5pt}
\psset{doublesep=1pt}
\psset{coilwidth=1.2mm}
\psset{coilheight=2}
\psset{coilaspect=20}
\psset{coilarm=2}
\newpsstyle{double}{linewidth=0.5pt,doubleline=true,arrowsize=5pt}
\newpsstyle{doubleexist}{linewidth=0.5pt,doubleline=true,arrowsize=5pt,linestyle=dashed}
\newpsstyle{etc}{linestyle=dotted}
\newpsstyle{exist}{linestyle=dashed}
\newpsstyle{thin}{linewidth=0.5pt}
\newpsstyle{thinexist}{linewidth=0.5pt, linestyle=dashed}
\newpsstyle{back}{linewidth=3mm,linearc=1}
\def\AArrow(#1,#2){\ncline[nodesep=0.3mm, linewidth=0.8pt, border=1.2pt]{->}{#1}{#2}}
\def\BArrow(#1,#2){\ncline[linecolor=blue, linewidth=1.2pt, linestyle=dotted, nodesep=0.3mm, border=1.2pt]{->}{#1}{#2}}
\def\CArrow(#1,#2){\ncline[linecolor=red, nodesep=0.3mm, linewidth=0.8pt, border=1.2pt]{->}{#1}{#2}}
\definecolor{color1}{rgb}{.88,.88,.88}
\definecolor{color2}{rgb}{1,.75,.5}
\definecolor{color2}{rgb}{1,.87,.83}
\definecolor{color4}{rgb}{.9,.9,.9}
\definecolor{color2}{rgb}{0.85,.92,1}
\definecolor{color20}{rgb}{0.9,.95,1}
\definecolor{color21}{rgb}{0.8,.9,1}
\definecolor{color22}{rgb}{0.75,.85,1}
\definecolor{color3}{rgb}{0,0,1}
\definecolor{color2}{rgb}{1,.85,0.85}
\definecolor{color20}{rgb}{.88,.88,.88}
\definecolor{color21}{rgb}{1,.8,.8}
\definecolor{color22}{rgb}{1,.75,.75}
\definecolor{color3}{rgb}{1,0,0}

\def\BPoint(#1,#2,#3){\cnode[style=thin,fillcolor=black,fillstyle=solid](#1,#2){0.5}{#3}}
\def\GPoint(#1,#2,#3){\cnode[style=thin,fillcolor=lightgray,fillstyle=solid](#1,#2){0.5}{#3}}
\def\WPoint(#1,#2,#3){\cnode[style=thin,fillcolor=white,fillstyle=solid](#1,#2){0.7}{#3}}
\def\AArrow(#1,#2){\ncline[nodesep=1mm, linewidth=0.8pt, border=2pt]{->}{#1}{#2}}
\def\BArrow(#1,#2){\ncline[linecolor=blue, linewidth=1.2pt, linestyle=dotted, nodesep=0.3mm, border=1.2pt]{->}{#1}{#2}}
\def\CArrow(#1,#2){\ncline[linecolor=red, nodesep=0.3mm, linewidth=0.8pt, border=1.2pt]{->}{#1}{#2}}
\def\DArrow(#1,#2){\ncline[linecolor=red,style=thickexist,nodesep=0.5mm,border=2pt]{->}{#1}{#2}}
\def\PArrow(#1,#2){\ncline[linestyle=dashed, nodesep=0.3mm, linewidth=0.8pt, border=0.8pt]{->}{#1}{#2}}

\newcommand\Vertex[2]{\psframebox[fillstyle=solid,fillcolor=color4]{\scriptsize $\aav_{#1}$: $\fr{#2}$}}
\newcommand\vertex[2]{\psframebox[linestyle=dashed]{\scriptsize $\aav_{#1}$: $\fr{#2}$}}
\newcommand\Vertexx[2]{\psframebox[linestyle=dashed]{\psframebox[fillstyle=solid,fillcolor=color4]{\scriptsize $\aav_{#1}$: $\fr{#2}$}}}
\newcommand\Vertexxx[2]{\psframebox[linestyle=dashed]{\psframebox[fillstyle=solid,fillcolor=color2]{\scriptsize $\aav_{#1}$: $\fr{#2}$}}}
\numberwithin{equation}{section}

\newcounter{ITEM}
\newcommand\ITEM[1]{\setcounter{ITEM}{#1}\leavevmode\hbox{\rm(\roman{ITEM})}}

\newcommand\1{\mathrm{1}}
\renewcommand\aa{a}
\renewcommand\AA{A}
\newcommand\aav{\U\aa}
\newcommand\act{\mathbin{\scriptscriptstyle\bullet}}

\newcommand\Aft[2]{\widetilde{\VR(2.1,0)\smash{\mathrm{#1}}}_{#2}}
\renewcommand\and{\text{and}}

\newcommand\antird{\Leftarrow}

\newcommand\antirds{\mathrel{{}^*\hspace{-1ex}\antird\hspace{0.5ex}}}

\newcommand\Att{\widetilde{\VR(2.1,0)\smash{\mathrm{A}}}_2}
\newcommand\bb{b}

\newcommand\bbv{\U\bb}

\newcommand\can{\iota}
\newcommand\canp{\iota^{\scriptscriptstyle+}}

\newcommand\cc{c}
\newcommand\CC{C}
\newcommand\ccv{\U\cc}

\newcommand\clp[1]{[#1]^{\HS{-0.2}\scriptscriptstyle+}}

\newcommand\ConjA{\mathbf{A}}
\newcommand\ConjB{\mathbf{B}}
\newcommand\ConjC{\mathbf{C}}
\newcommand\ConjCunif{\mathbf{C}_{\mathrm{unif}}}

\newcommand\dd{d}

\newcommand\DDD{\mathcal{D}}

\newcommand\ddv{\U\dd}

\newcommand\der{\partial}
\newcommand\derdiv{\partial}
\renewcommand\dh[1]{\Vert#1\Vert}
\renewcommand\div{<}
\newcommand\Div{\mathrm{Div}}
\newcommand\dive{\le}

\newcommand\divet{\mathrel{\widetilde{\VR(1.8,0)\smash{\dive}}}}
\newcommand\Divmax[1]{D_{#1}^\smax}
\newcommand\divt{\mathrel{\widetilde{\VR(1.8,0)\smash{<}}}}
\newcommand\ee{e}

\newcommand\eg{\textit{e.g.}}
\newcommand\EG[1]{\mathcal{U}(#1)}
\newcommand\ef{\varnothing}

\newcommand\equ{\bowtie}

\newcommand\ew{\varepsilon}

\newcommand\ff{f}
\newcommand\FF{F}
\newcommand\FRb[1]{\mathcal{F}^{\scriptscriptstyle\pm}_{\HS{-0.6}#1}}
\newcommand\FRp[1]{\mathcal{F}_{\HS{-0.4}#1}}
\newcommand\fr{\mathtt}
\renewcommand\gcd{\wedge}
\newcommand\gcdt{\mathbin{\widetilde\wedge}}
\renewcommand\ge{\geqslant}
\renewcommand\gg{g}

\newcommand\GR[2]{\langle#1\mid#2\rangle}
\newcommand\hh{h}

\newcommand\HS[1]{\hspace{#1ex}}

\newcommand\ie{\textit{i.e.}}
\newcommand\ii{i}
\newcommand\II{I}
\newcommand\inv{^{-1}}
\newcommand\INV[1]{\overline{#1}}
\newcommand\Irr{\mathrm{Irr}}
\newcommand\jj{j}

\newcommand\kk{k}
\newcommand\lcm{\vee}
\newcommand\lcmt{\mathbin{\widetilde\lcm}}
\renewcommand\le{\leqslant}

\newcommand\mm{m}
\newcommand\MM{M}

\newcommand\MON[2]{\langle#1\mid#2\rangle^{\HS{-0.5}\scriptscriptstyle+}}

\newcommand\nn{n}

\newcommand\NNNN{\mathbb{N}}

\newcommand\One[1]{\underline1_{#1}}
\newcommand\one{\underline1}

\newcommand\opp{\cdot}
\newcommand\pdots{\HS{0.2}{\cdot}{\cdot}{\cdot}\HS{0.2}}

\newcommand\pp{p}

\newcommand\PropH{\mathrm{H}}

\newcommand\qq{q}
\newcommand\quand{\quad\text{and}\quad}
\newcommand\qquand{\qquad\text{and}\qquad}
\newcommand\rd{\Rightarrow}

\newcommand\RDatp[1]{\RRR_{\HS{-0.2}\mathrm{at}}}

\newcommand\RDb[1]{\RRR_{\HS{-0.2}#1}^{\scriptscriptstyle\pm}}
\newcommand\RDbh[1]{\widehat\RRR_{\HS{-0.2}#1}^{\scriptscriptstyle\pm}}
\newcommand\rddiv{\mathbin{\rd_{\HS{-0.6}\scriptscriptstyle\mathrm{d\HS{-0.1}i\HS{-0.1}v}}}}
\newcommand\rddivs{\mathbin{\rd_{\HS{-0.6}\scriptscriptstyle\mathrm{d\HS{-0.1}i\HS{-0.1}v}}^*}}
\newcommand\rdh{\mathrel{\rd\hspace{-2ex}\widehat{\VR(1.7,0)}\HS{1}}}
\newcommand\rdhs{\mathrel{\rdh{\hspace{-0.5ex}}^*}}
\newcommand\Rdiv[1]{D_{#1}}
\newcommand\RDiv[1]{\DDD_{\HS{-0.2}#1}}
\newcommand\RDivb[1]{\DDD_{\HS{-0.2}#1}^{\scriptscriptstyle\pm}}
\newcommand\RDp[1]{\RRR_{\HS{-0.3}#1}}
\newcommand\rds{\rd^{\HS{-0.3}*}}
\newcommand\rdt{\mathrel{\HS{1}\widetilde\ \HS{-2}\rd}}
\newcommand\RDtb[1]{\RRRt^{\HS{-0.2}\scriptscriptstyle\pm}_{\HS{-0.3}#1}}
\newcommand\RDtp[1]{\RRRt_{\HS{-0.3}#1}}
\newcommand\rdts{\rdt^{\HS{-0.3}*}}
\renewcommand\red{\mathrm{red}}
\newcommand\Red[1]{R_{#1}}
\newcommand\Redmax[1]{R_{#1}^\smax}
\newcommand\redt{\widetilde{\VR(1.8,0)\smash\red}}
\newcommand\Redt[1]{\widetilde{R}_{#1}}
\newcommand\redtame{\mathrm{red}_{\mathrm{t}}}
\newcommand\resp{\mbox{\it resp}.,\ }

\newcommand\REV[1]{\widetilde{\VR(1.5,0){#1}}}
\newcommand\rr{r}
\newcommand\RR{R}
\newcommand\RRR{\mathcal{R}}
\newcommand\RRRh{\widehat{\VR(2.2,0)\smash\RRR}}

\newcommand\RRRt{\widetilde{\VR(2.2,0)\smash\RRR}}

\newcommand\sdots{ / \pdots / }
\newcommand\simeqb{\simeq^{\HS{-0.2}\scriptscriptstyle\pm}}
\newcommand\simeqp{\simeq}

\newcommand\smax{{\hbox{\rm\tiny max}}}
\renewcommand\ss{s}
\renewcommand\SS{S}
\newcommand\SSb{\overline{S}}

\renewcommand\tt{t}
\newcommand\TT{T}
\newcommand\tta{\mathtt{a}}
\newcommand\ttb{\mathtt{b}}
\newcommand\ttc{\mathtt{c}}

\let\U=\underline
\newcommand\UG[1]{\Gamma_{\!#1}}
\newcommand\univ[1]{U(#1)}
\newcommand\uu{u}

\def\VR(#1,#2){\vrule width0pt height#1mm depth#2mm}
\newcommand\vs{{\it vs.}\xspace}
\newcommand\vv{v}

\newcommand\wdots{, ...\HS{0.2},}
\newcommand\wit{\lambda}
\newcommand\ww{w}

\newcommand\xx{x}
\newcommand\XX{X}
\newcommand\xxh{\widehat\xx}
\newcommand\xxv{\underline\xx}
\newcommand\yy{y}

\newcommand\yyh{\widehat\yy}

\newcommand\zz{z}

\newcommand\ZZZZ{\mathbb{Z}}

\title[Multifraction reduction II]{Multifraction reduction II: conjectures for Artin-Tits groups}

\author{Patrick Dehornoy}

\address{Laboratoire de Math\'ematiques Nicolas Oresme, CNRS UMR 6139, Universit\'e de Caen, 14032 Caen cedex, France, and Institut Universitaire de France}
\email{patrick.dehornoy@unicaen.fr}
\urladdr{www.math.unicaen.fr/$\sim$dehornoy}



\begin{document}

\maketitle

\begin{abstract}
Multifraction reduction is a new approach to the word problem for Artin-Tits groups and, more generally, for the enveloping group of a monoid in which any two elements admit a greatest common divisor. This approach is based on a rewrite system (``reduction'') that extends free group reduction. In this paper, we show that assuming that reduction satisfies a weak form of convergence called semi-convergence is sufficient for solving the word problem for the enveloping group, and we connect semi-convergence with other conditions involving reduction. We conjecture that these properties are valid for all Artin-Tits monoids, and provide partial results and numerical evidence supporting such conjectures. 
\end{abstract}

\section{Introduction}

Artin-Tits groups are those groups that admit a positive presentation with at most one relation $\ss ... = \tt ...$ for each pair of generators~$\ss, \tt$ and, in this case, the relation has the form $\ss \tt \ss \tt ... = \tt \ss \tt \ss ...$, both sides of the same length~\cite{BriBou, GoP2}. It is still unknown whether the word problem is decidable for all Artin--Tits groups as, at the moment, decidability was established for particular families only: braid groups (E.\,Artin \cite{Art} in~1947), spherical type (P.\,Deligne~\cite{Dlg} and E.\,Brieskorn--K.\,Saito~\cite{BrS} in~1972), large type (K.I.\,Appel--P.E.\,Schupp \cite{ApS} in~1983), triangle-free (S.\,Pride \cite{Pri} in~1986), FC type (J.\,Altobelli~\cite{Alt} and A.\,Chermak~\cite{Che} in~1998). Later on, some of these groups were proved to be biautomatic or automatic~\cite{Eps, Chc}. 

This paper, which follows~\cite{Dit} but is self-contained, continues the investigation of multifraction reduction, a new approach to the word problem for Artin-Tits groups and, more generally, for the enveloping group~$\EG\MM$ of a cancellative monoid~$\MM$ in which any two elements admit a left and a right greatest common divisor (``gcd-monoid''). This approach is based on a certain algebraic rewrite system, called \emph{reduction}, which unifies and (properly) extends all previous rewrite systems based on exploiting the Garside structure of Artin--Tits monoids~\cite{Dez, Tat, HeM}. It is proved in~\cite{Dit} that, if the monoid~$\MM$ satisfies various properties involving the divisibility relations, all true in every Artin--Tits monoid, together with an additional assumption called the $3$-Ore condition, then reduction is convergent and every element of the enveloping group of~$\MM$ admits a unique representation by an irreducible multifraction, directly extending the classical result by \O.\,Ore about representation by irreducible fractions. 

In the current paper, we address the case of a general gcd-monoid, when the $3$-Ore condition is not assumed. In this case, reduction is not convergent, and there is no unique representation of the elements of the group~$\EG\MM$ by irreducible multifractions. However, we introduce a new, weaker condition called \emph{semi-convergence}, and prove that most of the applications of the convergence of reduction can be derived from its semi-convergence, in particular a solution of the word problem for~$\EG\MM$ whenever convenient finiteness conditions are satisfied. This makes the following conjecture crucial:

\begin{conjA}
 Reduction is semi-convergent for every Artin-Tits monoid.
\end{conjA}

A proof of Conjecture~$\ConjA$ would imply the decidability of the word problem for every Artin-Tits group. The reasons for believing in Conjecture~$\ConjA$ are multiple. One abstract reason is that reduction is really specific and uses both the whole Garside structure of Artin--Tits monoids and, for the finiteness of the set of basic elements, some highly nontrivial properties of the associated Coxeter groups~\cite{Din, DyH}: this may be seen as more promising than a generic approach based on, say, a ``blind'' Knuth--Bendix completion. Next, we state several related conjectures (``$\ConjB$'', ``$\ConjC$'', ``$\ConjCunif$''), of which some partial cases are proven and which suggest the existence of a rich rigid structure. Another reason is the existence of massive computer tests supporting all the conjectures and, at the same time, efficiently discarding wrong variations and dead-ends. Finally, the existence of a proof in the special case of FC~type is a positive point. In the same direction, a weak version of Conjecture~$\ConjA$ (sufficient for solving the word problem) was recently established for all Artin--Tits groups of sufficiently large type~\cite{Dix}: although saying nothing about~$\ConjA$ itself, this shows that reduction is relevant for a new family of Artin--Tits groups.

We present below a state-of-the-art description of the known results involving multifraction reduction, and report about computer experiments supporting Conjecture~$\ConjA$ and its variants. The paper is organized as follows. We gather in Section~\ref{S:Red} the needed prerequisites about multifractions, gcd-monoids, and reduction. Semi-convergence and Conjecture~$\ConjA$ are introduced in Section~\ref{S:Semi}, and their consequences are established. In Section~\ref{S:Tame}, we analyze specific cases of reduction, namely divisions and their extensions, tame reductions. This leads to a new property, stated as Conjecture~$\ConjB$, which is stronger than Conjecture~$\ConjA$ but easier to experimentally check and maybe to establish. Then, we introduce in Section~\ref{S:CrossConf} cross-confluence, a new property of reduction that involves both reduction and a symmetric counterpart of it. This leads to Conjecture~$\ConjC$ and its uniform version~$\ConjCunif$, again stronger than Conjecture~$\ConjA$ but possibly more accessible. In Section~\ref{S:SmallDepth}, we analyze the case of small depth multifractions. We prove in particular that semi-convergence for multifractions of depth~$2$ is equivalent to $\MM$ embedding into its enveloping group, and semi-convergence for multifractions of depth~$4$ is equivalent to a unique decomposition property for fractions in~$\EG\MM$. Finally, we gather in Section~\ref{S:Misc} reports about computer experiments and a few comments about further possible developments. 

\subsection*{Acknowledgments}

The author thanks Friedrich Wehrung for many discussions about the content of this paper. In particular, the notion of a lcm-expansion mentioned in Sec.~\ref{SS:Options} appeared during our joint work of interval monoids~\cite{Div}. The author also thanks both the editor and the referee, whose many suggestions certainly improved the exposition significantly.

\section{Multifraction reduction}\label{S:Red}

In this introductory section, we recall the notions of a multifraction and of a gcd-monoid, as well as the definition of multifraction reduction~\cite{Dit}.

\subsection{Multifractions}\label{SS:Multifrac}

If $\MM$ is a monoid, we denote by~$\EG\MM$ and~$\can$ the enveloping group of~$\MM$ and the canonical morphism from~$\MM$ to~$\EG\MM$, characterized by the universal property that every morphism from~$\MM$ to a group factors through~$\can$. By definition, every element~$\gg$ of~$\EG\MM$ can be expressed as
\begin{equation}\label{E:Fraction}
\can(\aa_1) \can(\aa_2)\inv \can(\aa_3) \can(\aa_4)\inv \pdots \text{\quad or \quad} \can(\aa_1)\inv \can(\aa_2) \can(\aa_3)\inv \can(\aa_4) \pdots,
\end{equation}
with $\aa_1 \wdots \aa_\nn$ in~$\MM$. We shall investigate~$\EG\MM$ using such expressions. In~\cite{Dit}, without loss of generality, we only considered, expressions~\eqref{E:Fraction} where the first term~$\can(\aa_1)$ is positive (possibly trivial, that is, equal to~$1$). Here, in particular in view of Section~\ref{S:CrossConf}, we skip that condition, and allow for both signs in the first entry. 

\begin{defi}\label{D:Multifrac}
Let $\MM$ be a monoid. Let $\INV\MM$ be a disjoint copy of~$\MM$; call the elements of~$\MM$ (\resp $\INV\MM$) positive (\resp negative). For $\nn \ge 1$, a \emph{positive} (\resp \emph{negative}) \emph{$\nn$-multifraction} on~$\MM$ is a finite sequence $(\aa_1 \wdots \aa_\nn)$ with entries in~$\MM \cup \INV\MM$, alternating signs, and $\aa_1$ in~$\MM$ (\resp $\INV\MM$). The set of all multifractions (\resp all positive multifractions) completed with the empty sequence~$\ef$ is denoted by~$\FRb\MM$ (\resp $\FRp\MM$). A multiplication is defined by
\begin{equation*}\label{E:SignedProd}
(\aa_1 \wdots \aa_\nn) \opp (\bb_1 \wdots \bb_\pp) = 
\begin{cases}
(\aa_1 \wdots \aa_{\nn-1}, \aa_\nn \bb_1, \bb_2 \wdots \bb_\pp)
&\quad\text{for $\aa_\nn$ and $\bb_1$ in~$\MM$,}\\
(\aa_1 \wdots \aa_{\nn-1}, \bb_1\aa_\nn, \bb_2 \wdots \bb_\pp)
&\quad\text{for $\aa_\nn$ and $\bb_1$ in~$\INV\MM$,}\\
(\aa_1 \wdots \aa_\nn , \bb_1 \wdots \bb_\pp)
&\quad\text{otherwise,}\\
\end{cases}
\end{equation*}
extended with $\aav \opp \ef = \ef \opp \aav = \aav$ for every~$\aav$. 
\end{defi}

We use $\aav, \bbv, ...$ as generic symbols for multifractions, and~$\aa_\ii$ for the $\ii$th entry of~$\aav$ counted from~$1$. For~$\aav$ in~$\FRb\MM$, the length of~$\aav$ (number of entries) is called its \emph{depth}, written~$\dh\aav$. We identify an element~$\aa$ of~$\MM$ with the depth one positive multifraction~$(\aa)$. Multfractions will play the role of iterated fractions, and the following convention is then convenient:

\begin{nota}\label{N:Frac}
For $\aa_1 \wdots \aa_\nn$ in~$\MM$, we put
\begin{equation}\label{E:Frac}
\aa_1 \sdots \aa_\nn:= (\aa_1, \INV{\aa_2}, \aa_3, \INV{\aa_4}, ...) \quand 
/ \aa_1 \sdots \aa_\nn:= (\INV{\aa_1}, \aa_2, \INV{\aa_3}, \aa_4, ...).
\end{equation}
We say that $\ii$ is \emph{positive} (\resp \emph{negative}) \emph{in}~$\aav$ if $\aa_\ii$ (\resp $\INV{\aa_\ii}$) occurs in the expansion of~$\aav$.
\end{nota}

With this convention, we recover the notation of~\cite{Dit}, where only positive multifractions are considered and $\INV\MM$ remains hidden. Multifractions are adequately illustrated by associating with every element~$\aa$ of~$\MM$ an arrow labeled~$\aa$, concatenating arrows to represent the product in~$\MM$, and associating with every multifraction the path made of (the arrows of) the successive entries with alternating orientations. The rules for the multiplication of~$\aav$ and~$\bbv$ can be read in the following diagrams:

- $\nn$ positive in~$\aav$, $1$ positive in~$\bbv$: 
\begin{picture}(68,4)(-2,0)
\psset{nodesep=0.5mm}
\psset{xunit=0.85mm}
\pcline[style=etc](0,0)(4,0)
\pcline{<-}(4,0)(12,0)
\pcline{->}(12,0)(20,0)\taput{$\aa_\nn$}
\pcline{->}(24,0)(32,0)\taput{$\bb_1$}
\pcline{<-}(32,0)(40,0)
\psline[style=etc](40,0)(44,0)
\pcline[style=etc](52,0)(60,0)
\pcline{<-}(60,0)(68,0)
\pcline{->}(68,0)(84,0)\taput{$\aa_\nn\bb_1$}
\pcline{<-}(84,0)(92,0)
\psline[style=etc](92,0)(96,0)
\put(18,-1){$\opp$}
\put(40,-1){$=$}
\end{picture}

- $\nn$ positive in~$\aav$, $1$ negative in~$\bbv$: 
\begin{picture}(68,4)(-1,0)
\psset{nodesep=0.5mm}
\psset{xunit=0.85mm}
\pcline[style=etc](0,0)(4,0)
\pcline{<-}(4,0)(12,0)
\pcline{->}(12,0)(20,0)\taput{$\aa_\nn$}
\pcline{<-}(24,0)(32,0)\taput{$\bb_1$}
\pcline{->}(32,0)(40,0)
\psline[style=etc](40,0)(44,0)
\pcline[style=etc](52,0)(60,0)
\pcline{<-}(60,0)(68,0)
\pcline{->}(68,0)(76,0)\taput{$\aa_\nn$}
\pcline{<-}(76,0)(84,0)\taput{$\bb_1$}
\pcline{->}(84,0)(92,0)
\psline[style=etc](92,0)(96,0)
\put(18,-1){$\opp$}
\put(40,-1){$=$}
\end{picture}

- $\nn$ negative in~$\aav$, $1$ positive in~$\bbv$: 
\begin{picture}(68,4)(-1,0)
\psset{nodesep=0.5mm}
\psset{xunit=0.85mm}
\pcline[style=etc](0,0)(4,0)
\pcline{->}(4,0)(12,0)
\pcline{<-}(12,0)(20,0)\taput{$\aa_\nn$}
\pcline{->}(24,0)(32,0)\taput{$\bb_1$}
\pcline{<-}(32,0)(40,0)
\psline[style=etc](40,0)(44,0)
\pcline[style=etc](52,0)(60,0)
\pcline{->}(60,0)(68,0)
\pcline{<-}(68,0)(76,0)\taput{$\aa_\nn$}
\pcline{->}(76,0)(84,0)\taput{$\bb_1$}
\pcline{<-}(84,0)(92,0)
\psline[style=etc](92,0)(96,0)
\put(18,-1){$\opp$}
\put(40,-1){$=$}
\end{picture}

- \VR(0,4)$\nn$ negative in~$\aav$, $1$ negative in~$\bbv$: 
\begin{picture}(68,4)(0,0)
\psset{nodesep=0.5mm}
\psset{xunit=0.85mm}
\pcline[style=etc](0,0)(4,0)
\pcline{->}(4,0)(12,0)
\pcline{<-}(12,0)(20,0)\taput{$\aa_\nn$}
\pcline{<-}(24,0)(32,0)\taput{$\bb_1$}
\pcline{->}(32,0)(40,0)
\psline[style=etc](40,0)(44,0)
\pcline[style=etc](52,0)(60,0)
\pcline{->}(60,0)(68,0)
\pcline{<-}(68,0)(84,0)\taput{$\bb_1\aa_\nn$}
\pcline{->}(84,0)(92,0)
\psline[style=etc](92,0)(96,0)
\put(18,-1){$\opp$}
\put(40,-1){$=$}
\end{picture}

\begin{prop}\label{P:EnvGroup}
\ITEM1 The set~$\FRb\MM$ equipped with~$\opp$ and~$\ef$ is a monoid generated by $\MM \cup \INV\MM$, and $\FRp\MM$ is the submonoid of~$\FRb\MM$ generated by $\MM \cup \{1 / \aa \mid \aa \in \MM\}$. The family of depth~one multifractions is a submonoid isomorphic to~$\MM$.

\ITEM2 Let $\simeqb$ be the congruence on~$\FRb\MM$ generated by $(1, \ef)$ and the pairs $(\aa / \aa, \ef)$ and $(/ \aa / \aa, \ef)$ with~$\aa$ in~$\MM$. For~$\aav$ in~$\FRb\MM$, let~$\can(\aav)$ be the $\simeqb$-class of~$\aav$. Then the group~$\EG\MM$ is (isomorphic to) $\FRb\MM{/}{\simeqb}$ and, for all $\aa_1 \wdots \aa_\nn$ in~$\MM$, we have
\begin{equation}\label{E:Eval}
\can(\aa_1 \sdots \aa_\nn) = \can(\aa_1) \can(\aa_2)\inv \can(\aa_3) \pdots, \quad \can(/ \aa_1 \sdots \aa_\nn) = \can(\aa_1)\inv \can(\aa_2) \can(\aa_3)\inv \pdots. 
\end{equation}

\ITEM3 The restriction of~$\simeqb$ to~$\FRp\MM$ is the congruence~$\simeqp$ generated by $(1, \ef)$ and the pairs $(\aa / \aa, \ef)$ and $(1 / \aa / \aa, \ef)$ with~$\aa$ in~$\MM$. The group~$\EG\MM$ is also isomorphic to~$\FRp\MM{/}{\simeqp}$. The translation $\aav \mapsto 1 \opp \aav$ maps~$\FRb\MM$ onto~$\FRp\MM$ and, for all~$\aav, \bbv$ in~$\FRb\MM$, the relation~$\aav \simeqb \bbv$ is equivalent to $1 \opp \aav \simeqp 1 \opp \bbv$.
\end{prop}

\begin{proof}
The argument is similar to the proof of~\cite[Proposition~2.4]{Dit}, and we only point the differences due to using signed multifractions. For~\ITEM1, associativity is checked directly, and the generating subsets for~$\FRb\MM$ and~$\FRp\MM$ follow from the equalities
\begin{equation}\label{E:Decomp}
\aa_1 \sdots \aa_\nn = \aa_1 \opp \INV{\aa_2} \opp \aa_3 \opp \INV{\aa_4} \opp \pdots
= \aa_1 \opp 1/\aa_2 \opp \aa_3 \opp 1/\aa_4 \opp \pdots:
\end{equation}
both hold in~$\FRb\MM$, and the second only involves positive multifractions.

For~\ITEM2 and~\ITEM3, for every~$\aa$ in~$\FRb\MM$, the definition of~$\simeqb$ implies $\can(\INV\aa) = \can(\aa)\inv$ and, writing $\canp(\aav)$ for the $\simeqp$-class of~$\aav$, that of~$\simeqp$ implies $\canp(1 / \aa) = \canp(\aa)\inv$. Hence both $\FRb\MM{/}{\simeqb}$ and $\FRp\MM{/}{\simeqp}$ are groups generated by~$\MM$. One easily checks that the latter groups satisfy the universal property defining~$\EG\MM$, and are therefore isomorphic to~$\EG\MM$. Then \eqref{E:Eval} directly follows from~\eqref{E:Decomp}.

Next, for every~$\aav$ in~$\FRb\MM$, the product $1 \opp \aav$ belongs to~$\FRp\MM$. Then, for~$\aav, \bbv$ in~$\FRb\MM$, write $\aav \approx \bbv$ for $1 \opp \aav \simeqp 1 \opp \bbv$. By considering all sign combinations and using relations like $1 / \aa\bb \simeqp 1/\bb \opp 1/ \aa$, one checks that $\approx$ is a congruence on~$\FRb\MM$, and it contains the pairs $(1, \ef)$, $(\aa / \aa, \ef)$, and $( / \aa/ \aa, \ef)$ that generate~$\simeqb$, as one finds for instance $1 \opp / \aa / \aa = 1 / \aa / \aa \simeqp \ef = 1 \opp \ef$. Hence $\approx$ includes~$\simeqb$. 
In the other direction, $\simeqb$ contains pairs that generate~$\simeqp$ and, being compatible with multiplication in~$\FRb\MM$, it is in particular compatible with multiplication in~$\FRp\MM$. So $\simeqp$ is included in~$\simeqb$, and $\aav \simeqb \bbv$ implies $1 \opp \aav \simeqp 1 \opp \bbv$ since $1$ is invertible mod~$\simeqb$. Hence, $\simeqb$ is included in~$\approx$ and, finally, $\simeqb$ and~$\approx$ coincide, which means that, for all~$\aav, \bbv$ in~$\FRb\MM$, we have
\begin{equation}\label{E:EquPos}
\aav \simeqb \bbv \quad \Longleftrightarrow\quad 1 \opp \aav \simeqp 1 \opp \bbv.
\end{equation} 
For~$\aav$ positive, we have $1 \opp \aav = \aav$, so \eqref{E:EquPos} implies in particular that $\simeqp$ is the restriction of~$\simeqb$ to~$\FRp\MM$, and the rest follows easily.
\end{proof}

Hereafter, we identify $\EG\MM$ with $\FRb\MM{/}{\simeqb}$ and~$\FRp\MM{/}{\simeqp}$. This representation is redundant in that, for every~$\aa$ in~$\MM$, the inverse~$\can(\aa)\inv$ of~$\can(\aav)$ is represented both by the depth~$1$ negative multifraction~$/ \aa$ and the depth~$2$ positive multifraction~$1/\aa$.

We conclude this introduction with some terminology that will be used frequently:

\begin{defi}
A multifraction~$\aav$ is called \emph{unital} if $\aav \simeqb \one$ holds, \ie, if $\aav$ represents~$1$ in the group~$\EG\MM$. It is called \emph{trivial} if all entries are equal to~$1$ or~$\INV1$. For $\nn > 0$, we write $\One\nn$ for $1 \sdots 1$, $\nn$ terms; for $\nn < 0$, we write $\One{\nn}$ for $/ 1 / 1 \sdots 1$, $\vert\nn\vert$ terms; in practice, we shall often omit the index and write~$\one$ for a trivial multifraction. 
\end{defi}

\subsection{Gcd-monoids}\label{SS:Gcd}

The reduction process we investigate requires that the ground mon\-oid is a gcd-monoid. We recall the basic definitions, referring to~\cite{Dit} (and~\cite{Dir}) for more details.

Let~$\MM$ be a monoid. For~$\aa, \bb$ in~$\MM$, we say that $\aa$ \emph{left divides}~$\bb$ or, equivalently, that $\bb$ is a \emph{right multiple} of~$\aa$, written~$\aa \dive \bb$, if $\aa\xx = \bb$ holds for some~$\xx$ in~$\MM$. If $\MM$ is a cancellative monoid and $1$ is the only invertible element in~$\MM$, the left divisibility relation is a partial order on~$\MM$. In this case, when they exist, the greatest common $\dive$-lower bound of two elements~$\aa, \bb$ is called their \emph{left gcd}, denoted by~$\aa \gcd \bb$, and their least common $\dive$-upper bound is called their \emph{right lcm}, denoted by~$\aa \lcm \bb$. 

Symmetrically, we say that $\aa$ \emph{right divides}~$\bb$ or, equivalently, that $\bb$ is a \emph{left multiple} of~$\aa$, written~$\aa \divet \bb$, if $\xx\aa = \bb$ holds for some~$\xx$. Under the same hypotheses, $\divet$ is a partial order on~$\MM$, with the derived right gcd and left lcm written~$\gcdt$ and~$\lcmt$.

\begin{defi}\label{D:GcdMon}
We say that $\MM$ is a \emph{gcd-monoid} if $\MM$ is a cancellative monoid, $1$ is the only invertible element in~$\MM$, and any two elements of~$\MM$ admit a left gcd and a right gcd.
\end{defi}

Typical examples of gcd-monoids are Artin-Tits monoids. Many more examples are known. In particular, every Garside or preGarside monoid~\cite{Dgk, Dir, GoP3} is a gcd-monoid.

Standard arguments \cite[Lemma\,2.15]{Dit} show that a gcd-monoid admits \emph{conditional right and left lcms}: it need not be true that any two elements admit a right lcm, but any two elements that admit a common right multiple admit a right lcm, and similarly on the left. 

The gcd and lcm operations of a gcd-monoid are connected in several ways with the product. We refer to~\cite[section~II.2]{Dir} for the (easy) proof of the rule for an iterated lcm:

\begin{lemm}\label{L:IterLcm}
If $\MM$ is a gcd-monoid and $\aa, \bb, \cc$ belong to~$\MM$, then $\aa \lcm \bb\cc$ exists if and only if $\aa \lcm \bb$ and $\aa' \lcm \cc$ do, where~$\aa'$ is defined by $\aa \lcm \bb = \bb \aa'$, and then we have
\begin{equation}\label{E:IterLcm}
\aa \lcm \bb\cc = \aa \opp \bb'\cc' = \bb\cc \opp \aa''.
\end{equation}
with $\aa \lcm \bb = \bb \aa' = \aa \bb'$ and $\aa' \lcm \cc = \aa' \cc' = \cc \aa''$. 
\end{lemm}

This implies in particular that $\aa \dive \bb\cc$ holds if and only if $\aa \lcm \bb$ exists and $\aa' \dive \cc$ holds, with~$\aa'$ defined by~$\aa \lcm \bb = \bb\aa'$.

\begin{lemm}\label{L:IterGcd}
If $\MM$ is a gcd-monoid and $\aa, \bb, \aa', \bb', \cc$ belong to~$\MM$ and satisfy $\aa\bb' = \bb \aa'$, then $\aa \gcd \bb = \aa' \gcd \cc = 1$ implies $\aa \gcd \bb\cc = 1$.
\end{lemm}

\begin{proof}
Assume $\xx \dive \aa$ and $\xx \dive \bb\cc$. By Lemma~\ref{L:IterLcm}, $\bb \lcm \xx$ must exist and, writing $\bb \lcm \xx = \bb\xx'$, we must have $\xx' \dive \cc$. On the other hand, $\xx \dive \aa$ implies $\xx \dive \aa\bb' = \bb\aa'$. So $\bb\aa'$ is a common right multiple of~$\bb$ and~$\xx$, hence it is a right multiple of their right lcm~$\bb \xx'$. As $\MM$ is left cancellative, $\bb \xx' \dive \bb \aa'$ implies $\xx' \dive \aa'$. Hence $\xx'$ left divides both~$\aa'$ and~$\cc$, and $\aa' \gcd \cc = 1$ implies $\xx' = 1$, whence $\xx \dive \bb$. Then $\xx$ left divides both~$\aa$ and~$\bb$, and $\aa \gcd \bb = 1$ implies $\xx = 1$, hence $\aa \gcd \bb\cc = 1$.
\end{proof}

We shall also need the notion of a noetherian monoid. If $\MM$ is a gcd-monoid, we use~$\div$ for the proper version of left divisibility: $\aa \div \bb$ holds if we have $\bb = \aa\xx$ for some non-invertible~$\xx$, \ie, for $\xx \not= 1$, and similarly for~$\divt$ vs.~$\divet$. 

\begin{defi}
A monoid~$\MM$ is called \emph{noetherian} if the relations~$\div$ and~$\divt$ are well-founded, \ie, every nonempty subset of~$\MM$ has a $\div$-minimal element and a $\divt$-minimal element.
\end{defi} 

Note that a monoid~$\MON\SS\RR$ is noetherian whenever each relation in~$\RR$ is \emph{homogeneous}, \ie, it has the form $\uu = \vv$ with~$\uu, \vv$ of the same length: indeed, $\aa \div\bb$ implies that any word in~$\SS$ representing~$\aa$ is shorter than any word representing~$\bb$, and an infinite $\div$-descending sequence cannot exist. Artin-Tits monoids are typical examples.

\subsection{Reduction of multifractions}\label{SS:Reduction}

Introduced in~\cite{Dit}, our tool for investigating the congruence~$\simeqb$ on~$\FRb\MM$ is \emph{reduction}, a family of partial depth-preserving transformations that, when defined, map a multifraction to a $\simeqb$-equivalent multifraction. These transformations are written as an action on the right: when defined, $\aav \act \RR$ is the result of applying~$\RR$ to~$\aav$. 

\begin{defi}\label{D:Red}
If $\MM$ is a gcd-monoid and $\aav, \bbv$ lie in~$\FRb\MM$, then, for $\ii \ge 1$ and $\xx$ in~$\MM$, we declare that $\bbv = \aav \act \Red{\ii, \xx}$ holds if we have $\dh\bbv = \dh\aav$, $\bb_\kk = \aa_\kk$ for $\kk \not= \ii - 1, \ii, \ii + 1$, and there exists~$\xx'$ (necessarily unique) satisfying
$$\begin{array}{lccc}
\text{for $\ii = 1$ positive in~$\aav$:\quad}
&&\bb_\ii \xx = \aa_\ii, 
&\bb_{\ii+1} \xx = \aa_{\ii+1},\\
\text{for $\ii = 1$ negative in~$\aav$:\quad}
&&\xx \bb_\ii = \aa_\ii, 
&\xx \bb_{\ii+1} = \aa_{\ii+1},\\
\text{for $\ii \ge 2$ positive in~$\aav$:\quad}
&\bb_{\ii-1} = \xx' \aa_{\ii-1}, 
&\bb_\ii \xx = \xx' \aa_\ii = \xx \lcmt \aa_\ii, 
&\bb_{\ii+1} \xx = \aa_{\ii+1},\\
\text{for $\ii \ge 2$ negative in~$\aav$:\quad}
&\smash{\bb_{\ii-1} = \aa_{\ii-1} \xx'}, 
&\smash{\xx \bb_\ii = \aa_\ii \xx' = \xx \lcm \aa_\ii}, 
&\xx \bb_{\ii+1} = \aa_{\ii+1}.
\end{array}$$
We write $\aav \rd \bbv$ if $\aav \act \Red{\ii, \xx}$ holds for some~$\ii$ and some $\xx \not= 1$, and use $\rds$ for the reflexive--transitive closure of~$\rd$. The rewrite system~$\RDb\MM$ so obtained on~$\FRb\MM$ is called \emph{reduction}, and its restriction to~$\FRp\MM$ (positive multifractions) is denoted by~$\RDp\MM$. A multifraction~$\aav$ is called \emph{$\RRR$-reducible} if $\aav \rd \bbv$ holds for at least one~$\bbv$, and \emph{$\RRR$-irreducible} otherwise.
\end{defi}

 The system~$\RDp\MM$ is the one investigated in~\cite{Dit}, where only positive multifractions are considered: the only difference between~\cite[Def.~3.4]{Dit} and~Def.~\ref{D:Red} is the adjunction in~$\RDb\MM$ of a rule for the reduction at level~$1$ of a negative multifraction. As $\Red{\ii, \xx}$ preserves the sign of multifractions, no specific notation is needed for the restriction of~$\Red{\ii, \xx}$ to positive multifractions.

The reduction systems~$\RDp\MM$ and~$\RDb\MM$ extend free reduction (deletion of factors~$\xx\inv \xx$ or~$\xx \xx\inv$): applying~$\Red{\ii, \xx}$ to~$\aav$ consists in removing~$\xx$ from~$\aa_{\ii + 1}$ and pushing it through~$\aa_\ii$ using an lcm operation, see Figure~\ref{F:Red}. A multifraction~$\aav$ is eligible for~$\Red{1, \xx}$ if and only if $\xx$ divides both~$\aa_1$ and~$\aa_2$, on the side coherent with their signs, and, eligible for~$\Red{\ii, \xx}$ with $\ii \ge 2$ if and only if $\xx$ divides~$\aa_{\ii + 1}$ and admits a common multiple with~$\aa_\ii$, on the due side again.

\begin{figure}[htb]
\begin{picture}(105,17)(0,-2)
\psset{nodesep=0.7mm}
\put(-7,6){$...$}
\psline[style=back,linecolor=color2]{c-c}(0,6)(15,0)(30,0)(45,6)
\psline[style=back,linecolor=color1]{c-c}(0,6)(15,12)(30,12)(45,6)
\pcline{->}(0,6)(15,12)\taput{$\aa_{\ii - 1}$}
\pcline{<-}(15,12)(30,12)\taput{$\aa_\ii$}
\pcline{->}(30,12)(45,6)\taput{$\aa_{\ii + 1}$}
\pcline{->}(0,6)(15,0)\tbput{$\bb_{\ii - 1}$}
\pcline{<-}(15,0)(30,0)\tbput{$\bb_\ii$}
\pcline{->}(30,0)(45,6)\tbput{$\bb_{\ii + 1}$}
\pcline[linewidth=1.5pt,linecolor=color3,arrowsize=1.5mm]{->}(30,12)(30,0)\trput{$\xx$}
\pcline{->}(15,12)(15,0)\tlput{$\xx'$}
\psarc[style=thin](15,0){3}{0}{90}
\put(21,5){$\Leftarrow$}
\put(51,6){$...$}
\psline[style=back,linecolor=color2]{c-c}(60,6)(75,0)(90,0)(105,6)
\psline[style=back,linecolor=color1]{c-c}(60,6)(75,12)(90,12)(105,6)
\pcline{<-}(60,6)(75,12)\taput{$\aa_{\ii - 1}$}
\pcline{->}(75,12)(90,12)\taput{$\aa_\ii$}
\pcline{<-}(90,12)(105,6)\taput{$\aa_{\ii + 1}$}
\pcline{<-}(60,6)(75,0)\tbput{$\bb_{\ii - 1}$}
\pcline{->}(75,0)(90,0)\tbput{$\bb_\ii$}
\pcline{<-}(90,0)(105,6)\tbput{$\bb_{\ii + 1}$}
\pcline[linewidth=1.5pt,linecolor=color3,arrowsize=1.5mm]{<-}(90,12)(90,0)\trput{$\xx$}
\pcline{<-}(75,12)(75,0)\tlput{$\xx'$}
\psarc[style=thin](75,0){3}{0}{90}
\put(81,5){$\Leftarrow$}
\put(109,6){$...$}
\end{picture}
\caption{\small The reduction rule~$\Red{\ii, \xx}$: starting from~$\aav$ (grey), we extract $\xx$ from~$\aa_{\ii + 1}$, push it through~$\aa_\ii$ by taking the lcm of~$\xx$ and~$\aa_\ii$ (indicated by the small round arc), and incorporate the remainder~$\xx'$ in~$\aa_{\ii - 1}$ to obtain~$\bbv = \aav \act \Red{\ii, \xx}$ (colored). The left hand side diagram corresponds to the case when $\ii$ is negative in~$\aav$, \ie, $\aa_\ii$ is crossed negatively, the right hand one to the case when $\ii$ is positive in~$\aav$, \ie, $\aa_\ii$ is crossed positively, with opposite orientations of the arrows.}
\label{F:Red}
\end{figure}
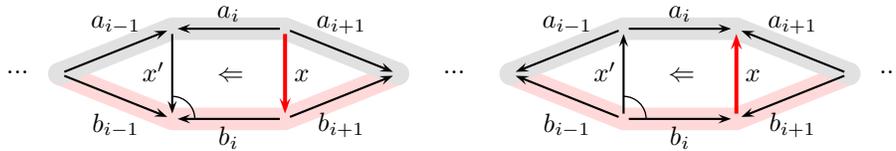

\begin{exam}\label{X:Att}
Let $\MM$ be the Artin-Tits monoid of type~$\Att$, here written 
$$\MON{\fr{a, b, c}}{\fr{aba=bab, \ bcb=cbc, \ cac=aca}},$$
and let $\aav := 1 /\ttc/\tta\ttb\tta$. Then $\aav$ is eligible for~$\Red{2,\tta}$ and~$\Red{2, \ttb}$, since $\tta$ and $\ttb$ left divide~$\tta\ttb\tta$ and admit a common right multiple with~$\ttc$. We find $\aav \act \Red{2,\tta} = \tta\ttc / \ttc\tta / \ttb\tta$ and $\aav \act \Red{2, \ttb} = \ttb\ttc / \ttc\ttb / \tta\ttb$. The latter are eligible for no reduction~$\Red{1, \xx}$, since $\tta\ttc$ and~$\ttc\tta$ (\resp $\ttb\ttc$ and $\ttc\ttb$) admit no nontrivial common right divisor, and for no reduction~$\Red{2, \xx}$, since the only nontrivial left divisors of~$\ttb\tta$ (\resp $\tta\ttb$) are $\ttb$ and~$\ttb\tta$ (\resp $\tta$ and~$\tta\ttb$), which admit no common right multiple with~$\ttc\tta$ (\resp $\tta\ttc$). Hence these multifractions are $\RRR$-irreducible.
\end{exam}


We now state the basic properties of reduction needed below. They directly extend those established for positive multifractions in~\cite{Dit}. Verifying them in the general case is easy.

\begin{lemm}\label{L:Basics}
Assume that $\MM$ be a gcd-monoid.

\ITEM1 The relation~$\rds$ is included in~$\simeqb$, \ie, $\aav \rds \bbv$ implies $\aav \simeqb \bbv$.

\ITEM2 The relation~$\rds$ is compatible with the multiplication of~$\FRb\MM$.

\ITEM3 For all~$\aav, \bbv$ and~$\pp, \qq$, the relation~$\aav \rds \bbv$ is equivalent to $\One\pp \opp \aav \opp \One\qq \rds \One\pp \opp \bbv \opp \One\qq$.
\end{lemm}

\begin{proof}
\ITEM1 It directly follows from the definition (and from Fig.~\ref{F:Red}) that $\bbv = \aav \act \Red{\ii, \xx}$ implies 
$$\can(\aa_{\ii - 1}) \can(\aa_\ii)\inv \can(\aa_{\ii + 1}) = \can(\bb_{\ii - 1}) \can(\bb_\ii)\inv \can(\bb_{\ii + 1}) $$
(\resp $\can(\aa_{\ii - 1})\inv \can(\aa_\ii) \can(\aa_{\ii + 1})\inv = \can(\bb_{\ii - 1})\inv \can(\bb_\ii) \can(\bb_{\ii + 1})\inv$) for $\ii$ negative (\resp positive) in~$\aav$.

\ITEM2 Assume $\bbv = \aav \act \Red{\ii, \xx}$, and let $\ccv$ be an $\rr$-multifraction. If the signs of~$\rr$ in~$\ccv$ and~$1$ in~$\aav$ are different, then $\ccv \opp \aav$ is the concatenation of~$\ccv$ and~$\aav$, similarly $\ccv \opp \bbv$ is the concatenation of~$\ccv$ and~$\bbv$, and $\ccv \opp \bbv = (\ccv \opp \aav) \act \Red{\ii + \rr, \xx}$ trivially holds. 

If $\rr$ is positive in~$\ccv$ and $1$ is positive in~$\aav$, we have $\ccv \opp \aav = \cc_1 \sdots \cc_{\rr- 1} / \cc_\rr\aa_1 / \aa_2 / \pdots$ and $\ccv \opp \bbv = \cc_1 \sdots \cc_{\rr- 1} / \cc_\rr\bb_1 / \bb_2 / \pdots$, and we obtain $\ccv \opp \bbv = (\ccv \opp \aav) \act \Red{\ii + \rr - 1, \xx}$: the point is that, if $\xx$ both right divides~$\aa_1$ and~$\aa_2$, it a fortiori right divides~$\cc_\rr \aa_1$ and~$\aa_2$. 

Finally, assume that $\rr$ negative in~$\ccv$ and $1$ is negative in~$\aav$. Then we find $\ccv \opp \aav = \cc_1 \sdots \aa_1\cc_\rr / \aa_2 / \pdots$ and $\ccv \opp \bbv = \cc_1 \sdots \bb_1\cc_\rr / \bb_2 / \pdots$: the argument is the same as above, mutatis mutandis: the assumption that $\aav \act \Red{1, \xx}$ is defined means that $\xx$ both left divides~$\aa_1$ and~$\aa_2$, which implies that it a fortiori left divides~$\aa_1 \cc_\rr$ and~$\aa_2$. Hence $(\ccv \opp \aav) \act \Red{\ii + \rr - 1, \xx}$ is defined and we find $\ccv \opp \bbv = (\ccv \opp \aav) \act \Red{\ii + \rr - 1, \xx}$ again. This completes compatibility with left multiplication.

The compatibility on the right is similar: adding extra entries cannot destroy the eligibility for reduction. Let $\nn = \dh\aav$. Everything is trivial for $\ii < \nn - 1$, so we assume $\ii = \nn - 1$. If the signs of~$\nn$ in~$\aav$ and of~$1$ in~$\ccv$ are different, the multiplication is a concatenation, and we obtain $\bbv \opp \ccv = (\aav \opp \ccv) \act \Red{\nn - 1, \xx}$ trivially. If $\nn$ is positive in~$\aav$ and $1$ is positive in~$\ccv$, the argument is the same as for~$\Red{\nn - 1, \xx}$. Finally, assume that $\nn$ is negative in~$\aav$ and $1$ is negative in~$\ccv$. Then we find $\aav \opp \ccv = \aa_1 \sdots \aa_{\nn - 1} / \cc_1\aa_\nn / \cc_2 / \pdots$ and $\bbv \opp \ccv = \bb_1 \sdots \bb_{\nn - 1} / \cc_1\bb_\nn / \cc_2 / \pdots$. The assumption that $\aav \act \Red{\nn - 1, \xx}$ is defined means that $\xx \lcmt \aa_{\nn - 1}$ exists and $\xx$ right divides~$\aa_\nn$, which implies that it a fortiori right divides~$\cc_1 \aa_\nn$. Hence $(\aav \opp \ccv) \act \Red{\nn - 1, \xx}$ is defined, yielding $\bbv \opp \ccv = (\aav \opp \ccv) \act \Red{\nn - 1, \xx}$ again. Thus reduction is compatible with multiplication on the right.

\ITEM3 That $\aav \rd \bbv$ implies $\One\pp \opp \aav \opp \One\qq \rd \One\pp \opp \bbv \opp \One\qq$ follows from~\ITEM2 directly. Conversely, assume $\One\pp \opp \bbv \opp \One\qq = (\One\pp \opp \aav \act \One\qq) \act \Red{\ii, \xx}$ with $\xx \not= 1$. Let $\nn = \dh\aav = \dh\bbv$, and assume that the entries of~$\aav$ occur in $\One\pp \opp \bbv \opp \One\qq$ from~$\rr + 1$ to~$\rr + \nn$ (with $\rr = \pp$ or $\rr = \pp - 1$ according to the sign of~$\pp$ in~$\One\pp$ and that of~$1$ in~$\aav$). Then we necessarily have $\rr + 1 \le \ii < \rr + \nn$. Indeed, $\ii < \rr$ and $\ii \ge \rr + \nn$ are impossible, since the $(\ii + 1)$st entry of $\One\pp \opp \aav \opp \One\qq$ is trivial. Moreover, in the case $\ii = \rr + 1$, the element~$\xx$ necessarily divides~$\aa_1$, since, otherwise, the $\rr$th entry of $\One\pp \opp \bbv \opp \One\qq$ could not be trivial. Hence, $\aav \act \Red{\ii, \xx}$ is defined, and it must be equal to~$\bbv$.
\end{proof}

Finally, we have a simple sufficient condition for termination.

\begin{lemm}\cite[Proposition~3.13]{Dit}\label{L:Termin}
If $\MM$ is a noetherian gcd-monoid, then $\RDb\MM$ is terminating: every sequence of reductions leads in finitely many steps to an $\RRR$-irreducible multifraction.
\end{lemm}

We skip the proof, which is exactly the same in the signed case as in the positive case, and consists in observing that $\aav \rd \bbv$ forces $\bbv$ to be strictly smaller than~$\aav$ for some antilexicographical ordering on~$\FRb\MM$ (comparing multifractions starting from the highest entry).

\subsection{The convergent case}\label{SS:Conv}

The rewrite system~$\RDb\MM$ (as any rewrite system) is called \emph{convergent} if every element, here every multifraction~$\aav$, admits a unique $\RRR$-irreducible reduct, usually denoted by~$\red(\aav)$. The main technical result of~\cite{Dit} is

\begin{prop}\label{P:Conv}
If $\MM$ is a noetherian gcd-monoid satisfying the $3$-Ore condition:
\begin{equation}\label{E:3Ore}
\parbox{120mm}{If three elements of~$\MM$ pairwise admit a common right $($resp. left$)$ multiple,\par\hfill then they admit a common right $($resp. left$)$ multiple,}
\end{equation}
then $\RDb\MM$ is convergent.
\end{prop}

When a monoid~$\MM$ is eligible for Proposition~\ref{P:Conv}, one easily deduces that two multifractions~$\aav, \bbv$ represent the same element of~$\EG\MM$ if and only if $\red(\aav) \opp \One\pp = \red(\bbv) \opp \One\qq$ holds for some~$\pp, \qq$ and, from there, that the monoid~$\MM$ embeds in its enveloping group~$\EG\MM$ and every element of~$\EG\MM$ is represented by a unique $\RRRh$-irreducible multifraction, where $\RDbh\MM$ is obtained from~$\RDb\MM$ by adding a rule that removes trivial final entries. It is also proved in~\cite{Dit} that, under mild additional finiteness assumptions on~$\MM$ (see Section~\ref{S:Semi} below), the relation~$\rds$ on~$\MM$ is decidable and, from there, so is the word problem for~$\EG\MM$ when the $3$-Ore condition is satisfied.

The above results are relevant for a number of gcd-monoids. We recall that an Artin-Tits monoid $\MM = \MON\SS\RR$ is said to be \emph{of spherical type} if the Coxeter group obtained by adding to~$(\SS, \RR)$ the relation $\ss^2 = 1$ for each~Ê$\ss$ in~$\SS$ is finite~\cite{BrS}. And $\MM$ is said to be \emph{of type~FC} if, for every subfamily~$\SS'$ of~$\SS$ such that, for all~$\ss, \tt$ in~Ê$\SS'$, there is a relation $\ss... = \tt... $ in~$\RR$, the submonoid of~$\MM$ generated by~$\SS'$ is spherical~\cite{Alt,GoP1}.

\begin{prop}\cite[Proposition~6.5]{Dit}\label{P:AT3Ore}
An Artin-Tits monoid satisfies the $3$-Ore condition if and only if it is of type~FC.
\end{prop} 

However, a number of Artin-Tits monoids fail to be of type~FC and therefore are not eligible for Proposition~\ref{P:Conv}, typically the monoid of type~$\Att$ considered in Example~\ref{X:Att}. So we are left with the question of either weakening the assumptions for Proposition~\ref{P:Conv}, or using a conclusion weaker than convergence.

\section{Semi-convergence}\label{S:Semi}

After showing in Subsection~\ref{SS:3Ore} that the $3$-Ore assumption cannot be weakened when proving the convergence of~$\RDb\MM$, we introduce in Subsection~\ref{SS:Semi} a new property of~$\RDb\MM$ called \emph{semi-convergence}, which, as the name suggests, is weaker than convergence. We conjecture that, for every Artin-Tits monoid, the system~$\RDb\MM$ is semi-convergent (``Conjecture~$\ConjA$''). We prove in Subsection~\ref{SS:SemiAppli} that most of the consequences known to follow from the convergence of~$\RDb\MM$ follow from its semi-convergence, in particular in terms of controlling the group~$\EG\MM$ from inside the monoid~$\MM$ and solving its word problem. Finally, we describe in Subsection~\ref{SS:ConvAlt} several variants of semi-convergence.

\subsection{The strength of the $3$-Ore condition}\label{SS:3Ore}

A first attempt for improving Proposition~\ref{P:Conv} could be to establish the convergence of~$\RDb\MM$ from an assumption weaker than the $3$-Ore condition. This approach fails, as the latter turns out to be not only sufficient, but also necessary. Hereafter we say that a monoid~$\MM$ satifies the right (\resp left) $3$-Ore condition when~\eqref{E:3Ore} is valid for right (\resp left) multiples. First, we recall

\begin{lemm}\label{L:Lcm}\cite[Lemma~2.12]{Dit}
If $\MM$ is a gcd-monoid, and $\aa, \bb, \cc, \dd$ are elements of~$\MM$ satisfying $\aa\dd = \bb\cc$, then $\aa\dd$ is the right lcm of~$\aa$ and~$\bb$ if and only if $\cc$ and $\dd$ satisfy $\cc \gcdt \dd = 1$.
\end{lemm}

\begin{prop}\label{P:UR}
Let $\MM$ be a gcd-monoid.

\ITEM1 If $\RDp\MM$ is convergent, then $\MM$ satisfies the right $3$-Ore condition.

\ITEM2 If $\RDb\MM$ is convergent, then $\MM$ satisfies the $3$-Ore condition.
\end{prop}

\begin{proof}
\ITEM1 Assume that $\xx, \yy, \zz$ belong to~$\MM$ and pairwise admit common right multiples, hence right lcms. Write
$$\xx \lcm \yy = \xx \yy' = \yy \xx', \quad \yy \lcm \zz = \yy \zz' = \zz \yy'', \quad \xx \lcm \zz = \xx \zz'' = \zz \xx''.$$
Let $\aav := 1/\xx/\yy \lcm \zz$. As in Example~\ref{X:Att}, we find 
$$\bbv = \aav \act \Red{2, \yy} = \yy' / \xx' / \zz' \quand \ccv = \aav \act \Red{2, \zz} = \zz'' / \xx'' / \yy''.$$
The assumption that $\RDp\MM$ is convergent implies that $\bbv$ and $\ccv$ both reduce to $\ddv:= \red(\aav)$. By construction, $\ddv$ is of depth~$3$, and $\dd_3$ must be a common right divisor of~$\bb_3$ and~$\cc_3$, which are $\zz'$ and~$\yy''$. Now, Lemma~\ref{L:Lcm} implies $\zz' \gcdt \yy'' = 1$, whence $\dd_3 = 1$. Therefore, there exist $\dd_1$ and~$\dd_2$ in~$\MM$ satisfying $1/\xx/\yy\lcm\zz \rds \dd_1 / \dd_2 / 1$. By Lemma~\ref{L:Basics}, we deduce $1/\xx/\yy\lcm\zz \simeqb \dd_1 / \dd_2 / 1$, hence, by~\eqref{E:Eval}, $\xx\inv (\yy \lcm \zz) = \dd_1 \dd_2\inv$ in~$\EG\MM$. This implies $\xx\dd_1 = (\yy \lcm \zz)\dd_2$ in~$\EG\MM$, hence in~$\MM$, since the assumption that $\RDp\MM$ is convergent implies that $\MM$ embeds in~$\EG\MM$. It follows that $\xx$ and $\yy \lcm \zz$ admit a common right multiple, hence that $\xx, \yy$, and~$\zz$ admit a common right multiple. Hence $\MM$ satisfies the right $3$-Ore condition.

\ITEM2 The argument is symmetric, with negative multifractions. Assume that $\xx, \yy, \zz$ belong to~$\MM$ and pairwise admit common left multiples, hence left lcms. Write
$$\xx \lcmt \yy = \yy'\xx = \xx'\yy, \quad \yy \lcmt \zz = \zz'\yy = \yy''\zz , \quad \xx \lcmt \zz = \zz''\xx = \xx''\zz .$$
Let $\aav := /1/\xx/\yy \lcm \zz$, in~$\FRb\MM \setminus \FRp\MM$. Then we have
$$\bbv = \aav \act \Red{2, \yy} = / \yy' / \xx' / \zz' \quand \ccv = \aav \act \Red{2, \zz} = / \zz'' / \xx'' / \yy''.$$
The assumption that $\RDb\MM$ is convergent implies that $\bbv$ and $\ccv$ admit a common $\RRR$-reduct, say~$\ddv$. By construction, $\ddv$ is of depth~$3$, and $\dd_3$ must be a common left divisor of~$\bb_3$ and~$\cc_3$, which are $\zz'$ and $\yy''$. By the symmetric counterpart of Lemma~\ref{L:Lcm}, we have $\zz' \lcm \yy'' = 1$, whence $\dd_3 = 1$. Therefore, there exist $\dd_1$ and~$\dd_2$ in~$\MM$ satisfying $\aav \rdts / \dd_1 / \dd_2 / 1$. By Lemma~\ref{L:Basics}, we deduce $/ 1 / \xx / \yy \lcmt \zz \simeqb / \dd_1 /\dd_2/ 1$, hence, by~\eqref{E:Eval}, $\xx (\yy \lcmt \zz)\inv = \dd_1\inv \dd_2$ in~$\EG\MM$, whence $\dd_1 \xx = \dd_2 (\yy \lcmt \zz)$ in~$\EG\MM$, hence in~$\MM$. This shows that $\xx$ and $\yy \lcmt \zz$ admit a common left multiple, hence that $\xx, \yy$, and~$\zz$ admit a common left multiple. Hence $\MM$ satisfies the left $3$-Ore condition.
\end{proof}

Note that the argument for~\ITEM1 cannot be used for~\ITEM2, because one should start with $\aav:= 1/1/\xx/\yy \lcmt \zz$ and then we know nothing about the first entry(ies) of~$\red(\aav)$. 

In principle, the right $3$-Ore condition is slightly weaker than the full $3$-Ore condition, and the convergence of~$\RDp\MM$ might be weaker than that of~$\RDb\MM$. However, when $\MM$ is an Artin-Tits monoid, all the above conditions are equivalent to $\MM$ being of type~FC and, therefore, none is weaker. So it seems hopeless to improve Proposition~\ref{P:Conv}.

\subsection{Semi-convergence}\label{SS:Semi}

We are thus led to explore the other way, namely obtaining useful information about~$\EG\MM$ from a property weaker than the convergence of~$\RDp\MM$ or~$\RDb\MM$. This is the approach we develop in the rest of this paper.

When the system~$\RDp\MM$ is not convergent, a $\simeq$-class may contain several $\RRR$-irreducible multifractions, and there is no distinguished one: in Example~\ref{X:Att}, the automorphism that exchanges~$\tta$ and~$\ttb$ exchanges the two $\RRR$-irreducible reducts of~$\aav$, making them indiscernible. 

However, a direct consequence of convergence is 

\begin{lemm}\label{L:Semi}
If $\MM$ is a gcd-monoid and $\RDb\MM$ is convergent, then, for every~$\aav$ in~$\FRb\MM$, 
\begin{equation}\label{E:Semi}
\text{If $\aav$ is unital, then $\aav \rds \one$ holds.}
\end{equation}
\end{lemm}

Indeed, if $\RDb\MM$ is convergent, $\aav \simeqb \one$ implies (and, in fact, is equivalent to) $\red(\aav) = \red(\one)$, hence $\red(\aav) = \one$, since $\one$ is $\RRR$-irreducible. Note that, by Lem\-ma~\ref{L:Basics}\ITEM1, the converse implication of~\eqref{E:Semi} is always true: $\aav \rds \one$ implies that $\aav$ and $\one$ represent the same element of~$\EG\MM$, hence that $\aav$ represents~$1$. 

When $\RDb\MM$ is not convergent, \eqref{E:Semi} still makes sense, and it is a priori a (much) weaker condition than convergence. This is the condition we shall investigate below:

\begin{defi}\label{D:Semi}
If $\MM$ is a gcd-monoid, we say that $\RDb\MM$ (\resp $\RDp\MM$) is \emph{semi-convergent} if \eqref{E:Semi} holds for every~$\aav$ in~$\FRb\MM$ (\resp in $\FRp\MM$).
\end{defi}

Thus Lemma~\ref{L:Semi} states that $\RDb\MM$ is semi-convergent whenever it is convergent. By Proposition~\ref{P:AT3Ore}, $\RDb\MM$ and $\RDp\MM$ are semi-convergent for every Artin-Tits monoid~$\MM$ of type~FC. But, as can be expected, semi-convergence is strictly weaker than convergence. We refer to~\cite{Div} for the construction of monoids for which $\RDb\MM$ is semi-convergent but not convergent.

The main conjecture we propose is:

\begin{conjA}\label{C:Main}
For every Artin-Tits monoid~$\MM$, the system~$\RDp\MM$ is semi-convergent.
\end{conjA}

We shall report in Section~\ref{S:Misc} about experimental data supporting Conjecture~$\ConjA$. For the moment, we just mention one example illustrating its predictions.

\begin{exam}\label{X:Att2}
Let $\MM$ be the Artin-Tits monoid of type~$\Att$. We saw in Example~\ref{X:Att} that $\RDb\MM$ is not convergent: $\aav = 1 / \ttc / \tta\ttb\tta$ admits the two distinct irreducible reducts $\tta\ttc / \ttc\tta / \ttb\tta$ and $\ttb\ttc / \ttc\ttb / \tta\ttb$. When we multiply the former by the inverse of the latter (see Subsection~\ref{SS:SemiAppli} below), we obtain the $6$-multifraction $\bbv = \tta\ttc / \ttc\tta / \ttb\tta / \tta\ttb / \ttc\ttb / \ttb\ttc$ which, by construction, is unital. Then Conjecture~$\ConjA$ predicts that $\bbv$ must reduce to~$\one$. This is indeed the case: we find
$$\bbv \act \Red{3, \tta\ttb} \Red{4, \ttc\ttb} \Red{5, \ttb\ttc} \Red{1, \tta\ttc} \Red{2, \ttc\ttb\ttc} \Red{3, \ttb\ttc} \Red{1, \ttb\ttc} = \one$$
 (as well as $\bbv \act \Red{5, \ttb\ttc} \Red{3, \tta\ttc} \Red{1, \tta\ttc} \Red{3, \ttb} \Red{4, \ttc} \Red{2, \ttc} = \one$: the reduction sequence is not unique).
\end{exam}

More generally, we can observe that, for every gcd-monoid~$\MM$, if $\aav'$ and $\aav''$ are two reducts of an $\nn$-multi\-fraction~$\aav$, then the $2\nn$-multifraction~$\bbv$ obtained by concatenating~$\aav'$ and the inverse of~$\aav''$ is always reducible whenever $\aav''$ is nontrivial: by construction, we have $\bb_\nn = \aa'_\nn$ and $\bb_{\nn + 1} = \aa''_\nn$. An obvious induction shows that $\aav \rds \aav'$ implies that $\aa'_\nn$ divides~$\aa_\nn$ (on the left or on the right, according to the sign of~$\nn$ in~$\aav$) and, similarly, $\aa''_\nn$ divides~$\aa_\nn$. Hence $\bb_\nn$ and~$\bb_{\nn + 1}$ admit a common multiple and, therefore, $\bbv$ is eligible for some reduction~$\Red{\nn, \xx}$ with $\xx \not= 1$ whenever $\bb_{\nn + 1}$, \ie, $\aa''_\nn$, is not~$1$. Finally, if $\aa''_\nn = 1$ holds, then $\bbv$ is eligible for~$\Red{2\nn - \mm, \aa''_\mm}$, where $\mm$ is the largest index such that $\aa''_\mm$ is nontrivial.

\subsection{Applications of semi-convergence}\label{SS:SemiAppli}

Most of the consequences of convergence already follow from semi-convergence---whence the interest of Conjecture~$\ConjA$. We successively consider the possibility of controlling the congruence~$\simeqb$, the decidability of the word problem, and what is called Property~$\PropH$.

\subsubsection*{Controlling~$\simeqb$}

In order to investigate~$\simeqb$ without convergence, we introduce a new operation on multifractions to represent inverses. 

\begin{nota}\label{N:REV}
We put $\REV\ef = \ef$, and, for every $\nn$-multifraction~$\aav$,
\begin{equation}
\REV\aav:= \begin{cases}
\ / \aa_\nn \sdots \aa_1 &\text{ if $\nn$ is positive in~$\aav$},\\
\ \aa_\nn \sdots \aa_1 &\text{ if $\nn$ is negative in~$\aav$}.
\end{cases}
\end{equation}
\end{nota}

\begin{lemm}\label{L:Inverse}
\ITEM1 For all multifractions~$\aav, \bbv$, we have $\REV{\aav \opp \bbv} = \REV\bbv \opp \REV\aav$.

\ITEM2 For every multifraction~$\aav$, we have $\can(\REV\aav) = \can(1 \opp \REV\aav) = \can(\aav)\inv$.
\end{lemm}

\begin{proof}
\ITEM1 If $\aav$ written as a sequence in~$\MM \cup \INV\MM$ is $(\xx_1 \wdots \xx_\nn)$, then $\REV\aav$ is $(\INV{\xx_\nn} \wdots \INV{\xx_1})$ (where we put $\INV{\INV\aa} = \aa$ for~$\aa$ in~$\MM$), and then $\REV{\aav \opp \bbv} = \REV\bbv \opp \REV\aav$ directly follows from the definition of the product.

Point~\ITEM2 comes from~\eqref{E:Eval}.
\end{proof}

Lemma~\ref{L:Inverse}\ITEM2 says that, if $\aav$ represents an element~$\gg$ of~$\EG\MM$, then both $\REV\aav$ and $1 \opp \REV\aav$ represent~$\gg\inv$. By definition, $\dh{\REV\aav} = \dh\aav$ always holds, and the operation $\,\REV{\ }\,$ on~$\FRb\MM$ is involutive. However, this operation does not restrict to~$\FRp\MM$: if $\aav$ is a positive multifraction, then $\REV\aav$ is positive if and only if $\dh\aav$ is even. In order to represent inverses inside~$\FRp\MM$, one can compose $\,\REV{\ }\,$ with a left translation by~$1$, thus representing the inverse of~$\aav$ by $1 \opp \REV\aav$. The inconvenience is that involutivity is lost: for~$\dh\aav$ odd and $\bbv = 1 \opp \REV\aav$, we find $1 \opp \REV\bbv = 1 \opp \aav \not= \aav$. 

\begin{lemm}\label{L:NCSimeq}
If $\MM$ is a gcd-monoid and $\RDp\MM$ is semi-convergent, then $\aav \simeqb \bbv$ is equivalent to~$1 \opp \aav \opp \REV\bbv \rds \one$ for all~$\aav, \bbv$ in~$\FRb\MM$.
\end{lemm}

\begin{proof}
As we have $1 \simeqb \ef$ and $\bbv \opp \REV\bbv \simeqb \ef$ by Lemma~\ref{L:Inverse}, $\aav \simeqb \bbv$ is equivalent to $1 \opp \aav \opp \REV\bbv \simeqb \ef$, hence, by~\eqref{E:Eval}, to $1 \opp \aav \opp \REV\bbv$ being unital, \ie, to $\can(1 \opp \aav \opp \REV\bbv) = 1$. By construction, $1 \opp \aav \opp \REV\bbv$ lies in~$\FRp\MM$. So, if $\RDp\MM$ is semi-convergent, $\can(1 \opp \aav \opp \REV\bbv) = 1$ is equivalent to $1 \opp \aav \opp \REV\bbv \rds \one$.
\end{proof}

As a direct application, we obtain

\begin{prop}\label{P:Embed}
If $\MM$ is a gcd-monoid and $\RDp\MM$ is semi-convergent, then $\MM$ embeds in its enveloping group~$\EG\MM$.
\end{prop}

\begin{proof}
Assume $\aa, \bb \in \MM$ and $\can(\aa) = \can(\bb)$, \ie, $\aa \simeqb \bb$. By Lemma~\ref{L:NCSimeq}, we must have $\aa \opp \REV\bb \rds \one$, which is $\aa / \bb \rds 1/1$. By definition of reduction, this means that there exists~$\xx$ in~$\MM$ satisfying $\aa/\bb \act \Red{1, \xx} = 1/1$. This implies $\aa = \bb \, (\, = \xx\, )$.
\end{proof}

\subsubsection*{The word problem for~$\EG\MM$}

If $\SS$ is any set, we denote by~$\SS^*$ the free monoid of all words in~$\SS$, using~$\ew$ for the empty word. For representing group elements, we consider words in~$\SS \cup \SSb$, where $\SSb$ is a disjoint copy of~$\SS$ consisting of one letter~$\INV\ss$ for each letter~$\ss$ of~$\SS$, due to represent~$\ss\inv$. If $\ww$ is a word in~$\SS \cup \SSb$, we denote by~$\INV\ww$ the signed word obtained from~$\ww$ by exchanging~$\ss$ and~$\INV\ss$ and reversing the order of letters. If $\MM$ is a monoid, $\SS$ is included in~$\MM$, and $\ww$ is a word in~$\SS$, we denote by~$\clp\ww$ the evaluation of~$\ww$ in~$\MM$. We extend this notation to words in~$\SS \cup \SSb$ by defining $\clp\ww$ to be the multifraction $\clp{\ww_1} \sdots \clp{\ww_{\nn}}$, where $(\ww_1 \wdots \ww_{\nn})$ is the unique sequence of words in~$\SS$ such that $\ww$ can be decomposed as $\ww_1 \, \INV{\ww_2}\, \ww_3 \, \INV{\ww_4} \, \pdots$ with $\ww_\ii \not= \ew$ for $1 < \ii \le \nn$.

\begin{lemm}\cite[Lemma~2.5]{Dit}\label{L:WP1}
For every monoid~$\MM$, every generating family~$\SS$ of~$\MM$, and every word~$\ww$ in~$\SS \cup \SSb$, the following are equivalent:

\ITEM1 The word $\ww$ represents~$1$ in~$\EG\MM$;

\ITEM2 The multifraction $\clp\ww$ satisfies $\clp\ww \simeqb 1$ in~$\FRb\MM$.
\end{lemm}

Thus solving the word problem for the group~$\EG\MM$ with respect to the generating set~$\SS$ amounts to deciding the relation $\clp\ww \simeqb 1$, which takes place in~$\FRb\MM$, hence essentially inside~$\MM$, as opposed to~$\EG\MM$. 

A few more definitions are needed. First, a gcd-monoid~$\MM$ is called \emph{strongly noetherian} if there exists a map~$\wit : \MM \to \NNNN$ satisfying, for all~$\aa, \bb$ in~$\MM$,
\begin{equation}\label{E:StrWit}
\wit(\aa\bb) \ge \wit(\aa) + \wit(\bb), \quad \text{and}\quad \wit(\aa) > 0 \text{\ for $\aa \not= 1$}.
\end{equation}
This condition is stronger than noetherianity, but it still follows from the existence of a presentation by homogeneous relations (same length on both sides): in this case, the word length induces a map~$\wit$ as in~\eqref{E:StrWit}. So every Artin-Tits monoid is strongly noetherian.

Next, we need the notion of a basic element. Noetherianity implies the existence of \emph{atoms}, namely elements that cannot be expressed as the product of two non-invertible elements. One shows \cite[Corollary~II.2.59]{Dir} that, if $\MM$ is a noetherian gcd-monoid, then a subfamily~$\SS$ of~$\MM$ generates~$\MM$ if and only if it contains all atoms of~$\MM$.

\begin{defi}\cite{Dgk}\label{D:Primitive}
If $\MM$ is a noetherian gcd-monoid, an element~$\aa$ of~$\MM$ is called \emph{right basic} if it belongs to the smallest family~$\XX$ that contains the atoms of~$\MM$ and is such that, if $\aa, \bb$ belong to~$\XX$ and $\aa \lcm \bb$ exists, then the element~$\aa'$ defined by $\aa \lcm \bb = \bb \aa'$ still belongs to~$\XX$. \emph{Left-basic} elements are defined symmetrically. We say that $\aa$ is \emph{basic} if it is right or left basic. 
\end{defi}

Note that, in the above definition, nothing is required when $\aa \lcm \bb$ does not exist. The key technical result is as follows:

\begin{lemm}\cite[Prop~3.27]{Dit}\label{L:Decid}
If $\MM$ is a strongly noetherian gcd-monoid with finite\-ly many basic elements and atom set~$\SS$, then the relation $\clp\ww \rds \one$ on $(\SS \cup \INV\SS)^*$ is decidable.
\end{lemm}

This result is \emph{not} trivial, because deciding whether a multifraction is eligible for some reduction requires to decide whether two elements of the ground monoid admit a common multiple, and this is the point, where the finiteness of the number of basic elements occurs crucially, as it provides an a priori upper bound on the size of this possible common multiple. Then, we immediately deduce:

\begin{prop}\label{P:WordPb}
If $\MM$ is a strongly noetherian gcd-monoid with finitely many basic elements and $\RDp\MM$ is semi-convergent, then the word problem for~$\EG\MM$ is decidable.
\end{prop}

\begin{proof}
Let $\SS$ be the atom set of~$\MM$. By Lemma~\ref{L:Decid}, the relation $\clp\ww \rds \one$ on words in~$\SS \cup \SSb$ is decidable. By~\eqref{E:Semi} (and by Lemma~\ref{L:Basics}\ITEM1), $\clp\ww \rds \one$ is equivalent to $\clp\ww \simeqb 1$. Finally, by Lemma~\ref{L:WP1}, $\clp\ww \simeqb 1$ is equivalent to $\ww$ representing~$1$ in~$\EG\MM$. Hence the latter relation is decidable.
\end{proof}

Note that, because the multifraction~$\clp\ww$ is always defined to be positive, we only need semi-convergence for~$\RDp\MM$ in the above argument. 

In the particular case of Artin-Tits monoids, we deduce

\begin{coro}
If Conjecture~$\ConjA$ is true, then the word problem for every Artin-Tits group is decidable.
\end{coro}

\begin{proof}
Let $\MM$ be an Artin-Tits monoid. We noted that $\MM$ is a gcd-monoid~\cite{BrS}, and that it is strongly noetherian. Next, $\MM$ has finitely many basic elements: this follows from (and, actually, is equivalent to) the result that every Artin-Tits monoid has a finite Garside family~\cite{Din, DyH}. Hence $\MM$ is eligible for Proposition~\ref{P:WordPb}.
\end{proof}


Let us conclude with algorithmic complexity. Lemma\,\ref{L:Decid} says nothing about the complexity of reduction. We show now the existence of an upper bound for the number of reductions. 

\begin{lemm}\label{L:Tower}
If $\MM$ is a strongly noetherian gcd-monoid with finitely many basic elements, then the number of reduction steps from an $\nn$-multifraction~$\aav$ is at most $\FF_\nn(\wit(\aa_1) \wdots \wit(\aa_\nn))$, where $\wit$ satisfies~\eqref{E:StrWit}, $\CC$ is the maximum of~$\wit(\aa) + 1$ for~$\aa$ basic in~$\MM$, and $\FF_\nn$ is inducti\-vely defined by $\FF_1(\xx) = \xx + 2$ and $\FF_\nn(\xx_1 \wdots \xx_\nn) = (\xx_1 + 1) \CC^{\FF_{\nn - 1} (\xx_2 \wdots \xx_\nn)}$.
\end{lemm}

\begin{proof}
An easy induction shows that the function~$\FF_\nn$ is increasing with respect to each variable and, for every $1 \le \ii < \nn$, it satisfies the inequality
\begin{equation}\label{E:Complex}
\FF_\nn(\CC \xx_1 \wdots \CC \xx_\ii, \xx_{\ii + 1}Ê- 1, \xx_{\ii + 2} \wdots \xx_\nn) < \FF_\nn(\xx_1 \wdots \xx_\nn).
\end{equation}
For~$\aav$ an $\nn$-multifraction, write~$\TT(\aav)$ for the maximal number of reduction steps from~$\aav$ and~$\wit(\aav)$ for $(\wit(\aa_1) \wdots \wit(\aa_\nn))$. We prove using induction on~$\rd$ the inequality $\TT(\aav) \le \FF_\nn(\wit(\aav))$ for every $\nn$-multifraction~$\aav$. Assume $\aav \act \Red{\ii, \xx} = \bbv$ with~$\xx$ an atom of~$\MM$ (what can assumed without loss of generality). We compare the sequences $\wit(\aav)$ and $\wit(\bbv)$. By definition, $\bb_{\ii + 1}$ is a proper divisor of~$\aa_{\ii + 1}$, which implies $\wit(\bb_{\ii + 1}) < \wit(\aa_{\ii + 1}) $. Next, $\aa_\ii$ is the product of at most $\wit(\aa_\ii)$ basic elements of~$\MM$, hence so is~$\bb_\ii$, implying $\wit(\bb_\ii) \le \CC \wit(\aa_\ii)$. Finally, $\aa_{\ii - 1}$ is the product of at most $\wit(\aa_{\ii - 1})$ basic elements of~$\MM$, hence $\bb_{\ii - 1}$ is the product of at most $\wit(\aa_{\ii - 1}) + 1$ basic elements, implying $\wit(\bb_{\ii - 1}) \le \CC \wit(\aa_{\ii - 1})$. Then the induction hypothesis implies $\TT(\bbv) \le \FF_\nn(\bbv)$, so, plugging the upper bounds for~$\bb_\ii$, and using that $\FF_\nn$ is increasing and~\eqref{E:Complex}, we find
$$\TT(\bbv) \le \FF_\nn(\CC\wit(\aa_1) \wdots \CC \wit(\aa_\ii), \wit(\aa_{\ii + 1})Ê- 1, \wit(\aa_{\ii + 2}) \wdots \wit(\aa_\nn)) < \FF_\nn(\wit(\aa_1) \wdots \wit(\aa_\nn)) = \FF_\nn(\wit(\aav)),$$
and $\TT(\aav) \le \TT(\bbv) + 1$ implies $\TT(\bbv) \le \FF_\nn(\bbv)$.
\end{proof}

The upper bound of Lemma~\ref{L:Tower} is not polynomial (very far from statistical data, which suggest a quadratic bound), but it is not very high either in the hierarchy of fast growing functions (it is ``primitive recursive''). From there, one can easily deduce a similar upper bound (tower of exponentials) for the word problem for~$\EG\MM$ when $\RDp\MM$ is semi-convergent.

\subsubsection*{Property~$\PropH$}

One says \cite{Dia, Dib, GoR} that \emph{Property~$\PropH$} is true for a presentation~$(\SS, \RR)$ of a monoid~$\MM$ if a word~$\ww$ in~$\SS \cup \SSb$ represents~$1$ in~$\EG\MM$ if and only if the empty word can be obtained from~$\ww$ using \emph{special} transformations, namely positive and negative equivalence and left and right reversing. Positive equivalence means replacing a positive factor of~$\ww$ (no letter~$\INV\ss$) with an $\RR$-equivalent word, negative equivalence means replacing the inverse of a positive factor with the inverse of an $\RR$-equivalent word, whereas right reversing consists in deleting some length two factor $\INV\ss \ss$ or replacing some length two factor~$\INV\ss \tt$ with~$\vv \INV\uu$ such that $\ss\vv = \tt\uu$ is a relation of~$\RR$, and left reversing consists in deleting some length two factor $\ss \INV\ss$ or replacing some length two factor~$\ss \INV\tt$ with~$\INV\uu \vv$ such that $\vv\ss = \uu\tt$ is a relation of~$\RR$. Roughly speaking, Property~$\PropH$ says that a word representing~$1$ can be transformed into the empty word without introducing new trivial factors~$\ss \INV\ss$ or~$\INV\ss \ss$, a situation directly reminiscent of Dehn's algorithm for hyperbolic groups, see~\cite[Section~1.2]{Dib}. 

Say that a presentation~$(\SS, \RR)$ of a monoid~$\MM$ is a \emph{right lcm presentation} if $\RR$ consists of one relation $\ss\uu = \tt\vv$ for each pair of generators~$\ss, \tt$ that admit a common right multiple, with $\ss\uu$ and~$\tt\vv$ representing~$\ss \lcm \tt$. The standard presentation of an Artin-Tits monoid is a right lcm presentation, and, symmetrically, a left lcm presentation.

\begin{prop}\label{P:PropH}
If $\MM$ is a gcd-monoid and $\RDp\MM$ is semi-convergent, Property~$\PropH$ is true for every presentation of~$\MM$ that is an lcm presentation on both sides.
\end{prop}

The point is that applying a rule~$\Red{\ii, \xx}$ to a multifraction~$\clp\ww$ can be decomposed into a sequence of special transformations as defined above. The argument is the same as in the case when $\RDb\MM$ is convergent \cite[Proposition~5.19]{Dit}, and we do not repeat it.

Thus Conjecture~$\ConjA$ would imply the statement conjectured in~\cite{Dia}:

\begin{coro}
If Conjecture~$\ConjA$ is true, Property~$\PropH$ is true for every Artin-Tits presentation.
\end{coro}


\subsection{Alternative forms}\label{SS:ConvAlt}

Here we mention several variants of semi-convergence.

\begin{prop}\label{P:NC}
If $\MM$ is a noetherian gcd-monoid, then $\RDb\MM$ $($\resp $\RDp\MM$$)$ is semi-convergent if and only if for every~$\aav$ in~$\FRb\MM$ $($\resp $\FRp\MM)$,
\begin{equation}\label{E:NC}
\parbox{110mm}{If $\aav$ is unital, then $\aav$ is either trivial or reducible.}
\end{equation}
\end{prop}

\begin{proof}
Assume that $\RDb\MM$ is semi-convergent, and let $\aav$ be a nontrivial unital multifraction in~$\FRb\MM$. By definition, $\aav \rds \one$ holds. As $\aav$ is nontrivial, the reduction requires at least one step, so $\aav$ cannot be $\RRR$-irreducible. Hence, \eqref{E:NC} is satisfied.

Conversely, assume~\eqref{E:NC}. As $\MM$ is noetherian, the rewrite system~$\RDb\MM$ is terminating, \ie, the relation~$\rd$ admits no infinite descending sequence. Hence we can use induction on~$\rd$ to establish~\eqref{E:Semi}. So let~$\aav$ be a unital multifraction in~$\FRb\MM$. If $\aav$ is $\rd$-minimal, \ie, if $\aav$ is $\RRR$-irreducible, then, by~\eqref{E:NC}, $\aav$ must be trivial, \ie, we have $\aav = \one$, whence $\aav \rds \one$. Otherwise, $\aav$ is $\RRR$-reducible, so there exist~$\ii, \xx$ such that $\bbv = \aav \act \Red{\ii, \xx}$ is defined. By construction, $\bbv$ is $\simeqb$-equivalent to~$\aav$, hence it is unital. By the induction hypothesis, we have $\bbv \rds \one$. By transitivity of~$\rds$, we deduce $\aav \rds \one$. Hence $\RDb\MM$ is semi-convergent.

The proof is similar for~$\RDp\MM$.
\end{proof}

Condition~\eqref{E:NC} can be restricted to more special unital multifractions.

\begin{defi}\label{D:Prime}
Call a multifraction~$\aav$ \emph{prime} if, for every~$\ii$ that is positive (\resp negative) in~$\aav$, the entries~$\aa_\ii$ and~$\aa_{\ii + 1}$ admit no nontrivial common right (\resp left) divisor.
\end{defi}

Since dividing adjacent entries by a common factor is a particular case of reduction, an $\RRR$-irreducible multifraction must be prime. The converse need not be true: for instance, the $6$-multifraction~$\bbv$ of Example~\ref{X:Att2} is prime, and it is $\RRR$-reducible.

\begin{prop}\label{P:NCPrime}
If $\MM$ is a noetherian gcd-monoid, then $\RDb\MM$ $($\resp $\RDp\MM$$)$ is semi-convergent if and only if for every~$\aav$ in~$\FRb\MM$ $($\resp $\FRp\MM)$,
\begin{equation}\label{E:NCPrime}
\parbox{110mm}{If $\aav$ is unital and prime, then $\aav$ is either trivial or reducible.}
\end{equation}
\end{prop}

\begin{proof}
By Proposition~\ref{P:NC}, the condition is necessary, since \eqref{E:NCPrime} is subsumed by~\eqref{E:NC}. For the converse implication, assume~\eqref{E:NCPrime}. As for Proposition~\ref{P:NC}, we establish~\eqref{E:Semi} using induction on~$\rd$. Let $\aav$ be a unital multifraction in~$\FRb\MM$. If $\aav$ is $\RRR$-irreducible, then it must be prime, for, otherwise, it is eligible for at least one division, which is a special case of reduction. Hence, $\aav$ must be~$\one$ by~\eqref{E:NCPrime}. Otherwise, $\aav$ is $\RRR$-reducible, there exist~$\ii, \xx$ such that $\bbv = \aav \act \Red{\ii, \xx}$ is defined, the induction hypothesis implies $\bbv \rds \one$, hence $\aav \rds \one$. Hence $\RDb\MM$ is semi-convergent. The proof for~$\RDp\MM$ is similar.
\end{proof}

\begin{coro}\label{C:NCPrime}
Conjecture~$\ConjA$ is true if and only if \eqref{E:NCPrime} holds for every Artin-Tits monoid~$\MM$ and every~$\aav$ in~$\FRp\MM$.
\end{coro}

We turn to another approach. Whenever the ground monoid~$\MM$ is noetherian, the rewrite systems~$\RDb\MM$ and~$\RDp\MM$ are terminating, hence they are convergent if and only if they are \emph{confluent}, meaning that
\begin{equation}\label{E:Conf}
\parbox{113mm}{If we have $\aav \rds \bbv$ and $\aav \rds \ccv$, there exists~$\ddv$ satisfying $\bbv \rds \ddv$ and $\ccv \rds \ddv$}
\end{equation}
(``diamond property''). We now observe that semi-convergence is equivalent to a weak form of confluence involving the unit multifractions~$\one$. 

\begin{prop}\label{P:1Conf}
If $\MM$ is a gcd-monoid, then $\RDb\MM$ $($\resp $\RDp\MM$$)$ is semi-convergent if and only if for every~$\aav$ in~$\FRb\MM$ $($\resp $\FRp\MM$$)$, 
\begin{equation}\label{E:1Conf}
\parbox{113mm}{The conjunction of $\aav \rds \bbv$ and $\aav \rds \one$ implies $\bbv \rds \one$.}
\end{equation}
\end{prop}

Relation~\eqref{E:1Conf} can be called \emph{$\one$-confluence} for~$\aav$, since it corresponds to the special case $\ccv = \one$ of~\eqref{E:Conf}: indeed, \eqref{E:Conf} with $\ccv = \one$ claims the existence of~$\ddv$ satisfying $\bbv \rds \ddv$ and $\one \rds \ddv$, and, as $\one$ is $\RRR$-irreducible, we must have $\ddv = \one$, whence $\bbv \rds \one$, as asserted in~\eqref{E:1Conf}. In order to establish Proposition~\ref{P:1Conf}, we need an auxiliary result, which connects~$\simeqb$ with the symmetric closure of~$\rds$ and is a sort of converse for Lemma~\ref{L:Basics}.

\begin{lemm}\label{L:Zigzag}
If $\MM$ is a gcd-monoid and $\aav, \bbv$ belong to~$\FRb\MM$, then $\aav \simeqb \bbv$ holds if and only if there exist a finite sequence $\ccv^0 \wdots \ccv^{2\rr}$ in~$\FRb\MM$ and $\pp, \qq$ in~$\ZZZZ$ satisfying
\begin{equation}\label{E:Zigzag}
\aav \opp \One\pp = \ccv^0 \rds \ccv^1 \antirds \ccv^2 \rds \ \pdots \ \antirds \ccv^{2\rr} = \bbv \opp \One\qq.
\end{equation}
\end{lemm}

\begin{proof}
For~$\aav, \bbv$ in~$\FRb\MM$, write $\aav \rdhs \bbv$ if $\aav \opp \rds \bbv \opp \One\pp$ holds for some~$\pp$. By Lemma~\ref{L:Basics}, $\aav \rds \bbv$ implies $\aav \opp \One\rr \rds \bbv \opp \One\rr$ and, therefore, the relation~$\rdhs$ is transitive. It is also compatible with multiplication: on the left, this follows from Lemma~\ref{L:Basics} directly. On the right, $\aav \rds \bbv \opp \One\pp$ implies $\aav \opp \ccv \rds \bbv \opp \One\pp\opp \ccv$ for every~$\ccv$, and we observe that $\One\pp \opp \ccv \rds \ccv \opp \One\pp$ always holds. Hence, the symmetric closure~$\approx$ of~$\rdhs$ is a congruence on~$\FRb\MM$. As we have $1 \rds \ef \opp 1$, $\aa / \aa \rds \ef \opp \One2$ and $/ \aa / \aa \rds \ef \opp \One{-2}$ for every~$\aa$ in~$\MM$, the congruence~$\approx$ contains pairs that generate~$\simeqb$. Hence $\aav \simeqb \bbv$ implies the existence of a zigzag in~$\rdhs$ and its inverse connecting~$\aav$ to~$\bbv$. Taking the maximum of~$\vert\rr\vert$ for~$\One\rr$ occurring in the zigzag, one obtains~\eqref{E:Zigzag}.
\end{proof}

\begin{proof}[Proof of Proposition~\ref{P:1Conf}]
Assume that $\RDb\MM$ is semi-convergent, and we have $\aav \rds \bbv$ and $\aav \rds \one$. By Lemma~\ref{L:Basics}, we have $\aav \simeqb \bbv$ and $\aav \simeqb 1$, hence $\bbv \simeqb \one$. As $\RDb\MM$ is semi-convergent, this implies $\bbv \rds \one$. So \eqref{E:1Conf} is satisfied, and $\RDb\MM$ is $\one$-confluent.

Conversely, assume that $\RDb\MM$ is $\one$-confluent. We first show using induction on~$\kk$ that, when we have a zigzag $\ccv^0 \rds \ccv^1 \antirds \ccv^2 \rds \ccv^3 \antirds \cc^4 \rds \pdots$, then $\ccv^0 = \one$ implies $\ccv^\kk \rds \one$ for every~$\kk$. For $\kk = 0$, this is the assumption. For $\kk$ even non-zero, we obtain $\ccv^\kk \rds \ccv^{\kk - 1} \rds \one$ using the induction hypothesis, whence $\ccv^\kk \rds \one$ by transitivity of~$\rds$. For~$\kk$ odd, we have $\ccv^{\kk - 1} \rds \one$ by the induction hypothesis and $\ccv^{\kk - 1} \rds \ccv^\kk$, whence $\ccv^\kk \rds \one$ by $\one$-confluence. 

Now assume that $\aav$ is unital. Lemma~\ref{L:Zigzag} provides $\pp, \qq, \rr$ and $\ccv^0 \wdots \ccv^{2\rr}$ satisfying
\begin{equation*}
\One\pp = \ccv^0 \rds \ccv^1 \antirds \ccv^2 \rds \ \pdots \ \antirds \ccv^{2\rr} = \aav \opp \One\qq.
\end{equation*}
As shown above, we deduce $\aav \opp \One\qq \rds \one$, whence $\aav \rds \one$ by Lemma~\ref{L:Basics}. Hence $\RDb\MM$ is semi-convergent.

Once again, the proof for~$\FRp\MM$ is the same.
\end{proof}

\begin{coro}\label{C:NC1Conf}
Conjecture~$\ConjA$ is true if and only if \eqref{E:1Conf} holds for every Artin-Tits monoid~$\MM$ and every~$\aav$ in~$\FRp\MM$.
\end{coro}

Proposition~\ref{P:1Conf} is important for testing Conjecture~$\ConjA$, because it shows that, if $\aav \rds \one$ holds, then every sequence of reductions from~$\aav$ inevitably leads to~$\one$. In other words, any reduction strategy may be applied without loss of generality.

\section{Divisions and tame reductions}\label{S:Tame}

When reduction is not convergent, it is not confluent either, and a multifraction may admit several reducts with no subsequent common reduct. However, by restricting to particular reductions, we can retrieve a (weak) form of confluence and let distinguished reducts appear. This is the approach we explore in this section. We start in Subsection~\ref{SS:Div} with divisions, which are particular reductions with good, but too weak properties. Then, in Subsection~\ref{SS:Tame}, we extend divisions into what we call tame reductions, which are those reductions that, in a sense, exclude no subsequent opportunities. Extending the example of divisions to tame reductions leads us in Subsection~\ref{SS:Dertame} to the natural notion of a maximal tame reduction and to Conjecture~$\ConjB$ about tame reductions from unital multifractions, which is stronger but more precise than Conjecture~$\ConjA$. 

 We feel that the many technical details, examples, and counter-examples appearing in this section and the next one are important, because they illustrate how subtle the mechanism of reduction is. Skipping such details would induce a superficial view and misleadingly suggest that things are more simple than they really are, possibly leading to naive attempts with no chance of success.

\subsection{Divisions}\label{SS:Div}

 Divisions are the most direct counterparts of free reductions in free monoids. They are the special cases of reduction when no remainder appears. No confluence result can be expected for divisions in a non-free monoid, but we shall see in Proposition~\ref{P:DerDiv}, the main result of this subsection, that, for every multifraction~$\aav$, there exists a unique, well-defined maximal reduct~$\derdiv\aav$ accessible from~$\aav$ by divisions. 

Following the model of reductions, we first fix notation for divisions.

\begin{defi}\label{D:Div}
If $\MM$ is a gcd-monoid and $\aav, \bbv$ belong to~$\FRb\MM$, we declare that $\bbv = \aav \act \Rdiv{\ii, \xx}$ holds if we have $\bbv = \aav \act \Red{\ii, \xx}$ and, in addition, $\xx$ right (\resp left) divides~$\aa_\ii$ if $\ii$ is positive (\resp negative) in~$\aav$. We use~$\RDivb\MM$ for the family of all~$\Rdiv{\ii, \xx}$ with~$\xx \not= 1$, write $\aav \rddiv \bbv $ if some rule of~$\RDivb\MM$ maps~$\aav$ to~$\bbv$, and $\rddivs$ for the reflexive--transitive closure of~$\rddiv$.
\end{defi}

So $\aav \act \Rdiv{\ii, \xx}$ is defined if and only if $\xx$ divides~$\aa_\ii$ and~$\aa_{\ii + 1}$ on the due side, and applying~$\Rdiv{\ii, \xx}$ means dividing~$\aa_\ii$ and~$\aa_{\ii+1}$ by~$\xx$. By definition, $\aav$ is $\DDD$-irreducible if and only if it is prime (Definition~\ref{D:Prime}), \ie, the gcds of adjacent entries (on the relevant side) are trivial.

Except in degenerated cases, \eg, free monoids, the system~$\RDivb\MM$ is not convergent: typically, for $\MM = \MON{\tta, \ttb}{\tta\ttb\tta = \ttb\tta\ttb}$ (Artin's $3$-strand braid monoid) and $\aav = \tta / \tta\ttb\tta / \ttb$, we find $\aav \act \Rdiv{2, \ttb} = \tta / \tta\ttb / 1$ and $\aav \act \Rdiv{1, \tta} = 1 / \tta\ttb / \ttb$, with no further division, and confluence can be restored only at the expense of applying some reduction~$\Red{\ii, \xx}$, here $\aav \act \Rdiv{2, \ttb} = \aav \act \Rdiv{1, \tta}\Red{2, \ttb}$. However, we shall see now that, for every multifraction~$\aav$, there exists a unique, well-defined maximal $\RRR$-reduct of~$\aav$ that can be obtained using divisions. 

 The first step is to observe that, for each level~$\ii$, there is always a maximal division at level~$\ii$, namely dividing by the gcd of the $\ii$th and $(\ii + 1)$st entries (on the due side). Indeed, assuming for instance $\ii$ positive in~$\aav$, the multifraction $\aav \act \Rdiv{\ii, \xx}$ is defined if and only if $\xx$ right divides both~$\aa_\ii$ and~$\aa_{\ii + 1}$, hence if and only if $\xx$ right divides the right gcd~$\aa_\ii \gcdt \aa_{\ii + 1}$.

\begin{nota}\label{N:Divmax}
 If $\MM$ is a gcd-monoid and $\aav$ is a multifraction on~$\MM$, then, for $\ii < \dh\aav$, we write $\aav \act \Divmax\ii$ for $\aav \act \Rdiv{\ii, \xx}$ with~$\xx$ the gcd of~$\aa_\ii$ and~$\aa_{\ii + 1}$ on the due side.
\end{nota}

 Next, we observe that, contrary to general irreducibility, (local) primeness is robust, in that, once obtained, it cannot be destroyed by subsequent divisions: 

\begin{lemm}
Say that a multifraction~$\aav$ is \emph{$\jj$-prime} if $\aav \act \Rdiv{\jj, \yy}$ is defined for no~$\yy \not= 1$. If $\aav$ is $\jj$-prime, then so is $\aav \act \Rdiv{\ii, \xx}$ for all~$\ii, \xx$.
\end{lemm}

\begin{proof}
Assume for instance $\jj$ positive in~$\aav$, and let $\bbv = \aav \act \Rdiv{\ii, \xx}$. We have either $\bb_\jj = \aa_\jj$ (for $\ii < \jj - 1$ and $\ii > \jj$) or $\bb_\jj \dive \aa_\jj$ (for $\ii = \jj - 1$); similarly, we have either $\bb_{\jj + 1} = \aa_{\jj + 1}$ (for $\ii < \jj$ and $\ii > \jj + 1$) or $\bb_{\jj + 1} \dive \aa_{\jj + 1}$ (for $\ii = \jj + 1$). Hence, in all cases, the assumption $\aa_\ii \gcdt \aa_{\ii + 1} = 1$ implies $\bb_\ii \gcdt \bb_{\ii + 1} = 1$. 
\end{proof}

Hence, if we start with a multifraction~$\aav$ and apply, in any order, maximal divisions~$\Divmax\ii$ in such a way that every level between~$1$ and~$\dh\aav-1$ is visited at least one, we always finish with a prime multifraction. The latter may depend on the order of the divisions, but we shall now see that there exists a preferred choice. 

\begin{prop}\label{P:DerDiv}
Let $\MM$ be a gcd-monoid. For every $\nn$-multifraction~$\aav$ on~$\MM$, put
\begin{equation}\label{E:DerDiv1}
\derdiv\aav:= \aav \act \Divmax{\nn - 1} \Divmax{\nn - 2} \pdots \Divmax1.
\end{equation}\label{E:DerDiv}
Then, $\derdiv\aav$ is prime, and, for every multifraction~$\bbv$ on~$\MM$, 
\begin{equation}\label{E:DerDiv2}
\text{$\aav \rddivs \bbv$ \quad implies \quad $\bbv \rddivs \derdiv\bbv \rds \derdiv\aav$.}
\end{equation}
\end{prop}

 Thus $\derdiv\aav$ is a reduct of every multifraction obtained from~$\aav$ using division. The proof of Proposition~\ref{P:DerDiv} is nontrivial and requires to precisely control the way divisions and reductions can be commuted. We begin with a confluence result. By~\cite[Lemma~4.6]{Dit}, there always exists a confluence solution for any two reductions at level~$\ii$ and~$\ii + 1$. This applies of course when one of the reductions is a division, but, in that case, we can say more. 

\begin{lemm}\label{L:LocConfDiv}
Assume that both $\aav \act \Red{\ii + 1, \xx}$ and $\aav \act \Rdiv{\ii, \yy}$ are defined. Then we have $\aav \act \Red{\ii + 1, \xx} \Rdiv{\ii, \zz} = \aav \act \Rdiv{\ii, \yy} \Red{\ii + 1, \xx}$, where $\zz$ is determined by the equalities 
$\aa_\ii = \aa \yy$, $\xx \lcm \aa = \aa \vv$, and $\vv \zz = \vv \lcm \yy$ (\resp $\aa_\ii = \yy \aa$, $\xx \lcmt \aa = \vv\aa$, and $\zz\vv = \vv \lcmt \yy$) for $\ii$ positive (\resp negative) in~$\aav$. Moreover, if $\yy$ is maximal for~$\aav$ (\ie, $\yy$ is the gcd of~$\aa_\ii$ and~$\aa_{\ii + 1}$), then $\zz$ is maximal for~$\aav \act \Red{\ii + 1, \xx}$.
\end{lemm}

\begin{proof}
(Figure~\ref{F:LocConfDiv}) Assume that $\ii$ is positive in~$\aav$, so $\aa_{\ii + 1}$ is negative in~$\aav$. Put $\bbv := \aav \act \Red{\ii + 1, \xx}$ and $\ccv := \aav \act \Rdiv{\ii, \yy}$. By definition, there exists~$\xx'$ satisfying
\begin{gather*}
\bb_{\ii-1} = \aa_{\ii-1}, \quad \bb_\ii = \aa_\ii \xx', \quad \xx\bb_{\ii + 1} = \aa_{\ii+ 1} \xx' = \xx \lcm \aa_{\ii + 1}, \quad \xx \bb_{\ii+2} = \aa_{\ii+2},\\
\cc_{\ii - 1} = \aa_{\ii - 1}, \quad \cc_\ii \yy = \aa_\ii , \quad \cc_\ii \yy = \aa_\ii, \quad \cc_{\ii + 2} = \aa_{\ii + 2}.
\end{gather*}
As $\aa_{\ii + 1}$ is~$\cc_{\ii + 1} \yy$, Lemma~\ref{L:IterLcm} implies the existence of~$\uu$, $\vv$, and~$\zz$ satisfying
\begin{equation}\label{E:AdjCase1}
\bb_{\ii + 1} = \uu \zz \quad \text{with} \quad \cc_{\ii + 1} \vv = \xx \uu = \xx \lcm \cc_{\ii + 1} \quand \yy \xx' = \vv \zz = \yy \lcm \vv.
\end{equation}
By construction, we have $\bb_\ii = \aa_\ii \xx' = \cc_\ii \yy \xx' = \cc_\ii \vv \zz$, which shows that $\zz$ right divides both~$\bb_\ii$ and~$\bb_{\ii + 1}$. Hence $\ddv := \bbv \act \Rdiv{\ii, \zz}$ is defined, and we have 
\begin{equation}\label{E:LocConfDiv1}
\dd_{\ii - 1} = \aa_{\ii - 1}, \quad \dd_\ii = \cc_\vv, \quad \dd_{\ii + 1} = \uu, \quad \dd_{\ii + 2} = \bb_{\ii + 2}.
\end{equation}
On the other hand, by assumption, $\xx$ left divides~$\cc_{\ii + 2}$, which is~$\aa_{\ii + 2}$, and $\xx$ and~$\cc_{\ii + 1}$ admit a common right multiple, namely their right lcm $\xx\uu$. Hence, $\ccv \act \Red{\ii, \xx}$ is defined, and comparing with~\eqref{E:LocConfDiv1} directly yields the expected equality $\ccv \act \Red{\ii, \xx} = \bbv \act \Rdiv{\ii, \zz}$.

It remains to prove that, if $\yy$ is maximal for~$\aav$, then $\zz$ is maximal for~$\bbv$. So assume $\yy = \aa_\ii \gcdt \aa_{\ii + 1}$. We deduce $\cc_\ii \gcdt \cc_{\ii + 1} = 1$. On the other hand, by Lemma~\ref{L:Lcm}, the assumption $\cc_{\ii + 1} \vv = \cc_{\ii + 1} \lcm \xx$ implies $\uu \gcdt \vv = 1$. Then (the symmetric counterpart of) Lemma~\ref{L:IterGcd} implies $\cc_\ii \vv \gcdt \uu = 1$, whence $\bb_\ii \gcdt \bb_{\ii + 1} = \zz$. 

A symmetric argument applies when $\ii$ is negative in~$\aav$.
\end{proof}

\begin{figure}[htb]
\begin{picture}(70,32)(0,7)
\psset{nodesep=0.7mm}
\psline[style=back,linecolor=color2](-8,11.5)(15,11.5)(15,22)(55,22)(55,31.5)(70.5,31.5)(78,31.5)
\psline[style=back,linecolor=color1](-8,12)(15,12)(35,12)(35,32)(78,32)
\pcline{->}(35,32)(55,32)\taput{$\xx$}
\pcline{->}(55,32)(70,32)\taput{$\bb_{\ii+2}$}
\psline[style=thin](35,35)(35,37)(69,37)(69,35)\put(48,39){$\aa_{\ii + 2}$}
\pcline{->}(35,32)(35,22)\tlput{$\cc_{\ii + 1}$}\trput{$\aa$}
\pcline{->}(55,32)(55,22)\trput{$\uu$}
\pcline{->}(15,22)(35,22)\taput{$\cc_\ii$}
\pcline{->}(35,22)(55,22)\taput{$\vv$}
\psarc[style=thin](55,22){3.5}{90}{180}
\psarc[style=thin](55,12){3.5}{90}{180}
\pcline[style=double](15,22)(15,12)
\pcline{->}(35,22)(35,12)\trput{$\yy$}
\pcline{->}(55,22)(55,12)\trput{$\zz$}
\pcline{->}(15,12)(35,12)\tbput{$\aa_\ii$}
\pcline{->}(35,12)(55,12)\tbput{$\xx'$}
\psline[style=thin](16,10)(16,8)(55,8)(55,10)
\pcline[linestyle=none](16,8)(55,8)\tbput{$\bb_\ii$}
\pcline{<-}(0,12)(15,12)\tbput{$\aa_{\ii - 1}$}
\psline[style=thin](57,12)(59,12)(59,31)(57,31)\put(60,21){$\bb_{\ii + 1}$}
\psline[style=exist]{->}(-8,12)(0,12)
\psline[style=exist]{<-}(70,32)(78,32)
\end{picture}
\caption{\small Proof of Lemma~\ref{L:LocConfDiv}: if $\aav$ is eligible both for~$\Red{\ii + 1, \xx}$ and~$\Rdiv{\ii, \yy}$, we can start with either and converge to the colored path.}
\label{F:LocConfDiv}
\end{figure}
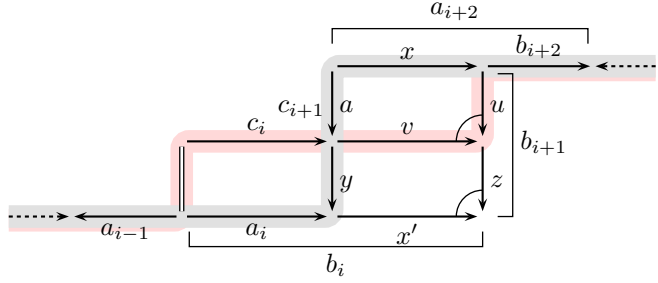

 Next, we observe that reduction and division commute when performed at distant levels: the result is easy for very distant levels, slightly more delicate when the levels are closer. 

\begin{lemm}\label{L:Remote}
Assume $\bbv = \aav \act \Red{\ii, \xx}$, and $\jj \not= \ii, \ii + 1, \ii - 2$ $($\resp $\jj = \ii - 2$$)$. Put $\aav':= \aav \act \Divmax\jj$ and $\bbv' := \bbv \act \Divmax\jj$. Then we have $\bbv' = \aav' \act \Red{\ii, \xx}$ $($\resp $\bbv' = \aav' \act \Red{\ii, \xx}\Divmax{\ii -2}$$)$.
\end{lemm}

\begin{proof}
For $\jj \le \ii - 3$ or $\jj \ge \ii + 2$, the reduction~$\Red{\ii, \xx}$ does not change the $\jj$th and $(\jj + 1)$st entries, so the greatest division at level~$\jj$ remains the same, and commutation is straightforward. For $\jj = \ii - 1$, Lemma~\ref{L:LocConfDiv} gives $\bbv' = \aav' \act \Red{\ii, \xx}$.

Assume $\jj = \ii - 2$. Then the reduction~$\Red{\ii, \xx}$ does not change the $\jj$th entry, but it possibly increases the $(\jj + 1)$st entry. So, if $\Rdiv{\jj, \yy}$ is the maximal $\jj$-division for~$\aav$, then $\Rdiv{\jj, \yy}$ applies to~$\bbv$, but it need not be the maximal $\jj$-division for~$\bbv$. Expanding the definitions, we obtain the commutation relation $\aav \act \Red{\ii, \xx} \Rdiv{\jj, \yy} = \aav \act \Rdiv{\jj, \yy}\Red{\ii, \xx}$, meaning $\bbv \act \Rdiv{\jj, \yy} = \aav' \act \Red{\ii, \xx}$, together with $\bbv' = (\bbv \act \Rdiv{\jj, \yy})\Divmax{\jj}$, whence $\bbv' = \aav' \act \Red{\ii, \xx}\Divmax{\ii -2}$, as expected.
\end{proof}

 The last preliminary result, needed for the end of the proof of Proposition~\ref{P:DerDiv}, connects~$\derdiv\aav$ and~$\derdiv\bbv$ in the (very special) case when $\bbv$ is an elementary reduct of~$\aav$ and $\aav$ is prime at every sufficiently large level. 

\begin{lemm}\label{L:SmallIndices}
If $\aav$ and $\bbv$ are $\jj$-prime for $\jj \ge \ii$, then $\bb = \aa \act \Red{\ii + 1, \xx}$ implies $\derdiv\aav \rds \derdiv\bbv$.
\end{lemm}

\begin{proof}
By assumption, we have $\derdiv\aav = \aav \act \Divmax{\ii -1} \pdots \Divmax1$ and $\derdiv\bbv = \bbv \act \Divmax{\ii -1} \pdots \Divmax1$. Put $\aav^\ii:= \aav$, $\bbv^\ii := \bbv$ and, inductively, $\aav^\jj = \aav^{\jj + 1} \act \Divmax\jj$, $\bbv^\jj = \bbv^{\jj + 1} \act \Divmax\jj$ for~$\jj$ decreasing from~$\ii - 1$ to~$1$, yielding $\derdiv\aav = \aav^1$ and $\derdiv\bbv = \bbv^1$. We prove using induction on~$\jj$ decreasing from~$\ii$ to~$1$ that, for every~$\yy$, there exist $\xx_0$, $\xx_1 \wdots \xx_\kk$ with $\ii - 2\kk > \jj$ satisfying $\bbv^\jj =\nobreak \aav^\jj \act \Red{\ii + 1, \xx_0} \Red{\ii - 1, \xx_1} \pdots \Red{\xx - 2\kk + 1, \xx_\kk}$. By assumption, the property is true for $\jj = \ii$, with $\kk =\nobreak 0$ and~$\xx_0 = \xx$. Assume $\ii > \jj \ge 1$. By induction hypothesis, we have $\bbv^{\jj + 1} = \aav^{\jj + 1} \act \Red{\ii + 1, \xx_0} \Red{\ii - 1, \xx_1} \pdots \Red{\ii - 2\kk + 1, \xx_\kk}$ for some $\xx_0 \wdots \xx_\kk$ with $\ii - 2 \kk > \jj + 1$. By repeated applications of Lemma~\ref{L:Remote}, we deduce that each reduction~$\Red{\ii - 2\ell + 1, \xx_\ell}$ commutes with~$\Divmax\jj$, except the last one in the case $\ii - 2\kk = \jj + 2$, in which case Lemma~\ref{L:Remote} prescribes to add one more reduction (a division) $\Red{\ii - 2\kk - 1, \xx_{\kk + 1}}$. In this way, we obtain either $\bbv^\jj = \aav^\jj \act \Red{\ii + 1, \xx_0} \Red{\ii - 1, \xx_1} \pdots \Red{\ii - 2\kk + 1, \xx_\kk}$, or $\bbv^\jj = \aav^\jj \act \Red{\ii + 1, \xx_0} \Red{\ii - 1, \xx_1} \pdots \Red{\ii - 2\kk + 1, \xx_\kk}\Red{\ii - 2\kk - 1, \xx_{\kk + 1}}$, and the induction continues.
\end{proof}

We can now complete the argument for Proposition~\ref{P:DerDiv}. The proof that $\derdiv\aav$ is prime is easy, but that of the relation~\eqref{E:DerDiv} is more delicate. 

\begin{proof}[Proof of Proposition~\ref{P:DerDiv}]
It follows from the definition that $\aav \act \Divmax{\nn-1}$ is $(\nn - 1)$-prime, then that $\aav \act \Divmax{\nn-1}\Divmax{\nn-2}$ is $(\nn - 1)$- and $(\nn - 2)$-prime, etc., hence that $\derdiv\aav$ is $\ii$-prime for $1 \le \ii < \nn$, hence it is prime.

We now establish~\eqref{E:DerDiv2}, \ie, prove that $\aav \rddivs \bbv$ implies $\derdiv\bbv \rds \derdiv\aav$. For an induction, it is sufficient to prove that $\aav \rddiv \bbv$ implies $\derdiv\bbv \rds \derdiv\aav$. So, we assume $\bbv = \aav \act \Rdiv{\ii, \zz}$, and aim at proving $\derdiv\bbv \rddivs \derdiv\aav$. Put $\dh\aav = \nn$. By definition, $\derdiv\aav$ and $\derdiv\bbv$ are obtained by performing $\nn - 1$ successive divisions, and we shall establish a step-by-step connection summarized in Figure~\ref{F:Partial}. Put $\aav^\nn := \aav$, $\bbv^\nn := \bbv$, and let $\aav^\ii$ (\resp $\bbv^\ii$) be obtained from~$\aav^{\ii + 1}$ (\resp $\bbv^{\ii + 1}$) by applying $\Divmax\ii$, so that we finally have $\derdiv\aav = \aav^1$ and $\derdiv\bbv = \bbv^1$. We assume that $\ii$ is positive in~$\aav$.

Consider first $\jj \ge \ii + 2$. Applying Lemma~Ê\ref{L:Remote}, we inductively obtain $\bbv^\jj = \aav^\jj \act \Rdiv{\ii, \zz}$, implying the commutativity of the $\nn - \ii - 2$ left hand squares in the diagram of Figure~\ref{F:Partial}. 

Now let $\xx = \aa^{\ii + 1} \gcd \aa^{\ii + 2}$. By definition, we have $\aav^{\ii + 1} = \aav^{\ii + 2} \act \Rdiv{\ii + 1, \xx}$. By Lemma~\ref{L:LocConfDiv}, there exists~$\ccv$ and~$\zz'$ satisfying $\ccv = \aav^{\ii + 2} \act \Rdiv{\ii + 1, \xx} \Rdiv{\ii, \zz'} = \aav^{\ii + 2} \act \Rdiv{\ii, \zz} \Red{\ii + 1, \xx}$, which reads $\ccv = \aav^{\ii + 1} \Rdiv{\ii, \zz'} = \bbv^{\ii + 2} \act \Red{\ii + 1, \xx}$.

Next, $\ccv$ is obtained from~$\aav^{\ii + 1}$ by some $\ii$-division, whereas $\aav^\ii$ is obtained from~$\aav^{\ii +1}$ by the maximal $\ii$-division, hence $\aav^\ii$ must be obtained from~$\ccv$ by some further $\ii$-division, namely the maximal $\ii$-division for~$\ccv$. So, $\aav^\ii = \ccv \act \Divmax\ii$ holds.

On the other hand, $\bbv^{\ii + 1}$ is obtained from~$\bbv^{\ii + 2}$ by the maximal $(\ii + 1)$-division, namely~$\Rdiv{\ii + 1, \yy}$ with $\yy = \bb_{\ii + 1} \gcd \bb_{\ii + 2}$. As we have $\xx = \aa_{\ii + 2} \gcd \aa_{\ii + 2}$ and $\bb_{\ii + 2} = \aa_{\ii + 2}$, the relation $\bb_{\ii + 1} \dive \aa_{\ii + 1}$ implies $\yy \dive \xx$, say $\xx = \yy \xx'$. If follows that reducing~$\xx$ at level~$\ii + 1$ in~$\bbv^{\ii + 2}$ amounts to first dividing by~$\yy$ and then reducing~$\xx'$, \ie, we have $\bbv \act \Red{\ii + 1, \xx} = \bbv \act \Rdiv{\ii + 1, \yy} \Red{\ii + 1, \xx'} = \bbv^{\ii + 1} \act \Red{\ii + 1, \xx'}$.

Now, two reductions apply to~$\bbv^{\ii + 1}$, namely~$\Red{\ii + 1, \xx'}$, which leads to~$\ccv$, and $\Divmax\ii$, which is, say, $\Rdiv{\ii + 1, \yy}$ and leads to~$\bbv^\ii$. Applying Lemma~\ref{L:LocConfDiv} again, we obtain the existence of~$\yy'$ satisfying $\bb^{\ii + 1}Ê\act \Red{\ii + 1, \xx'} \Rdiv{\ii, \yy'} = \bb^{\ii + 1} \act \Rdiv{\ii, \yy} \Red{\ii + 1, \xx'}$, which boils down to $\ccv \act \Rdiv{\ii, \yy'}Ê= \bbv^\ii \act \Red{\ii + 1, \xx'}$. Moreover, as $\Rdiv{\ii, \yy}$ is the maximal $\ii$-division applying to~$\bbv^{\ii + 1}$, Lemma~\ref{L:LocConfDiv} implies that $\Rdiv{\ii, \yy'}$ is the maximal $\ii$-division applying to~$\ccv$. We obtained above $\aav^\ii = \ccv \act \Divmax\ii$, so we deduce $\aav^\ii = \bbv^\ii \act \Red{\ii + 1, \xx'}$.

From there, we are in position for applying Lemma~\ref{L:SmallIndices} (with $\aav$ and $\bbv$ interchanged): we have $\aav^\ii = \bbv^\ii \act \Red{\ii + 1, \xx'}$, and, by construction, $\aav^\ii$ and~$\bbv^\ii$ are $\jj$-prime for every~$\jj \ge \ii$. Then Lemma~\ref{L:SmallIndices} ensures $\bbv^1 \rds \aav^1$, which is $\derdiv\bbv \rds \derdiv\aav$.
\end{proof}

\begin{figure}[htb]
$\setlength{\arraycolsep}{5.8mm} 
\begin{array}{cccccccc}
\Rnode{n12}{\aav^\nn} & \Rnode{n22}{\aav^{\nn - 1}} & \Rnode{n32}{\aav^{\ii + 2}} & \Rnode{n42}{\aav^{\ii + 1}} & \Rnode{n52}{\aav^\ii} \\[9mm]
\Rnode{n11}{\bbv^\nn} & \Rnode{n21}{\bbv^{\nn - 1}} & \Rnode{n31}{\bbv^{\ii + 2}} & \Rnode{n41}{\ccv} & \Rnode{n51}{\aav^\ii} & \Rnode{n61}{\aav^{\ii - 1}} & \Rnode{n71}{\aav^2} & \Rnode{n81}{\derdiv\aav}\\[9mm]
&& \Rnode{n30}{\bbv^{\ii + 2}} & \Rnode{n40}{\bbv^{\ii + 1}} & \Rnode{n50}{\bbv^\ii} & \Rnode{n60}{\bbv^{\ii - 1}} & \Rnode{n70}{\bbv^2} & \Rnode{n80}{\derdiv\bbv}
\end{array}
\psset{nodesep=5pt}
\ncline[style=double]{->}{n12}{n22}\taput{\Divmax{\nn - 1}}
\ncline[style=etc]{n22}{n32}
\ncline[style=double]{->}{n32}{n42}\taput{\Divmax{\ii + 1}}
\ncline[style=double]{->}{n42}{n52}\taput{\Divmax\ii}
\ncline[style=double]{->}{n12}{n11}\trput{\Rdiv{\ii, \zz}}
\ncline[style=double]{->}{n22}{n21}\trput{\Rdiv{\ii, \zz}}
\ncline[style=double]{->}{n32}{n31}\trput{\Rdiv{\ii, \zz}}
\ncline[style=double]{->}{n42}{n41}\trput{\Rdiv{\ii, \zz'}}
\ncline[style=double]{n52}{n51}
\ncline[style=double]{->}{n51}{n61}\taput{\Divmax{\ii - 1}}
\ncline[style=etc]{n61}{n71}
\ncline[style=double]{->}{n71}{n81}\taput{\Divmax1}
\ncline[style=etc]{n11}{n21}
\ncline[style=double]{->}{n11}{n21}\tbput{\Divmax{\nn - 1}}
\ncline[style=etc]{n21}{n31}
\ncline[style=double]{->}{n31}{n41}\tbput{\Red{\ii + 1, \xx}}
\ncline[style=double]{->}{n41}{n51}\tbput{\Divmax\ii}
\ncline[style=double]{n31}{n30}
\ncline[style=double]{<-}{n41}{n40}\trput{\Red{\ii + 1, \xx'}}
\ncline[style=double]{<-}{n51}{n50}\trput{\Red{\ii + 1, \xx'}}
\ncline[linecolor=white]{n61}{n60}\tlput{\rlap{\quad Lemma~\ref{L:SmallIndices}}}
\ncline[style=double]{<-}{n81}{n80}\trput{*}
\ncline[style=double]{->}{n50}{n60}\tbput{\Divmax{\ii - 1}}
\ncline[style=etc]{n60}{n70}
\ncline[style=double]{->}{n70}{n80}\tbput{\Divmax1}
\ncline[style=double]{->}{n30}{n40}\tbput{\Divmax{\ii + 1}}
\ncline[style=double]{->}{n40}{n50}\tbput{\Divmax\ii}
$
\caption{\small Comparing the computations of $\derdiv\aav$ and $\derdiv(\aav \act \Rdiv{\ii, \xx})$.}
\label{F:Partial}
\end{figure}

 Let us denote by~$\Irr(\aav)$ the family of all $\RRR$-irreducible reducts of~$\aav$. The failure of confluence means that $\Irr(\aav)$ may content more than one element, and controlling~$\Irr(\aav)$ is one of the main challenges in the current approach. As, for every multifraction~$\aav$, we have now a distinguished reduct~$\derdiv\aav$, a natural task is to compare~$\Irr(\derdiv\aav)$ with~$\Irr(\aav)$. By definition, $\aav \rddivs \derdiv\aav$ implies $\Irr(\derdiv\aav) \subseteq \Irr(\aav)$, and, by Proposition~\ref{P:DerDiv}, $\aav \rddiv \bbv$ implies $\bbv \rds \derdiv\aav$, whence $\Irr(\derdiv\aav) \subseteq \Irr(\derdiv\bbv)$. The next examples show these easy inclusions are the best we can expect in general. 

\begin{exam}\label{X:Der}
In the Artin-Tits monoid of type~$\Att$, let $\aav = \fr{ab/aba/aca}$. Then one finds $\Irr(\aav) = \{\aav_1, \aav_2\}$, with $\aav_1 = \aav \act \Rdiv{2, \tta} = \fr{ab/ba/ca}$ and $\aav_2 = \aav \act \Rdiv{1, \fr{ab}} \Red{2, \fr{c}} = \fr{cb/bc/ac}$. Now, we obtain $\derdiv\aav = \aav_1$, whence $\Irr(\derdiv\aav) = \{\aav_1\}$: so, by performing divisions, we lost one of the irreducible reducts of~$\aav$. On the other hand, for $\bbv = \aav \act \Rdiv{1, \fr{ab}} = \fr{1/b/aca}$, we find $\derdiv\bbv = \bbv$ and $\Irr(\derdiv\bbv) = \Irr(\bbv) = \{\aav_1, \aav_2\} = \Irr(\aav)$: so $\aav \rddiv \bbv$ does not imply $\Irr(\derdiv\bbv) \subseteq \Irr(\derdiv\aav)$. 
\end{exam}

\begin{rema}
 The order of divisions is important in the definition of~$\derdiv\aav$ and, even at the expense of using reductions instead of divisions, we cannot start from low levels in general. Indeed, for $\dh\aav = 3$, Lemma~\ref{L:LocConfDiv} implies $\derdiv\aav = \aav \act \Red{1, \xx_1}\Red{2, \xx_2}$, with $\xx_\ii$ the (relevant) gcd of~$\aa_\ii$ and~$\aa_{\ii + 1}$, but this expression of~$\der\aav$ as $\aav \act \prod_{\ii = 1}^{\ii = \dh\aav - 1} \Red{\ii, \xx_\ii}$ with~$\xx_\ii$ gcd of~$\aa_\ii$ and~$\aa_{\ii + 1}$ does not work for $\dh\aav \ge 4$: for instance, for $\aav := \aa/\aa/\aa/\aa$, one finds $\aav \act \Red{1, \xx_1}\Red{2, \xx_2}\Red{3, \xx_3} = \aa/\aa/1/1 \not= \derdiv\aav = \one$.
\end{rema}

\subsection{Tame reductions}\label{SS:Tame}

 What motivates studying divisions specifically is that the latter satisfy a form of confluence: by Lemma~\ref{L:LocConfDiv} and the results of~\cite{Dit}, if a multifraction~$\aav$ is eligible for a division~$\Rdiv{\ii, \xx}$ and for another reduction~$\Red{\jj, \yy}$, a common reduct for $\aav \act \Rdiv{\ii, \xx}$ and~$\aav \act \Red{\jj, \yy}$ always exists. Moreover, we saw in Proposition~\ref{P:DerDiv} that there always exists a maximal div-reduct with good compatibility properties of the associated operator~$\derdiv$. However, because many prime multifractions are not irreducible, and, in particular, many prime unital multifractions are not trivial, it is hopeless to analyze reduction in terms of divisions exclusively, making it natural to try to extend the family of divisions so as to preserve its main property, namely ``guaranteed confluence''. This leads to tame reductions, and the main result here is that, exactly as in the case of divisions, there exists for each multifraction~$\aav$ and each level~$\ii$ a maximal tame $\ii$-reduction applying to~$\aav$. 

\begin{defi}
If $\aav$ is a multifraction, we say that $\xx$ is an \emph{$\ii$-reducer} for~$\aav$ if $\aav \act \Red{\ii, \xx}$ is defined; we then say that an $\ii$-reducer~$\xx$ is \emph{tame for~$\aav$} if, for all~$\jj, \yy$ such that $\aav \act \Red{\jj, \yy}$ is defined, $\aav \act \Red{\ii, \xx}$ and $\aav \act \Red{\jj, \yy}$ admit a common reduct; otherwise, $\xx$ is called \emph{wild for~$\aav$}.
\end{defi}

Thus $\xx$ is a tame $\ii$-reducer for~$\aav$ if reducing~$\xx$ in~$\aav$ leaves all possibilities for further reductions open, whereas reducing a wild $\ii$-reducer excludes at least one subsequent confluence.

\begin{exam}\label{X:Tame1}
If $\MM$ is a gcd-monoid satisfying the $3$-Ore condition, the system~$\RDb\MM$ is confluent and, therefore, every reducer is tame for every multifraction it applies to. By contrast, in the Artin-Tits of type~$\Att$, for $\aav = \fr{1/c/aba}$, both $\tta$ and~$\ttb$ are $2$-reducers for~$\aav$, but $\aav \act \Red{2, \fr{a}}$ and $\aav \act \Red{2, \fr{b}}$ admit no common reduct, hence $\fr{a}$ and~$\fr{b}$ are wild $2$-reducers for~$\aav$.
\end{exam}

 To prove the existence of the maximal tame $\ii$-reducer in Proposition~\ref{P:MaxTame} below, we shall use convenient characterizations of tame reducers established in Lemmas~\ref{L:Tame1} and~\ref{L:Tame2}. 

\begin{lemm}\label{L:Tame1}
If $\MM$ is a gcd-monoid, $\aav$ is a multifraction on~$\MM$, and $\aav \act \Red{\ii, \xx}$ is defined, then $\xx$ is a tame $\ii$-reducer for~$\aav$ if and only if for $\ii$ positive $($\resp negative$)$ in~$\aav$, the elements~$\xx, \yy$, and $\aa_\ii$ admit a common right $($\resp left$)$ multiple whenever $\aav \act \Red{\ii, \yy}$ is defined.
\end{lemm}

\begin{proof}
Assume that $\xx$ is a tame $\ii$-reducer for~$\aav$, and let~$\yy$ be an $\ii$-reducer for~$\aav$. By definition, $\aav \act \Red{\ii, \xx}$ and $\aav \act \Red{\ii, \yy}$ admit a common reduct, which is necessarily of the form~$\aav \act \Red{\ii, \zz}$ for some~$\zz$. Then there exist~$\uu, \vv$ satisfying 
$$\aav \act \Red{\ii, \zz} = (\aav \act \Red{\ii, \xx}) \act \Red{\ii, \uu} = (\aav \act \Red{\ii, \yy}) \act \Red{\ii, \vv}.$$
Assuming $\ii$ negative in~$\aav$, we deduce $\zz = \xx\uu = \yy\vv$. Hence $\zz$ is a right multiple of~$\xx \lcm \yy$ and, therefore, $\aav \act \Red{\ii, \xx \lcm \yy}$ is defined as well, implying that $\aa_\ii$, $\xx$, and~$\yy$ admit a common right multiple. The argument is symmetric when $\ii$ is positive in~$\aav$. 

Conversely, assume that $\xx$ is an $\ii$-reducer for~$\aav$ and, for every $\ii$-reducer~$\yy$, the elements~$\aa_\ii$, $\xx$, and $\yy$ admit a common multiple, say a common right multiple, assuming that $\ii$ is negative in~$\aav$. Then $\xx \lcm \yy$ left divides~$\aa_{\ii + 1}$ since $\xx$ and~$\yy$ do, and $\aa_\ii$ and $\xx \lcm \yy$ admit a common right multiple. Hence $\aav \act \Red{\ii, \xx \lcm \yy}$ is defined. Then, writing $\xx \lcm \yy = \xx \yy' = \yy \xx'$, we have 
$$\aav \act \Red{\ii, \xx \lcm \yy} = (\aav \act \Red{\ii, \xx}) \act \Red{\ii, \yy'} = (\aav \act \Red{\jj, \yy}) \act \Red{\ii, \xx'},$$
which shows that $\aav \act \Red{\ii, \xx}$ and $\aav \act \Red{\ii, \yy}$ admit a common reduct. On the other hand, for~$\jj \not= \ii$, Lemmas~4.18 and~4.19 from~\cite{Dit} imply that $\aav \act \Red{\ii, \xx}$ and $\aav \act \Red{\jj, \yy}$ always admit a common reduct. Therefore, $\xx$ is a tame $\ii$-reducer for~$\aav$.
\end{proof}

If $\MM$ is a noetherian gcd-monoid, every nonempty family~$\XX$ of left divisors of an element~$\aa$ necessarily admits $\div$-maximal elements, \ie, elements~$\zz$ such that there is no~$\xx$ with $\zz \div \xx$ in the family: take~$\zz$ so that $\zz\inv \aa$ is $\divt$-minimal in $\{\xx\inv \aa \mid \xx \in \XX\}$. Hence, in particular, for every multifraction~$\aav$ and every level~$\ii$, there exist maximal $\ii$-reducers for~$\aav$.

\begin{lemm}\label{L:Tame2}
If $\MM$ is a noetherian gcd-monoid and $\aav$ is a multifraction on~$\MM$, an $\ii$-reducer~$\xx$ for~$\aav$ is tame if and only if $\xx$ divides every maximal $\ii$-reducer for~$\aav$.
\end{lemm}

\begin{proof}
Assume that $\xx$ is a tame $\ii$-reducer for~$\aav$, and $\yy$ is a maximal $\ii$-reducer for~$\aav$. By Lemma~\ref{L:Tame1}, $\aa_\ii$, $\xx$, and~$\yy$ admit a common multiple, hence an lcm, and, therefore, the lcm of~$\xx$ and~$\yy$ is again an $\ii$-reducer for~$\aav$. The assumption that $\yy$ is maximal implies that this lcm is~$\yy$, \ie, that $\xx$ divides~$\yy$ (on the due side).

Conversely, assume that $\xx$ is an $\ii$-reducer for~$\aav$ that divides every maximal $\ii$-reducer. Let $\yy$ be an arbitrary $\ii$-reducer for~$\aav$. As $\MM$ is noetherian, $\yy$ divides at least one maximal $\ii$-reducer, say~$\zz$. By assumption, $\xx$ divides~$\zz$, hence so does the lcm of~$\xx$ and~$\yy$. Since $\zz$ is an $\ii$-reducer for~$\aav$, so is its divisor the lcm of~$\xx$ and~$\yy$ and, therefore, $\aav \act \Red{\ii, \xx}$ and $\aav \act \Red{\ii, \yy}$ admit a common reduct. Hence $\xx$ is tame for~$\aav$.
\end{proof}

\begin{prop}\label{P:MaxTame}
If $\MM$ is a noetherian gcd-monoid and $\aav$ is a multifraction on~$\MM$, then, for every~$\ii < \dh\aav$, there exists a unique greatest tame $\ii$-reducer for~$\aav$, namely the gcd of all maximal $\ii$-reducers for~$\aav$. The latter is a multiple of the gcd of~$\aa_\ii$ and~$\aa_{\ii + 1}$.
\end{prop}

\begin{proof}
Let $\zz$ be the gcd (on the relevant side) of all maximal $\ii$-reducers for~$\aav$. Then $\zz$, and every divisor~$\xx$ of~$\zz$, divides every maximal $\ii$-reducer for~$\aav$, hence, by Lemma~\ref{L:Tame2}, it is a tame $\ii$-reducer for~$\aav$. 

Conversely, if $\xx$ is a tame $\ii$-reducer for~$\aav$, then, by Lemma~\ref{L:Tame2}, $\xx$ divides every maximal $\ii$-reducer for~$\aav$, hence it divides their gcd~$\zz$. Hence $\zz$ is the greatest tame $\ii$-reducer for~$\aav$.

Finally, assume that $\xx$ divides~$\aa_\ii$ and~$\aa_{\ii + 1}$. Then $\aav \act \Rdiv{\ii, \xx}$, hence a fortiori $\aav \act \Red{\ii, \xx}$ is defined, so $\xx$ is an $\ii$-reducer for~$\aav$. Let $\yy$ be any $\ii$-reducer for~$\aav$. Then the lcm of~$\aa_\ii$ and~$\yy$ is a common multiple of~$\aa_\ii$, $\xx$, and~$\yy$. Hence, by Lemma~\ref{L:Tame1}, $\xx$ is a tame $\ii$-reducer for~$\aav$. This applies in particular when $\xx$ is the gcd of~$\aa_\ii$ and~$\aa_{\ii + 1}$.
\end{proof}

 On the shape of what we did with divisions, we introduce 

\begin{defi}\label{D:Dermax}
If $\MM$ is a noetherian gcd-monoid and $\aav$ is a multifraction on~$\MM$, then, for $\ii < \dh\aav$, the unique element~$\xx$ whose existence is stated in Proposition~\ref{P:MaxTame} is called the \emph{greatest tame $\ii$-reducer} for~$\aav$; we then write $\aav \act \Redmax\ii$ for~$\aav \act \Red{\ii, \xx}$. 
\end{defi}

\begin{exam}
When $\MM$ satisfies the $3$-Ore condition, every reducer is tame, and $\aav \act \Redmax\ii$ coincides with the maximal $\ii$-reduct of~$\aav$, as used in~\cite[Section~6]{Dit}. Otherwise, wild reducers may exist and $\aav \act \Redmax\ii$ need not be a maximal $\ii$-reduct of~$\aav$: in the Artin-Tits monoid of type~$\Att$, for $\aav = \fr{1/c/aba}$, we have $\aav \act \Redmax2 = \aav$, since there is no nontrivial tame $2$-reducer.
\end{exam}

By definition, the greatest tame $\ii$-reducer for~$\aav$ only depends on the entries~$\aa_\ii$ and~$\aa_{\ii + 1}$, and on the sign of~$\ii$ in~$\aav$. By Lemma~\ref{L:Tame2}, it can be computed easily as a gcd of maximal reducers. Note that the greatest tame $\ii$-reducer for~$\aav$ may be strictly larger than the gcd of~$\aa_\ii$ and~$\aa_{\ii + 1}$: for instance, in the Artin-Tits monoid of type~$\Att$, there exist two maximal $2$-reducers for~$\aav := \fr{1/a/cabab}$, namely $\fr{caa}$ and~$\fr{cab}$, both wild, and the greatest tame $2$-reducer is their left gcd~$\fr{ca}$, a proper multiple of the left gcd of~$\fr{a}$ and~$\fr{cabab}$, which is~$1$.

If $\MM$ is a noetherian gcd-monoid, starting from an arbitrary multifraction~$\aav$ and repeatedly performing (maximal) tame reductions leads in finitely many steps to a $\simeqb$-equivalent multifraction that is \emph{tame-irreducible}, meaning eligible for no tame reduction. By Proposition~\ref{P:MaxTame}, a division is always tame, so a tame-irreducible multifraction is prime. Adapting the proof of Proposition~\ref{P:NCPrime} yields:

\begin{prop}\label{P:NCTame}
If $\MM$ is a noetherian gcd-monoid, then $\RDb\MM$ $($\resp $\RDp\MM$$)$ is semi-convergent if and only if, for every~$\aav$ in~$\FRb\MM$ $($\resp $\FRp\MM)$,
\begin{equation}\label{E:NCTame}
\parbox{110mm}{If $\aav$ is unital and tame-irreducible, then $\aav$ is either trivial or reducible.}
\end{equation}
\end{prop}

\begin{coro}\label{C:NCTame}
Conjecture~$\ConjA$ is true if and only if \eqref{E:NCTame} holds for every Artin-Tits monoid~$\MM$ and every~$\aav$ in~$\FRp\MM$.
\end{coro}

The above results suggest to investigate tame-irreducible multifractions more closely. By definition, only wild reductions may apply to a tame-irreducible multifraction. A possible approach for establishing~\eqref{E:NCTame} could be to study the irreducible reducts of tame-irreducible multifractions. It happens frequently that, if $\aav$ is tame-irreducible and admits several (wild) reducts~$\aav_1 \wdots \aav_\mm$, then the reducts of the various~$\aav_\kk$s are pairwise disjoint, as if every such reduct kept a trace of~$\aav_\kk$. If true, such a property might lead to a proof of Conjecture~$\ConjA$ using Corollary~\ref{C:NCTame} and arguments similar to those alluded to in the proof of Proposition~\ref{P:MaxSteps}. But the assumption is not readily correct. Indeed, always in the Artin-Tits monoid of type~$\Att$, the $6$-multifraction $\aav := \fr{1/c/aba/bc/a/bcb}$ is tame-irreducible, it admits four wild reducers, namely $\fr{a}$ and~$\fr{b}$ at level~$2$, and $\fr{b}$ and~$\fr{c}$ at level~$5$, and the associated $2$-reducts of~$\aav$ admit a common reduct
$$(\aav \act \Red{2, \fr{a}}) \act \Red{5, \fr{c}}\Red{3,\fr{b}} \Red{1,\fr{ac}} \Red{2,\fr{b}} \Red{3,\fr{c}} = \fr{bc/accb/ca/ab/ca/cb} = (\aav \act \Red{2,\fr{b}}) \act \Red{5,\fr{c}} \Red{3,\fr{cb}}.$$
However, in this example, there is no confluence for the $5$-reducts, and, more generally, we have no example where the highest level wild reducts of a tame-irreducible multifraction admit a common reduct.

\subsection{The operator~$\redtame$ and Conjecture~$\ConjB$}\label{SS:Dertame}

 Our main claim in this section is that, for every multifraction~$\aav$, there exists a distinguished tame reduct of~$\aav$, denoted~$\redtame(\aav)$, that can be computed easily, and that should be~$\one$ whenever $\aav$ is unital: this is what we call Conjecture~$\ConjB$.

Just mimicking the approach of Section~\ref{SS:Div} and trying to identify a unique maximal tame reduct on the shape of~$\derdiv\aav$ cannot work, because the tame reducts of a multifraction need not admit a common reduct: in the context of type~$\Att$ again, $\fr{a}$ and~$\fr{b}$ are tame $4$-reducers for~$\aav := \fr{1/c/1/1/aba}$, but $\aav \act \Red{4, \fr{a}}$ and $\aav \act \Red{4, \fr{b}}$ admit no common reduct (by the way, in this case, the two irreducible reducts of~$\aav$ can be reached using tame reductions only, respectively $\Red{4,\fr{a}} \Red{2,\fr{a}} \Red{4,\fr{ba}}$ and $\Red{4,\fr{b}} \Red{2,\fr{b}} \Red{4,\fr{ab}}$. 



 However, it is shown in~\cite[Section~6]{Dit} that, if $\MM$ is a noetherian gcd-monoid satisfying the $3$-Ore condition, hence in a case when all reductions are tame, there exists, for each~$\nn$, a universal sequence of levels~$\univ\nn$ such that, starting with any $\nn$-multifraction~$\aav$ and applying the maximal (tame) reduction at the successive levels prescribed by~$\univ\nn$ inevitably leads to the unique $\RRR$-irreducible reduct~$\red(\aav)$ of~$\aav$. It is then natural to copy the recipe in the general case, and to introduce:
 
\begin{defi}
If $\MM$ is a noetherian gcd-monoid, then, for every depth~$\nn$ multifraction~$\aav$ on~$\MM$, we put $\redtame(\aav) := \aav \act \Redmax{\univ\nn}$, where $\univ\nn$ is empty for $\nn = 0,1$ and is $(1, 2 \wdots \nn - 1)$ followed by~$\univ{\nn - 2}$ for $\nn \ge 2$, and, for~$\underline\ii = (\ii_1 \wdots \ii_\ell)$, we write $\aav \act \Redmax{\underline\ii}$ for $\aav \act \Redmax{\ii_1} \pdots \Redmax{\ii_\ell}$.
\end{defi}

Thus, by~\cite[Proposition~6.7]{Dit}, if $\MM$ is a noetherian gcd-monoid satisfying the $3$-Ore condition, $\red(\aav) = \redtame(\aav)$ holds for every~$\aav$ in~$\FRp\MM$. In this case, $\bbv \rds \redtame(\aav)$ holds for every reduct~$\bbv$ of~$\aav$, and $\redtame(\aav)$ is always $\RRR$-irreducible. It is easy to see that, in the general case, these properties do not extend to arbitrary (namely, not necessarily unital) multifractions. 

\begin{exam}\label{X:Tame2}
In the Artin-Tits monoid of type~$\Att$, let $\aav := \fr{1/c/aba}$. Then $\univ3 =\nobreak (1, 2)$ leads to $\redtame(\aav) = \aav \act \Redmax1 \Redmax2 = \aav$, since $1$ and $\fr{c}$ have no nontrivial common divisor, and there is no tame $2$-reducer for~$\aav$. On the other hand, both~$\fr{a}$ and~$\fr{b}$ are $2$-reducers for~$\aav$, and neither $\aav \act \Red{2, \fr{a}} \rds \redtame(\aav)$ nor $\aav \act \Red{2, \fr{b}} \rds \redtame(\aav)$ holds.

Next, let $\bbv:= \fr{ac/aca/aba}$. We find $\redtame(\bbv) = \bbv \act \Rdiv{1, \fr{ac}} = \fr{1/c/aba}$, to be compared with $\derdiv\bbv = \bbv \act \Rdiv{2, \fr{a}} = \fr{ac/ca/ba}$. Then $\derdiv\bbv \rds \redtame(\bbv)$ fails, so $\bbv \rddivs \bbv'$ does not imply $\bbv' \rds \redtame(\bbv)$.

Finally, let $\ccv := \fr{\1/c/aba/cb}$. We find $\redtame(\ccv) = \ccv \act \Rdiv{3,\fr{b}} = \fr{\1/c/ba/c}$, which is not irreducible, nor even tame-irreducible: we have $\redtame^2(\ccv) = \redtame(\ccv) \act \Red{2, \fr{b}}\Red{3, \fr{c}} = \fr{bc/cb/a/c}$.
\end{exam}

However, these negative facts say nothing about tame reductions starting from a unital multifraction, and, in spite of many tries, no counter-example was ever found so far to: 

\begin{conjB}
If $\MM$ is an Artin-Tits monoid, then $\redtame(\aav) = \one$ holds for every unital multifraction~$\aav$ in~$\FRp\MM$.
\end{conjB}

By definition, $\aav \rds \redtame(\aav)$ holds, so $\redtame(\aav) = \one$ is a strengthening of~$\aav \rds \one$ in which we assert not only that $\aav$ reduces to~$\one$ but also that it goes to~$\one$ in some prescribed way. Thus:

\begin{fact}
Conjecture~$\ConjB$ implies Conjecture~$\ConjA$.
\end{fact} 

Although Conjecture~$\ConjB$ is more demanding than Conjecture~$\ConjA$, it might be easier to establish (or to contradict), as it predicts a definite equality rather than an existential statement. As recalled above, \cite[Proposition~6.7]{Dit} implies that Conjecture~$\ConjB$ is true for every Artin-Tits monoid of FC~type. On the other hand, an example of a gcd-monoid (but not an Artin--Tits one) for which (the counterpart of) Conjecture~$\ConjA$ is true but (that of) Conjecture~$\ConjB$ is false is given in~\cite{Div}.

\section{Cross-confluence}\label{S:CrossConf}

 Besides tame reductions and Conjecture~$\ConjB$ of Section~\ref{S:Tame}, we now develop another approach to Conjecture~$\ConjA$, involving both the reduction system~$\RDp{}$ and a symmetric counterpart~$\RDtp{}$ of~$\RDp{}$. The properties of reduction and its counterpart are just symmetric, but interesting features appear when both are used simultaneously, in particular what we call \emph{cross-confluence}, a completely novel property to the best of our knowledge. We are then led to a new statement, Conjecture~$\ConjC$, which would imply Conjecture~$\ConjA$ and, from there, the decidability of the word problem.

The section comprises four subsections. First, right reduction, the symmetric counterpart of (left) reduction, is introduced in Subsection~\ref{SS:RightRed}, and its basic properties are established. Next, cross-conflence, which combines reduction and its counterpart, is introduced in Subsection~\ref{SS:CrossConf}, and partial results are established. Then Conjecture~$\ConjC$ and its uniform version~$\ConjCunif$ are stated and discussed in Subsection~\ref{SS:UnifCC}. Finally, we briefly study in an Appendix the termination of the joint system obtained by merging left and right reduction, a natural topic with nontrivial results, but not directly connected so far to our main conjectures.

\subsection{Right reduction}\label{SS:RightRed}

 By definition, the reduction rule~$\Red{\ii, \xx}$ of Definition~\ref{D:Red} consists in pushing a factor~$\xx$ to the left (small index entries) in the multifraction it is applied to: for this reason, we shall call it a \emph{left} reduction. From now on, we shall also consider symmetric counterparts, naturally called \emph{right} reductions, where elements are pushed to the right.

\begin{defi}\label{D:RightRed}
If $\MM$ is a gcd-monoid and $\aav, \bbv$ lie in~$\FRb\MM$, then, for $\ii \ge 1$ and $\xx$ in~$\MM$, we declare that $\bbv = \aav \act \Redt{\ii, \xx}$ holds if we have $\dh\bbv = \dh\aav$, $\bb_\kk = \aa_\kk$ for $\kk \not= \ii - 1, \ii, \ii + 1$, and there exists~$\xx'$ (necessarily unique) satisfying
$$\begin{array}{lccc}
\text{for $\ii < \dh\aav$ positive in~$\aav$:\quad}
&\xx \bb_{\ii - 1} = \aa_{\ii - 1}, 
&\xx \bb_\ii = \aa_\ii \xx' = \xx \lcm \aa_\ii, 
&\bb_{\ii + 1} = \aa_{\ii + 1} \xx',\\
\text{for $\ii < \dh\aav$ negative in~$\aav$:\quad}
&\bb_{\ii - 1} \xx = \aa_{\ii - 1}, 
&\bb_\ii \xx = \xx' \aa_\ii = \xx \lcmt \aa_\ii, 
&\bb_{\ii + 1} = \xx' \aa_{\ii + 1},\\
\text{for $\ii = \dh\aav$ positive in~$\aav$:\quad}
&\xx \bb_{\ii - 1} = \aa_{\ii - 1}, 
&\xx \bb_\ii = \aa_\ii,\\
\text{for $\ii = \dh\aav$ negative in~$\aav$:\quad}
&\bb_{\ii - 1} \xx = \aa_{\ii - 1}, 
&\bb_\ii \xx = \aa_\ii.
\end{array}$$
 We write $\aav \rdt \bbv$ if $\aav \act \Redt{\ii, \xx}$ holds for some~$\ii$ and some $\xx \not= 1$, and use $\rdts$ for the reflexive--transitive closure of~$\rdt$. The rewrite system~$\RDtb\MM$ so obtained on~$\FRb\MM$ is called \emph{right reduction}, and its restriction to~$\FRp\MM$ (positive multifractions) is denoted by~$\RDtp\MM$. 
\end{defi}

The action of~$\Redt{\ii, \xx}$ is symmetric of that of~$\Red{\ii, \xx}$: one extracts~$\xx$ from~$\aa_{\ii-1}$, lets it cross~$\aa_\ii$ using an lcm, and incorporates the resulting remainder in~$\aa_{\ii+1}$, thus carrying~$\xx$ from level~$\ii-1$ to level~$\ii+1$, see Figure~\ref{F:RightRed}. As in the case of~$\Red{1, \xx}$, the action of~$\Redt{\nn, \xx}$ for $\nn = \dh\aav$ is adapted to avoid creating a $(\nn+1)$st entry.

\begin{figure}[htb]
\begin{picture}(95,18)(0,-2)
\psset{nodesep=0.7mm}
\put(-7,6){$...$}
\put(-15,12){$\Red{\ii, \xx}:$}
\psline[style=back,linecolor=color2]{c-c}(0,6)(15,0)(30,0)(45,6)
\psline[style=back,linecolor=color1]{c-c}(0,6)(15,12)(30,12)(45,6)
\pcline{->}(0,6)(15,12)\taput{$\aa_{\ii - 1}$}
\pcline{<-}(15,12)(30,12)\taput{$\aa_\ii$}
\pcline{->}(30,12)(45,6)\taput{$\aa_{\ii + 1}$}
\pcline{->}(0,6)(15,0)\tbput{$\bb_{\ii - 1}$}
\pcline{<-}(15,0)(30,0)\tbput{$\bb_\ii$}
\pcline{->}(30,0)(45,6)\tbput{$\bb_{\ii + 1}$}
\pcline[linewidth=1.5pt,linecolor=color3,arrowsize=1.5mm]{->}(30,12)(30,0)\trput{$\xx$}
\pcline{->}(15,12)(15,0)\tlput{$\xx'$}
\psarc[style=thin](15,0){3}{0}{90}
\put(21,5){$\Leftarrow$}
\put(51,6){$...$}
\psline[style=back,linecolor=color2]{c-c}(60,6)(75,0)(90,0)(105,6)
\psline[style=back,linecolor=color1]{c-c}(60,6)(75,12)(90,12)(105,6)
\pcline{<-}(60,6)(75,12)\taput{$\aa_{\ii - 1}$}
\pcline{->}(75,12)(90,12)\taput{$\aa_\ii$}
\pcline{<-}(90,12)(105,6)\taput{$\aa_{\ii + 1}$}
\pcline{<-}(60,6)(75,0)\tbput{$\bb_{\ii - 1}$}
\pcline{->}(75,0)(90,0)\tbput{$\bb_\ii$}
\pcline{<-}(90,0)(105,6)\tbput{$\bb_{\ii + 1}$}
\pcline[linewidth=1.5pt,linecolor=color3,arrowsize=1.5mm]{<-}(90,12)(90,0)\trput{$\xx$}
\pcline{<-}(75,12)(75,0)\tlput{$\xx'$}
\psarc[style=thin](75,0){3}{0}{90}
\put(81,5){$\Leftarrow$}
\put(109,6){$...$}
\end{picture}

\begin{picture}(95,22)(0,-2)
\psset{nodesep=0.7mm}
\put(-7,6){$...$}
\put(-15,12){$\Redt{\ii, \xx}:$}
\psline[style=back,linecolor=color2]{c-c}(0,6)(15,0)(30,0)(45,6)
\psline[style=back,linecolor=color1]{c-c}(0,6)(15,12)(30,12)(45,6)
\pcline{->}(0,6)(15,12)\taput{$\aa_{\ii - 1}$}
\pcline{<-}(15,12)(30,12)\taput{$\aa_\ii$}
\pcline{->}(30,12)(45,6)\taput{$\aa_{\ii + 1}$}
\pcline{->}(0,6)(15,0)\tbput{$\bb_{\ii - 1}$}
\pcline{<-}(15,0)(30,0)\tbput{$\bb_\ii$}
\pcline{->}(30,0)(45,6)\tbput{$\bb_{\ii + 1}$}
\pcline[linewidth=1.5pt,linecolor=color3,arrowsize=1.5mm]{<-}(15,12)(15,0)\tlput{$\xx$}
\pcline{<-}(30,12)(30,0)\trput{$\xx'$}
\psarc[style=thin](30,0){3}{90}{180}
\put(21,5){$\Rightarrow$}
\put(51,6){$...$}
\psline[style=back,linecolor=color2]{c-c}(60,6)(75,0)(90,0)(105,6)
\psline[style=back,linecolor=color1]{c-c}(60,6)(75,12)(90,12)(105,6)
\pcline{<-}(60,6)(75,12)\taput{$\aa_{\ii - 1}$}
\pcline{->}(75,12)(90,12)\taput{$\aa_\ii$}
\pcline{<-}(90,12)(105,6)\taput{$\aa_{\ii + 1}$}
\pcline{<-}(60,6)(75,0)\tbput{$\bb_{\ii - 1}$}
\pcline{->}(75,0)(90,0)\tbput{$\bb_\ii$}
\pcline{<-}(90,0)(105,6)\tbput{$\bb_{\ii + 1}$}
\pcline[linewidth=1.5pt,linecolor=color3,arrowsize=1.5mm]{->}(75,12)(75,0)\tlput{$\xx$}
\pcline{->}(90,12)(90,0)\trput{$\xx'$}
\psarc[style=thin](90,0){3}{90}{180}
\put(81,5){$\Rightarrow$}
\put(109,6){$...$}
\end{picture}
\caption{\small Comparing the left reduction~$\Red{\ii, \xx}$ and the right reduction~$\Redt{\ii, \xx}$: in the first case (top), one pushes the factor~$\xx$ from~$\aa_{\ii+1}$ to~$\aa_{\ii-1}$ through~$\aa_\ii$, in the second case (bottom), one pushes~$\xx$ from~$\aa_{\ii-1}$ to~$\aa_{\ii+1}$ through~$\aa_\ii$.}
\label{F:RightRed}
\end{figure}
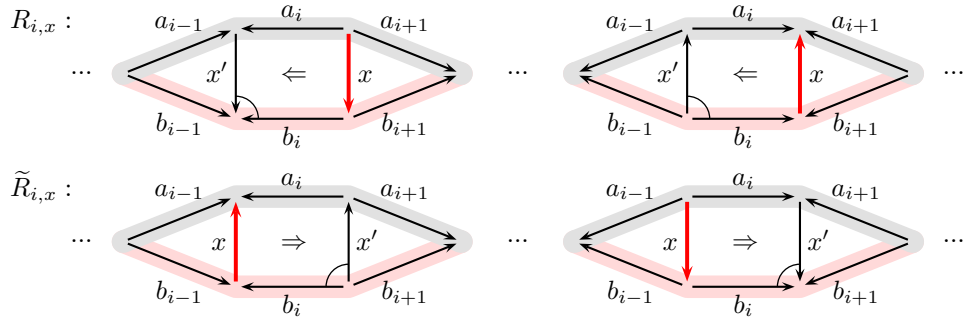

\begin{rema}\label{R:NotInverse}
Right reduction is \emph{not} an inverse of left reduction: when the reduced factor~$\xx$ crosses (in one direction or the other) the entry~$\aa_\ii$, using the lcm operation cancels common factors. Typically, if $\xx$ divides~$\aa_\ii$, then both $\Red{\ii, \xx}$ and~$\Redt{\ii, \xx}$ amount to dividing by~$\xx$ and, therefore, their actions coincide. See Remark~\ref{R:Inverse} for more on this.
\end{rema}

 The rest of this subsection is devoted to the basic properties of right reduction and their connection with those of left reduction, in particular with respect to convergence and semi-convergence. As can be expected, the convenient tool is an operation exchanging left and right reduction, in this case the duality map~$\REV{\ }$ of Notation~\ref{N:REV}. 

\begin{lemm}\label{L:LRDual}
If $\MM$ is a gcd-monoid, then, for all~$\aav, \bbv$ in~$\FRb\MM$, 
\begin{equation}\label{E:LRDual}
\aav \rd \bbv \text{ is equivalent to }\REV\aav \rdt \REV\bbv.
\end{equation}
\end{lemm}

\begin{proof}
Comparing the definitions shows that, if $\aav$ and~$\bbv$ have depth~$\nn$, then $\bbv = \aav \act \Red{\ii, \xx}$ is equivalent to $\REV\bbv = \REV\aav \act \Redt{\nn + 1 - \ii, \xx}$, whence~\eqref{E:LRDual}.
\end{proof}

 The following properties of right reduction follow almost directly from their counterpart involving left reduction. The only point requiring some care is the lack of involutivity of\HS{1.2}$\REV{\ }$, itself resulting from the lack of surjectivity of the map~$\aav \mapsto 1 \opp \aav$. 

\begin{lemm}\label{L:RBasics}
Assume that $\MM$ is a gcd-monoid. 

\ITEM1 The relation~$\rdts$ is included in~$\simeqb$, \ie, $\aav \rdts \bbv$ implies $\aav \simeqb \bbv$.

\ITEM2 The relation~$\rdts$ is compatible with the multiplication of~$\FRb\MM$.

\ITEM3 For all~$\aav, \bbv$ and~$\pp, \qq$, the relation~$\aav \rdts \bbv$ is equivalent to $\One\pp \opp \aav \opp \One\qq \rdts \One\pp \opp \bbv \opp \One\qq$.

\ITEM4 If $\aav, \bbv$ belong to~$\FRp\MM$, then $\aav \rd \bbv$ is equivalent to $1 \opp \REV\aav \rdt 1 \opp \REV\bbv$, and $\aav \rdt \bbv$ is equivalent to $1 \opp \REV\aav \rd 1 \opp \REV\bbv$.
\end{lemm}

\begin{proof}
\ITEM1 By~\eqref{E:LRDual}, $\aav \rdts \bbv$ implies $\REV\aav \rds \REV\bbv$, whence $\can(\REV\aav) = \can(\REV\bbv)$ by Lemma~\ref{L:Basics}\ITEM1, hence $\can(\aav) = \can(\bbv)$ by Lemma~\ref{L:Inverse}, \ie, $\aav \simeqb \bbv$. 

\ITEM2 For all~$\ccv, \ddv$, the relation $\aav \rdts \bbv$ implies $\REV\aav \rds \REV\bbv$, whence $\REV\ddv \opp \REV\aav \opp \REV\ccv \rds \REV\ddv \opp \REV\bbv \opp \REV\ccv$ by Lemma~\ref{L:Basics}\ITEM2, which is $(\ccv \opp \aav \opp \ddv)\REV{\ } \rds (\ccv \opp \bbv \opp \ddv)\REV{\ }$ by Lemma~\ref{L:Inverse}, and finally $\ccv \opp \aav \opp \ddv \rds \ccv \opp \bbv \opp \ddv$ by~\eqref{E:LRDual} again.

\ITEM3 Using duality as above, the result follows now from Lemma~\ref{L:Basics}\ITEM3.

\ITEM4 By~\eqref{E:LRDual}, $\aav \rd \bbv$ is equivalent to $\REV\aav \rdt \REV\bbv$, hence, by~\ITEM3, it is also equivalent to $1 \opp \aav \rd 1 \opp \bbv$. Similarly, by~\eqref{E:LRDual} again, $\aav \rdt \bbv$ is equivalent to $\REV\aav \rd \REV\bbv$, hence, by Lemma~\ref{L:Basics}, it is also equivalent to $1 \opp \REV\aav \rd 1 \opp \REV\bbv$. 
\end{proof}
 
By~\eqref{E:LRDual}, an infinite sequence of right reductions would provide an infinite sequence of left reductions, and vice versa, so $\RDtb\MM$ is terminating if and only if $\RDb\MM$ is. Comparing irreducible elements is straightforward:

\begin{lemm}\label{L:LRIrred}
If $\MM$ is a gcd-monoid, a multifraction~$\aav$ is $\RRR$-irreducible if and only if $\REV\aav$ is $\RRRt$-irreducible. For~$\aav$ positive with~$\dh\aav$ even, $\aav$ is $\RRRt$-irreducible if and only if $\REV\aav$ is $\RRR$-irreducible.
\end{lemm}

\begin{proof}
Assume $\REV\aav \rdt \bbv$. By~\eqref{E:LRDual}, we deduce $\aav = \REV{\REV\aav} \rd \REV\bbv$, hence $\aav$ is not $\RRR$-irreducible. So $\aav$ being $\RRR$-irreducible implies that $\REV\aav$ is $\RRRt$-irreducible. Conversely, assume $\REV\aav \rd \bbv$. By~\eqref{E:LRDual}, we deduce $\aav = \REV{\REV\aav} \rdt \REV\bbv$, hence $\aav$ is not $\RRRt$-irreducible. So $\aav$ being $\RRRt$-irreducible implies that $\REV\aav$ is $\RRR$-irreducible. For the second equivalence, use the involutivity of\HS{1.2}$\REV{\ }$ on positive multifractions of even depth (but we claim nothing for positive multifractions of odd length). 
\end{proof}

 Using the above technical results, we can compare convergence and semi-convergence for left and right reduction. Below, observe the difference between~\ITEM1 and~\ITEM2, which involve only one direction but multifractions of both signs, and~\ITEM3, which only involves positive multifractions and requires using both left and right reductions. This distinction is one of the reasons for considering both positive and negative multifractions in this paper (contrary to~\cite{Dit}). 

\begin{prop}
For every gcd-monoid~$\MM$, the following are equivalent:

\ITEM1 The system~$\RDb\MM$ is convergent (\resp semi-convergent);

\ITEM2 The system~$\RDtb\MM$ is convergent (\resp semi-convergent);

\ITEM3 The systems~$\RDp\MM$ and $\RDtp\MM$ are convergent (\resp semi-convergent).
\end{prop}

\begin{proof}
We begin with convergence. Assume that $\RDb\MM$ is convergent. Let $\aav$ belong to~$\FRb\MM$. Let $\bbv = \red(\REV\aav)$. By definition, we have $\REV\aav \rds \bbv$. Then, by Lemmas~\ref{L:RBasics} and~\ref{L:LRIrred}, we have $\aav \rdts \REV\bbv$ and $\REV\bbv$ is $\RRRt$-irreducible. Assume that $\ccv$ is $\RRRt$-irreducible and $\aav \rdts \ccv$ holds. By Lemmas~\ref{L:RBasics} and~\ref{L:LRIrred} again, we deduce $\REV\aav \rds \REV\ccv$ and $\REV\ccv$ is $\RRR$-irreducible. The assumption that $\RDb\MM$ is convergent implies $\REV\ccv = \bbv$, whence $\ccv = \REV\bbv$. Hence $\REV\bbv$ is the only $\RRRt$-irreducible $\RRRt$-reduct of~$\aav$. Therefore, $\RDtb\MM$ is convergent, and \ITEM1 implies~\ITEM2. The converse implication is proved in the same way, so \ITEM1 is equivalent to~\ITEM2.

Since all $\RRR$-reducts and $\RRRt$-reducts of a positive multifraction are positive, it is clear that \ITEM1 implies that $\RDp\MM$ is convergent, and \ITEM2 implies that $\RDtp\MM$ is convergent. Conversely, assume that both $\RDp\MM$ and $\RDtp\MM$ are convergent. Let $\aav$ be an arbitrary multifraction on~$\MM$. Assume first that $\aav$ is positive. Let $\bbv$ be the unique $\RRR$-irreducible reduct of~$\aav$. Assume now that $\aav$ is negative. Then $\REV\aav$ is positive. Let $\bbv$ be the unique $\RRRt$-irreducible $\RRRt$-reduct of~$\REV\aav$. By Lemmas~\ref{L:RBasics} and~\ref{L:LRIrred}, $\REV\bbv$ is $\RRR$-irreducible, and $\aav \rds \bbv$ holds. Now assume that $\ccv$ is $\RRR$-irreducible and $\aav \rds \ccv$ holds. Then $\REV\ccv$ is $\RRR$-irreducible and $\REV\aav \rdts \REV\ccv$ holds. As $\REV\aav$ is positive and $\RDtp\MM$ is convergent, we deduce $\REV\ccv = \bbv$, whence $\ccv = \REV\bbv$. Hence $\RDb\MM$ is convergent, and \ITEM3 implies~\ITEM1. This completes the argument for convergence.

Assume now that $\RDb\MM$ is semi-convergent. Let $\aav$ be a unital multifraction. By Lem\-ma~\ref{L:Basics}, $\REV\aav$ is unital as well, hence we must have $\REV\aav \rds \one$, which, by~\eqref{E:LRDual}, implies $\aav \rdts \REV\one = \one$. Hence $\RDtb\MM$ is semi-convergent. The converse implication is similar, so \ITEM1 and~\ITEM2 are equivalent in this case as well. On the other hand, \ITEM1 and \ITEM2 clearly imply~\ITEM3. Finally, assume that both $\RDp\MM$ and~$\RDtp\MM$ are semi-convergent. Let $\aav$ be a unital multifraction. If $\aav$ is positive, the assumption that $\RDp\MM$ is semi-convergent implies $\aav \rds \one$. If $\aav$ is negative, then $\REV\aav$ is positive, and the assumption that $\RDtp\MM$ is semi-convergent implies $\REV\aav \rdts \one$, whence $\aav \rds \REV\one = \one$ by~\eqref{E:LRDual}. Hence $\RDb\MM$ is semi-convergent. So \ITEM3 implies~\ITEM1, which completes the argument for semi-convergence.
\end{proof}

In the convergent case, the above proof implies, with obvious notation,
$\redt(\aav) = (\red(\REV\aav))\REV{\ }$.

\begin{rema}
It is shown in~\cite[Sec.~3]{Dit} that, when left reduction is considered, trimming final trivial entries essentially does not change reduction. This is not true for right reduction: deleting trivial final entries can change the reducts, as deleting trivial initial entries does in the case of left reduction.
\end{rema}

\subsection{ The cross-confluence property}\label{SS:CrossConf}

We now introduce our main new notion, a variant of confluence that combines left and right reduction.

\begin{defi}\label{D:CrossConf}
If $\MM$ is a gcd-monoid, we say that $\RDp\MM$ is \emph{cross-confluent} if, for all~$\aav, \bbv, \ccv$ in~$\FRp\MM$, 
\begin{equation}\label{E:CrossConf}
\parbox{113mm}{If we have $\aav \rdts \bbv$ and $\aav \rdts \ccv$, there exists~$\ddv$ satisfying $\bbv \rds \ddv$ and $\ccv \rds \ddv$.}
\end{equation}
\end{defi} 

 So cross-confluence for~$\RDp\MM$ means that left reduction provides a solution for the confluence pairs of right reduction. We shall naturally say that $\RDb\MM$ is cross-confluent if~\eqref{E:CrossConf} holds for all~$\aav, \bbv, \ccv$ in~$\FRb\MM$. On the other hand, we say that $\RDtp\MM$ is \emph{cross-confluent} if right reduction provides a solution for the confluence pairs of left reduction, that is, if
\begin{equation}\label{E:CrossConft}
\parbox{113mm}{If we have $\aav \rds \bbv$ and $\aav \rds \ccv$, there exists~$\ddv$ satisfying $\bbv \rdts \ddv$ and $\ccv \rdts \ddv$.}
\end{equation}
holds for all~$\aav, \bbv, \ccv$ in~$\FRp\MM$. Finally, $\RDtb\MM$ is cross-confluent if \eqref{E:CrossConft} holds for all~$\aav, \bbv, \ccv$ in~$\FRb\MM$.

\begin{rema}
 The definition of cross-confluence involves both left and right reduction. But, owing to the equivalence~\eqref{E:LRDual}, it can alternatively be stated as a property involving left reduction exclusively. Indeed, in the case of~$\RDb\MM$, cross-confluence is equivalent to 
\begin{equation}\label{E:CrossConf1}
\parbox{113mm}{If we have $\REV\aav \rds \REV\bbv$ and $\REV\aav \rds \REV\ccv$, there exists~$\ddv$ satisfying $\bbv \rds \ddv$ and $\ccv \rds \ddv$,}
\end{equation}
and similarly for~$\RDtb\MM$. In the case of~$\RDp\MM$, when we restrict to positive multifractions, we also have an alternative statement involving~$\rds$ exclusively, but it takes the less symmetric form
\begin{equation}\label{E:CrossConf1p}
\parbox{113mm}{If we have $1 \opp \REV\aav \rds 1 \opp \REV\bbv$ and $1 \opp \REV\aav \rds 1 \opp \REV\ccv$, \\\null\hfill there exists~$\ddv$ satisfying $\bbv \rds \ddv$ and $\ccv \rds \ddv$,}
\end{equation}
 because the $\REV{\ }$ operation is not involutive in this case. 
\end{rema}

 Using duality, we obtain for the cross-confluences of~$\RDb{}$ and~$\RDtb{}$ and that of their positive versions the same connection as in the case of convergence and semi-convergence: 

\begin{prop}\label{P:CCAlt}
For every gcd-monoid~$\MM$, the following are equivalent:

\ITEM1 The system~$\RDb\MM$ is cross-confluent;

\ITEM2 The system~$\RDtb\MM$ is cross-confluent;

\ITEM3 The systems~$\RDp\MM$ and $\RDtp\MM$ are cross-confluent.
\end{prop}

\begin{proof}
Assume that $\RDb\MM$ is cross-confluent. Applying~\eqref{E:CrossConf} to~$\REV\aav$, $\REV\bbv$, and~$\REV\ccv$ and using~\eqref{E:LRDual}, we deduce~\eqref{E:CrossConft}, so $\RDtb\MM$ is also cross-confluent. Doing the same from~\eqref{E:CrossConft} returns to~\eqref{E:CrossConf}. So \ITEM1 and \ITEM2 are equivalent.

Next, all $\RRR$- and $\RRRt$-reducts of a positive multifraction are positive, so, if $\RDb\MM$ is cross-confluent, its retriction to~$\FRp\MM$ is also cross-confluent. Similarly, if $\RDtb\MM$ is cross-confluent, its retriction to~$\FRp\MM$ is cross-confluent. Hence \ITEM1 implies~\ITEM3. Finally, assume that $\RDp\MM$ and $\RDtp\MM$ are cross-confluent, and $\aav$, $\bbv$, $\ccv$ belong to~$\FRb\MM$ and satisfy $\aav \rdts \bbv$ and $\aav \rdts \ccv$. If $\aav$, hence $\bbv$ and~$\ccv$ as well, are positive, the assumption that $\RDp\MM$ is cross-confluent implies the existence of~$\ddv$ satisfying~\eqref{E:CrossConf}. Assume now that $\aav$, hence $\bbv$ and~$\ccv$, are negative. Then $\REV\aav$, $\REV\bbv$, $\REV\ccv$ are positive, and we have $\REV\aav \rds \REV\bbv$ and $\REV\aav \rds \REV\ccv$. The assumption that $\RDtp\MM$ is cross-confluent implies the existence of~$\ddv$ satisfying $\REV\bbv \rdts \ddv$ and $\REV\ccv \rdts \ddv$. By~\eqref{E:LRDual}, we deduce $\bbv \rds \REV\ddv$ and $\ccv \rds \REV\ddv$. So \eqref{E:CrossConf} holds, $\RDb\MM$ is cross-confluent, and \ITEM3 implies~\ITEM1.
\end{proof}

 In view of our main purpose, namely establishing the semi-convergence of (left) reduction, the main result is the following connection, which locates cross-confluence as an intermediate between convergence and semi-convergence:

\begin{prop}\label{P:ConvCC}
Assume that $\MM$ is a noetherian gcd-monoid.

\ITEM1 If $\RDb\MM$ is convergent, then $\RDb\MM$ is cross-confluent.

\ITEM2 If $\RDb\MM$ is cross-confluent, then $\RDb\MM$ is semi-convergent.

\ITEM3 Mutatis mutandis, the same implications hold for~$\RDp\MM$.
\end{prop}

We begin with an auxiliary result:

\begin{lemm}\label{L:RedOne}
Assume that $\MM$ is a gcd-monoid.

\ITEM1 If $\RDb\MM$ is cross-confluent, then, for every~$\aav$ in~$\FRb\MM$, the relations $\aav \rds \one$ and $\aav \rdts \one$ are equivalent. 

\ITEM2 If $\RDp\MM$ is cross-confluent, then, for every~$\aav$ in~$\FRp\MM$, the relation $\aav \rdts \one$ implies $\aav \rds \one$.
\end{lemm}

\begin{proof}
\ITEM1 Assume $\aav \in \FRb\MM$ and $\aav \rds \one$. By Proposition~\ref{P:CCAlt}, the cross-confluence of~$\RDb\MM$ implies that of~$\RDtb\MM$. By definition, we also have $\aav \rds \aav$. Then \eqref{E:CrossConft} implies the existence of~$\ddv$ satisfying $\one \rdts \ddv$ and $\aav \rdts \ddv$. By definition, $\one$ is $\RRRt$-irreducible, so $\one \rdts \ddv$ implies $\ddv = \one$, whence $\aav \rdts \one$. Conversely, assume $\aav \rdts \one$. We have $\aav \rdts \aav$, and \eqref{E:CrossConf} implies the existence of~$\ddv$ satisfying $\aav \rds \ddv$ and $\one \rds \ddv$, whence $\ddv = \one$, and $\aav \rds \one$.

\ITEM2 When $\aav$ lies in~$\FRp\MM$, the latter argument remains valid, and it shows again that $\aav \rdts \one$ implies $\aav \rds \one$. (By contrast, the former argument need not extend, as there is a priori no reason why the cross-confluence of~$\RDp\MM$ should imply that of~$\RDtp\MM$.)
\end{proof}

\begin{proof}[Proof of Proposition~\ref{P:ConvCC}]
\ITEM1 Let $\aav$ belong to~$\FRb\MM$, and assume $\aav \rdts \bbv$ and $\aav \rdts \ccv$. By Lemma~\ref{L:RBasics}, we have $\aav \simeqb \bbv \simeqb \ccv$, whence $\bbv \rds \ddv$ and $\ccv \rds \ddv$, where $\ddv$ is the (unique) $\RRR$-irreducible reduct of~$\aav$. Hence \eqref{E:CrossConf} is satisfied. 

\ITEM2 Assume that $\RDb\MM$ is cross-confluent, and we have $\aav \rds \one$ and $\aav \rds \bbv$ for some~$\aav$ in~$\FRb\MM$. By~\eqref{E:CrossConft}, which is valid since, by Proposition~\ref{P:CCAlt}, the cross-confluence of~$\RDb\MM$ implies that of~$\RDtb\MM$, there exists~$\ddv$ satisfying $\bbv \rdts \ddv$ and $\one \rdts \ddv$. Since $\one$ is $\RRRt$-irreducible, we must have $\ddv = \one$, whence $\bbv \rdts \one$. By Lemma~\ref{L:RedOne}, we deduce $\bbv \rds \one$. Hence $\RDb\MM$ is $\one$-confluent and, therefore, by Proposition~\ref{P:1Conf}, it is semi-convergent.

\ITEM3 For~\ITEM1, the argument is the same in the case $\aav \in \FRp\MM$. For~\ITEM2, assume that $\RDp\MM$ is cross-confluent, and we have $\aav \rds \one$ and $\aav \rds \bbv$ for some~$\aav$ in~$\FRp\MM$. By Lemma~\ref{L:RBasics}\ITEM4, we have $1 \opp \REV\aav \rdts 1 \opp \one = \one$ and $1 \opp \aav \rdts 1 \opp \bbv$, and $1 \opp \REV\aav$ is positive. As $\RDp\MM$ is cross-confluent, we deduce the existence of~$\ddv$ satisfying $\one \rds \ddv$ and $1 \opp \REV\bbv \rds \ddv$. As $\one$ is $\RRR$-irreducible, we have $\ddv = \one$, whence $1 \opp \REV\bbv \rds \one$, and, by Lemma~\ref{L:RBasics}\ITEM4 again, $\bbv \rdts \one$. By Lemma~\ref{L:RedOne}\ITEM2, we deduce $\bbv \rds \one$. Hence $\RDp\MM$ is $\one$-confluent and, by Proposition~\ref{P:1Conf}, it is semi-convergent.
\end{proof}

\subsection{Conjectures~$\ConjC$ and~$\ConjCunif$}\label{SS:UnifCC}

 We thus arrive at what we think is the main conjecture in this paper:

\begin{conjC}
For every Artin-Tits monoid~$\MM$, the system~$\RDp\MM$ is cross-confluent.
\end{conjC}

 By Proposition~\ref{P:ConvCC}, Conjecture~$\ConjC$ implies Conjecture~$\ConjA$, whence the decidability of the word problem of the group, and it is true for every Artin--Tits monoid of type~FC. Note that Conjecture~$\ConjC$ is different from Conjectures~$\ConjA$ and~$\ConjB$ in that it predicts something for all multifractions, not only for unital ones. So, in a sense, it is a more structural property, which we think is interesting independently of any application. 

No proof of Conjecture~$\ConjC$ is in view so far in the general case, but we now observe that \emph{local} cross-confluence, namely cross-confluence with single reduction steps on the left, is always true.

\begin{prop}\label{P:CrossOne}
If $\MM$ is a gcd-monoid, then, for all~$\aav, \bbv, \ccv$ in~$\FRb\MM$,
\begin{gather}
\label{E:CrossOne1}
\parbox{113mm}{If we have $\aav \rdt \bbv$ and $\aav \rdt \ccv$, there exists~$\ddv$ satisfying $\bbv \rds \ddv$ and $\ccv \rds \ddv$.}\\
\label{E:CrossOne2}
\parbox{113mm}{If we have $\aav \rd \bbv$ and $\aav \rd \ccv$, there exists~$\ddv$ satisfying $\bbv \rdts \ddv$ and $\ccv \rdts \ddv$.}
\end{gather}
\end{prop}

 The proof relies on the following preparatory result:

\begin{lemm}\label{L:LeftRightDiv}
Assume that $\MM$ is a gcd-monoid and $\aav$ is a multifraction on~$\MM$ such that $\aav \act \Red{\ii, \xx}$ is defined. If $\ii$ is negative $($\resp positive$)$ in~$\aav$, let $\xx'$ and~$\xxh$ be defined by $\aa_\ii \xx' = \aa_\ii \lcm \xx$ and $\xxh = \aa_\ii \gcd \xx$ $($\resp $\xx' \aa_\ii = \aa_\ii \lcmt \xx$ and $\xxh = \aa_\ii \gcdt \xx$$)$. Then, we have
\begin{equation}
\label{E:LeftRightDiv}
\aav \act \Red{\ii, \xx} \Redt{\ii, \xx'} = \aav \act \Rdiv{\ii, \xxh}.
\end{equation}
\end{lemm}

\begin{proof}
(Figure~\ref{F:LRRed}) Let $\aav' = \aav \act \Red{\ii, \xx}$. Assuming $\ii$ negative in~$\aav$, we have 
\begin{equation}\label{E:LeftRightDiv1}
\aa'_{\ii - 1} = \aa_{\ii - 1} \xx', \quad \xx \aa'_\ii = \aa_\ii \xx' = \aa_\ii \lcm \xx, \quad \xx \aa'_{\ii + 1} = \aa_{\ii + 1}.
\end{equation}
By construction, $\xx'$ right divides~$\aa'_{\ii - 1}$, and $\xx'$ and~$\aa'_\ii$ admit a common left multiple, namely~$\aa_\ii \xx'$. Hence $\aav'' = \aav' \act \Redt{\ii, \xx'}$ is defined, and it is determined by
\begin{equation}\label{E:LeftRightDiv2}
\aa''_{\ii - 1} \xx' = \aa'_{\ii - 1}, \quad \aa''_\ii \xx' = \xx'' \aa'_\ii = \aa'_\ii \lcmt \xx', \quad \aa''_{\ii + 1} = \xx'' \aa'_{\ii + 1}.
\end{equation}
By definition, the left lcm $\xx'' \aa'_\ii$ of~$\xx'$ and~$\aa'_\ii$ left divides their common left multiple~$\xx \aa'_\ii$, which implies the existence of~$\xxh$ satisfying $\xx = \xxh \xx''$, and, from there, $\aa_\ii \xx' = \xx \aa'_\ii = \xxh \xx'' \aa'_\ii = \xxh \aa''_\ii \xx'$, whence $\aa_\ii = \xxh \aa''_\ii$. Merging~\eqref{E:LeftRightDiv1} and~\eqref{E:LeftRightDiv2}, we deduce $\aa''_{\ii - 1} = \aa_{\ii - 1}$, $\xxh \aa''_\ii = \aa_\ii$, and $\xxh \aa''_{\ii + 1} = \aa_{\ii + 1}$, which shows that $\aav''$ is obtained from~$\aav$ by left dividing the $\ii$th and $(\ii + 1)$st entries by~$\xxh$, \ie, by applying the division~$\Rdiv{\ii, \xxh}$. 

The argument is symmetric when $\ii$ is positive in~$\aav$.
\end{proof}

\begin{figure}[htb]
\begin{picture}(55,16)(0,0)
\psset{unit=1.1mm}
\psline[style=back,linecolor=color2]{c-c}(0,6.5)(15,13.5)(30,9)(35,0)(50,7)
\psline[style=back,linecolor=color1]{c-c}(0,7)(15,14)(35,14)(50,7)
\psset{nodesep=0.7mm}
\pcline{->}(0,7)(15,14)\taput{$\aa_{\ii - 1}$}
\pcline{<-}(15,14)(35,14)\taput{$\aa_\ii$}
\pcline{->}(35,14)(50,7)\taput{$\aa_{\ii + 1}$}
\pcline{->}(0,6)(15,0)\tbput{$\aa'_{\ii - 1}$}
\pcline{<-}(15,0)(35,0)\tbput{$\aa'_\ii$}
\pcline{->}(35,0)(50,7)\tbput{$\aa'_{\ii + 1}$}
\pcline{->}(35,14)(35,0)\trput{$\xx$}
\pcline{->}(15,14)(15,0)\tlput{$\xx'$}
\pcline{->}(30,9)(35,0)\nbput[npos=0.5]{$\xx''$}
\pcline{<-}(15,14)(30,9)\tbput{$\aa''_\ii$}
\pcline{->}(35,14)(30,9)\put(32,12){$\xxh$}
\psarc[style=thin](15,0){3}{0}{90}
\psarc[style=thin](30,9){3}{165}{300}
\end{picture}
\caption{\small Composing the left reduction $\Red{\ii, \xx}$ and the inverse right reduction $\Redt{\ii, \xx'}$ amounts to performing the division $\Rdiv{\ii, \xxh}$ where $\xxh$ is the gcd of~$\aa_\ii$ and~$\xx$.}
\label{F:LRRed}
\end{figure}

When $\aav \act \Redt{\ii, \xx}$ is defined, a symmetric argument gives
\begin{equation}\label{E:RightLeftDiv}
\aav \act \Redt{\ii, \xx} \Red{\ii, \xx'} = \aav \act \Rdiv{\ii - 1, \xxh},
\end{equation}
where $\xx'$ and~$\xxh$ are now defined by $\xx' \aa_\ii = \aa_\ii \lcmt \xx$ and $\xxh = \aa_\ii \gcdt \xx$ (\resp $\aa_\ii \xx' = \aa_\ii \lcm \xx$ and $\xxh = \aa_\ii \gcd \xx$) if $\ii$ is negative (\resp positive) in~$\aav$. The index of the division is shifted ($\ii - 1$ instead of~$\ii$) relatively to~\eqref{E:LeftRightDiv}, because, as a set of pairs, $\Rdiv{\ii, \xx}$ is included in~$\Redt{\ii+1,\xx}$.

 We can now complete the argument for Proposition~\ref{P:CrossOne}. 

\begin{proof}[Proof of Proposition~\ref{P:CrossOne}]
Assume $\bbv = \aav \act \Redt{\ii, \xx}$ and $\ccv = \aav \act \Redt{\jj, \yy}$. By Lemma~\ref{L:LeftRightDiv}, we have
$$\bbv \act \Red{\ii, \xx'} = \aav \act \Rdiv{\ii, \xxh} \quad\text{and}\quad \ccv' \act \Red{\jj, \yy'} = \aav \act \Rdiv{\jj, \yyh}$$
for some~$\xx', \xxh$ and $\yy', \yyh$. By Proposition~\ref{P:DerDiv}, $\derdiv\aav$ is a common reduct of all multifractions obtained from~$\aav$ by a division, hence in particular of $\aav \act \Rdiv{\ii, \xx''}$ and~$\aav \act \Rdiv{\jj, \yy''}$. Thus $\bbv \rds \derdiv\aav$ and $\ccv \rds \derdiv\bbv$ hold, hence \eqref{E:CrossOne1} is satisfied with $\ddv = \derdiv\aav$.

The proof of~\eqref{E:CrossOne2} is symmetric, using \eqref{E:RightLeftDiv} instead of~\eqref{E:LeftRightDiv}, and $\ddv = \derdiv\aav$ again.
\end{proof}

 In the case of a single rewrite system, local confluence implies confluence whenever the system is terminating (by Berman's well known diamond lemma). There is no hope of a similar result here, both because the union of~$\RDp{}$ and~$\RDtp{}$ is not terminating in general---see \eqref{X:NonTermin} below---and because, in the definition of cross-confluence, the arrows~$\rd$ and~$\rdt$ are not in a position for a natural induction. 

\begin{rema}\label{R:Inverse}
Right reduction is close to being an inverse of left reduction. Indeed, provided the ground monoid is noetherian, every reduction is a product of atomic reductions, namely reductions of the form~$\Red{\ii, \xx}$ or~$\Redt{\ii, \xx}$ with~$\xx$ an atom. Now, if $\xx$ is an atom, the gcd of~$\xx$ and~$\aa_\ii$ is either~$\aa_\ii$, meaning that $\xx$ divides~$\aa_\ii$, or~$1$. In the former case, $\aav \act \Red{\ii, \xx}$ is~$\aav \act \Rdiv{\ii, \xx}$, whereas, in the latter, Lemma~\ref{L:LeftRightDiv} implies $\aav = (\aav \act \Red{\ii, \xx}) \act \Redt{\ii, \xx'}$, \ie, left reducing~$\xx$ in~$\aav$ is the inverse of right reducing~$\xx'$. Thus, writing $\RDiv{}$ for the family of divisions and $\RDatp{}$ for that of atomic left reductions, $\RDp{}$ is generated by~$\RDiv{} \cup \RDatp{}$, whereas $\RDtp{}$ is generated by~° $\RDiv{} \cup \RDatp{}\inv$. By Lemma~\ref{L:LocConfDiv} and the results of~\cite{Dit}, confluence between $\RDiv{}$ and $\RDatp{}$ is always true, whereas confluence between~$\RDatp{}$ and $\RDatp{}\inv$ is trivial. Therefore, one might hope that cross-confluence diagrams can always be constructed by assembling the various types of elementary confluence diamonds. This is not true: using a tedious case-by-case argument, one can indeed establish cross-confluence in the case when, in~\eqref{E:CrossConf}, $\bbv$ and~$\ccv$ are obtained from~$\aav$ by two atomic reduction steps, but there is no hope to go very far in this direction, both because of the counter-examples of Example~\ref{X:PseudoDiv}, and because, in any case, cross-confluence cannot be true for an arbitrary noetherian gcd-monoid, since there exist such monoids for which the counterparts of Conjectures~$\ConjA$ and~$\ConjC$ fail~\cite{Div}: if true, cross-confluence has to be a specific property of Artin--Tits monoids, or at least of a restricted family of gcd-monoids. 
\end{rema}

 The following examples are given to show that naive attempts to extend the local cross-confluence result of Proposition~\ref{P:CrossOne} are due to fail. 

\begin{exam}\label{X:PseudoDiv}
Proposition~\ref{P:CrossOne} shows that, if $\bbv$ is obtained by applying one right reductions to~$\aav$, then applying one well chosen left reduction to~$\bbv$ provides a multifraction~$\ccv$ obtained by one division from~$\aav$. The result fails when $1$ is replaced by~$\kk \ge 2$. Indeed, in the Artin-Tits monoid of type~$\Att$, consider $\aav = \fr{\1/a/ca/cb/b}$ and $\bbv = \aav \act \Redt{3,\fr{a}} \Redt{5,\fr{b}} = \fr{\1/\1/ca/cb/\1}$, (which is $\RRRt$-irreducible). The only way to left reduce~$\bbv$ is to start with $\Red{1,\fr{ca}}$, leading to~$\ccv = \fr{ca}/1/\fr{cb}/1$, not reachable from~$\aav$ by two, or any number, of divisions.

In the above case, we have $\ccv = \aav \act \Red{4,\fr{b}} \Red{2, \fr{cac}}$, and therefore there is no contradiction with the weaker conclusion that $\ccv$ is obtained both from~$\aav$ by applying $\kk$~left reductions. The following example (with $\kk = 3$) shows that this is not true either. Indeed, consider $\aav = \fr{ca/cb/bc/ba}$ and $\bbv = \aav \act \Redt{3,\fr{bc}} \Redt{2,\fr{ca}} \Redt{4,\fr{a}} = \fr{\1/\1/ac/ab}$. As predicted by Conjecture~$\ConjC$, $\aav$ and $\bbv$ admit common left reducts, but the latter are~$\ccv$ and $\ccv \act \Rdiv{2,\fr{c}}$, with $\ccv = \fr{ac/cab/c/\1} = \aav \act \Red{2,\fr{b}} \Red{3,\fr{a}} \Rdiv{2, \fr{a}} \Red{3, \fr{b}} \Rdiv{1,\fr{aa}} \Rdiv{2, \fr{b}}$, not reachable from~$\aav$ using less than six left reductions. What is surprising here is that, if we put $\bbv' := \aav \act \Redt{3,\fr{bc}} \Redt{2,\fr{ca}} = \fr{\1/\1/aca/aba}$, then $\bbv'$ left reduces to~$\aav$, whereas $\bbv = \bbv' \act \Rdiv{3,\fr{a}}$ only left reduces to~$\ccv$, very far from~$\aav$: one single division may change left reducts completely.
\end{exam}

 We conclude with one more conjecture. The conjunction of Propositions~\ref{P:DerDiv} and~\ref{P:CrossOne} shows not only that any two right reducts~$\bbv, \ccv$ of a multifraction~$\aav$ admit a common left reduct~$\ddv$, but even that there exists~$\ddv$ only depending on~$\aav$, namely~$\derdiv\aav$, that witnesses for all elementary right reducts of~$\aav$ simultaneously. This suggests to consider a strong version of cross-confluence:

\begin{defi}\label{D:StrongCrossConf}
If $\MM$ is a gcd-monoid, we say that $\RDp\MM$ is \emph{uniformly cross-confluent} if there exists a map~$\nabla$ from~$\FRp\MM$ to itself such that, for every~$\aav$ in~$\FRp\MM$, the relation $\bbv \rds \nabla\aav$ holds for every right reduct~$\bbv$ of~$\aav$.
\end{defi} 

 In the case when reduction is convergent, defining $\nabla \aav = \red(\aav)$ provides a convenient witness, and therefore the conclusion of Proposition~\ref{P:ConvCC} can be strengthened to uniform cross-confluence. We propose:

\begin{conjCunif}
For every Artin-Tits monoid~$\MM$, the system~$\RDp\MM$ is uniformly cross-confluent.
\end{conjCunif}

 By the above observation, Conjecture~$\ConjCunif$ is true for every Artin--Tits monoid of type~FC, and no counter-example could be found so far in any other type. It implies Conjecture~$\ConjC$ and, therefore, Conjecture~$\ConjA$, but it is more demanding. However, if an explicit definition of~$\nabla\aav$ could be found, one can reasonably hope that the proof of Conjecture~$\ConjCunif$ would then reduce to a series of verifications. But, here again, naive attempts fail: two natural candidates for~$\nabla\aav$ could be either~$\redtame(\aav)$ (which works in the convergent case), or (if it always exists) a maximal common ancestor of all irreducible reducts in the tree of all left reducts of~$\aav$, but the example of Figure~\ref{F:CrossConf} shows that neither of these choices works in every case.

\begin{figure}[htb]
\begin{picture}(105,85)(0,-2)
\psset{nodesep=0.8mm}
\psset{framearc=.5}
\psset{linewidth=0.4pt}
\psset{labelsep=0pt}
\psset{border=2pt}
\rput[c](93,55){\rnode{1}{\Vertexxx1{a/bac/bb/aca}}}
\rput[c](60,34){\rnode{2}{\Vertexx2{a/ac/b/aca}}}
\rput[c](83,27){\rnode{3}{\Vertex3{a/bcbac/ccb/ca}}}
\rput[c](105,20){\rnode{4}{\Vertex4{a/babac/aab/ac}}}
\rput[c](16,15){\rnode{5}{\Vertexx5{/cbac/ccb/ca}}}
\rput[c](60,12){\rnode{6}{\Vertex6{a/bcac/cb/ca}}}
\rput[c](30,0){\rnode{7}{\Vertexx7{/bac/cb/ca}}}
\rput[c](100,0){\rnode{8}{\Vertex8{a/baac/ab/ac}}}
\rput[c](15,75){\rnode{9}{\vertex9{/cbac/bcbb/aca}}}
\rput[c](0,48){\rnode{10}{\vertex{10}{/bac/bcb/aca}}}
\rput[c](62,56){\rnode{11}{\vertex{11}{a/c/ba/acaab}}}
\rput[c](50,72){\rnode{12}{\vertex{12}{/ac/caba/acaab}}}
\rput[c](28,58){\rnode{13}{\vertex{13}{/ac/cba/acaa}}}
\rput[c](30,41){\rnode{14}{\vertex{14}{/ac/cb/aca}}}
\nccurve[style=double,angleA=180,angleB=0,ncurv=0.3]{->}{3}{5}\nbput[npos=0.55]{$\scriptstyle\Rdiv{1,\tta}$}
\ncline[style=double]{->}{3}{6}\trput{$\scriptstyle\Rdiv{2,\ttc}$}
\ncline[style=double]{->}{1}{2}\tlput{$\scriptstyle\Rdiv{2,\ttb}$}
\ncline{->}{1}{3}\tlput{$\scriptstyle\Red{3,\ttc}$}
\ncline{->}{1}{4}\trput{$\scriptstyle\Red{3,\tta}$}
\nccurve[angleA=250,angleB=110,ncurv=0.6]{->}{2}{6}\nbput[npos=0.2]{$\scriptstyle\Red{3,\ttc}$}
\nccurve[angleA=280,angleB=135,ncurv=0.6]{->}{2}{8}\naput[npos=0.7]{$\scriptstyle\Red{3,\tta}$}
\ncline[style=double]{->}{4}{8}\trput{$\scriptstyle\Rdiv{2,\tta}$}
\ncline[style=double]{->}{5}{7}\naput[npos=0.6]{$\scriptstyle\Rdiv{2,\ttc}$}
\nccurve[style=double,angleA=245,angleB=0,ncurv=1]{->}{6}{7}\naput[npos=0.4]{$\scriptstyle\Rdiv{1,\tta}$}
\nccurve[style=exist,angleA=100,angleB=10,ncurv=0.6]{->}{1}{9}\naput[npos=0.3]{$\scriptstyle\Redt{2,\tta}$}
\nccurve[angleA=15,angleB=90,ncurv=0.6]{->}{9}{1}\naput[npos=0.75]{$\scriptstyle\Red{2,\ttb\ttc}$}
\nccurve[style=double,angleA=190,angleB=170,ncurv=0.8]{->}{9}{5}\nbput[npos=0.3]{$\scriptstyle\Rdiv{3,\ttc}$}
\ncline[style=double]{->}{9}{10}\nbput[npos=0.5]{$\scriptstyle\Rdiv{2,\ttc}$}
\nccurve[style=double,angleA=260,angleB=180,ncurv=1]{->}{10}{7}\naput[npos=0.2]{$\scriptstyle\Rdiv{3,\ttc}$}
\ncarc[style=exist]{->}{2}{11}\naput[npos=0.5]{$\scriptstyle\Redt{3,\tta}$}
\ncarc[]{->}{11}{2}\naput[npos=0.4]{$\scriptstyle\Red{3,\tta\ttb}$}
\ncarc[style=exist]{->}{11}{12}\naput[npos=0.5]{$\scriptstyle\Redt{2,\tta}$}
\ncarc[]{->}{12}{11}\naput[npos=0.5]{$\scriptstyle\Red{2,\ttc\tta}$}
\ncline[style=double]{->}{12}{13}\nbput[npos=0.5]{$\scriptstyle\Rdiv{3,\ttb}$}
\ncline[style=double]{->}{10}{14}\taput{$\scriptstyle\Rdiv{2,\ttb}$}
\ncline[style=double]{->}{13}{14}\naput[npos=0.5]{$\scriptstyle\Rdiv{3,\tta}$}
\ncline[style=exist]{->}2{14}\nbput[npos=0.5]{$\scriptstyle\Redt{2,\tta}$}
\nccurve[style=exist,angleA=45,angleB=270,ncurv=1]{->}{7}{14}\naput[npos=0.7]{$\scriptstyle\Redt{3,\ttb}$}
\nccurve[angleA=280,angleB=40,ncurv=1]{->}{14}{7}\naput[npos=0.3]{$\scriptstyle\Red{3,\ttc}$}
\nccurve[angleA=335,angleB=175,ncurv=1]{->}{14}{2}\nbput[npos=0.5]{$\scriptstyle\Red{2,\ttc}$}
\end{picture}
\caption{\small Left and right reducts of $\aav_1 := \fr{a/bac/bb/aca}$ in the Artin-Tits monoid of type~$\Att$: there are $8$~left reducts (in grey), among which $\aav_7$ and~$\aav_8$ are $\RDp{}$-irreducible, and $10$~right reducts (dashed lines), among which $\aav_{14}$ is $\RDtp{}$-irreducible. Plain (\resp dashed) arrows correspond to left (\resp right) reductions that are not divisions, double arrows correspond to divisions (which are both left and right reductions); all left reductions decrease the distance to the bottom. As Conjecture~$\ConjCunif$ predicts, there exists a common left reduct for all right reducts: in this case, there is only one, namely~$\aav_7$, and it is neither $\redtame(\aav_1)$ nor the maximal common ancestor of~$\aav_7$ and~$\aav_8$, both equal to~$\aav_2$.}
\label{F:CrossConf}
\end{figure}

\subsection*{Appendix: Mixed termination}\label{SS:MixedTermin}

 Although it is not directly connected with cross-confluence, we mention here one further result involving both left and right reduction. First, we know that, in a noetherian context, (left) reduction is terminating, meaning that there is no infinite sequence of reductions. It turns out that a stronger finiteness result holds:

\begin{prop}\label{P:FinitelyMany}
If $\MM$ is a finitely generated noetherian gcd-monoid, then every multifraction~$\aav$ on~$\MM$ admits only finitely many left reducts, and finitely many right reducts.
\end{prop}

\begin{proof}
As $\MM$ is noetherian and contains no nontrivial invertible element, a subfamily of~$\MM$ is generating if and only if it contains all atoms. Hence, the assumption that $\MM$ is finitely generated implies that the atom set~$\AA$ of~$\MM$ is finite. Let $\aav$ belong to~$\FRb\MM$. We construct a tree~$T_{\aav}$, whose nodes are pairs~$(\bbv, \ss)$, where $\bbv$ is a left reduct of~$\aav$ and $\ss$ is a finite sequence in~$\AA \times \NNNN$: the root of~$T_{\aav}$ is~$(\aav, \ew)$, and, using ${}^\frown$ for concatenation, the sons of~$(\bbv, \ss)$ are all pairs $(\bbv \act \Red{\ii, \xx}, \ss {}^\frown (\ii, \xx))$ such that $\bbv \act \Red{\ii, \xx}$ is defined. As $\AA$ is finite, for every multifraction~$\bbv$, the number of pairs~$(\ii, \xx)$ with $\xx$ in~$\AA$ and $\bbv \act \Red{\ii, \xx}$ defined is finite. Hence each node in~$T_{\aav}$ has finitely many immediate successors. On the other hand, the assumption that $\MM$ is noetherian implies that $\RDb\MM$ is terminating and, therefore, the tree~$T_{\aav}$ has no infinite branch. Hence, by K\" onig's lemma, $T_{\aav}$ is finite. As every left reduct of~$\aav$ appears (maybe more than once) in~$T_{\aav}$, the number of such reducts is finite.

The argument for right reducts is symmetric.
\end{proof}

 Thus, it makes sense to wonder whether, starting from a multifraction~$\aav$, the family of all multifractions that can be obtained from~$\aav$ using left and right reduction is finite. The argument for Proposition~\ref{P:FinitelyMany} does not extend, because the well-orders witnessing for the termination of left and right reductions are not the same, and it is easy to see that the result itself fails in general: starting from $\aav := \fr{\1/a/bc/\1}$ in the Artin-Tits of type~$\Att$, we find $\aav \act \Red{2,\fr{b}}\Redt{3,\fr{a}} = \fr{ba/b/ca/ac}$, whence, repeating three times,
\begin{equation}\label{X:NonTermin}
\aav \act \Red{2,\fr{b}} \Redt{3, \fr{a}} \Red{2,\fr{c}} \Redt{3,\fr{b}} \Red{2,\fr{a}} \Redt{3,\fr{c}} = \fr{bacbac/a/bc/acbacb} = \fr{bacbac} \opp \aav \opp \fr{acbacb}.
\end{equation}
Hence the multifractions $\aav \act (\Red{2,\fr{b}} \Redt{3, \fr{a}} \Red{2,\fr{c}} \Redt{3,\fr{b}} \Red{2,\fr{a}} \Redt{3,\fr{c}})^\pp$ make for $\pp \ge 0$ an infinite non-terminating (and non-periodic) sequence with respect to $\RRR \cup \RRRt$.

By contrast, let us mention without detailed proof a finiteness result valid whenever the ground monoid~$\MM$ is a Garside monoid~\cite{Dgk, Dir}, \ie, a strongly noetherian gcd-monoid possessing in addition an element~$\Delta$ (``Garside element'') whose left and right divisors coincide, generate~$\MM$, and are finite in number.

\begin{prop}
If $\MM$ is a Garside monoid, then, for every multifraction~$\aav$ on~$\MM$, the family of all $(\RRR \cup \RRRt)$-reducts of~$\aav$ is finite.
\end{prop}

\begin{proof}[Proof (sketch)]
Let $\Delta$ be a Garside element in~$\MM$ and let $\aav$ be a multifraction on~$\MM$. Then there exists a positive integer~$\dd$ such that the path associated with~$\aav$ can be drawn in the finite fragment of the Cayley graph of~$\MM$ made of the divisors of~$\Delta^\dd$: this is the notion of a path ``drawn in $\Div(\Delta^\dd)$'' as considered in~\cite{Dfo}. Then the family of all paths drawn in~$\Div(\Delta^\dd)$ is closed under the special transformations alluded to in Subsection~\ref{SS:SemiAppli}, and, therefore, all multifractions that can be derived from~$\aav$ using~$\rd$ and~$\rdt$ are drawn in the same finite fragment $\Div(\Delta^\dd)$ of the Cayley graph. As we consider multifractions with a fixed depth, only finitely many of them can be drawn in a finite fragment of a Cayley graph.
\end{proof}

The argument extends to every Artin-Tits monoid~$\MM$ of type~FC, replacing the finite family of divisors of the Garside element~$\Delta$ with the union of the finitely many finite families of divisors of the Garside elements~$\Delta_\II$, where $\II$ is a family of atoms of~$\MM$ that generates a spherical type submonoid of~$\MM$.

\section{Finite approximations}\label{S:SmallDepth}

The semi-convergence of~$\RDb\MM$ and Conjectures~$\ConjA$,~$\ConjB$, and~$\ConjC$, involve multifractions of arbitrary depth. Further results appear in the particular case of small depth multifractions. The cases of depth~$2$ and, more interestingly, of depth~$4$ are addressed in Subsections~\ref{SS:Semi2} and~\ref{SS:Semi4}, where connections with the embeddability in the group and the uniqueness of fractional decompositions, respectively, are established. An application of the latter to partial orderings of the group is established in Subsection~\ref{SS:Ordering}. Finally, we describe in Subsection~\ref{SS:Semi6} a connection between Conjecture~$\ConjB$ and van Kampen diagrams for unital $\nn$-multifractions.

\subsection{The $\nn$-semi-convergence property}\label{SS:Semi2}

The rewrite system~$\RDb\MM$ has been called semi-convergent if \eqref{E:Semi} holds for every multifraction on~$\MM$, \ie, if $\aav$ being unital implies $\aav \rds \one$.

\begin{defi}\label{D:Semin}
If $\MM$ is a gcd-monoid, we say that $\RDb\MM$ $($\resp $\RDp\MM$$)$ is \emph{$\nn$-semi-convergent} if~\eqref{E:Semi} holds for every $\nn$-multifraction~$\aav$ in~$\FRb\MM$ (\resp $\FRp\MM$).
\end{defi}

Accordingly, we shall use Conjecture~$\ConjA_\nn$ for the restriction of Conjecture~$\ConjA$ to depth~$\nn$ multifractions, and similarly for~$\ConjB$ and~$\ConjC$. Some easy connections exist. Of course, if $\RDb\MM$ is $\nn$-semi-convergent, then so is its subsystem~$\RDp\MM$.

\begin{lemm}\label{L:SemiDown}
If $\MM$ is a gcd-monoid and $\RDb\MM$ $($\resp $\RDp\MM$$)$ is $\nn$-semi-convergent, it is $\pp$-semi-convergent for~$\pp < \nn$.
\end{lemm}

\begin{proof}
Let $\aav$ be a nontrivial unital $\pp$-multifraction, with $\pp < \nn$. There exists~$\rr$ (equal to $\nn -\pp$ or $\nn - \pp + 1$) such that $\aav \opp \One\rr$ has width~$\nn$, and it is also nontrivial and unital. As $\RDb\MM$ is $\nn$-semi-convergent, we have $\aav \opp \One\rr \rds \one$, which implies $\aav \rds \one$ by Lemma~\ref{L:Basics}\ITEM3. So $\RDb\MM$ is $\pp$-semi-convergent.
\end{proof}

On the other hand, by repeating the proof of Proposition~\ref{P:NC}, we obtain

\begin{lemm}\label{L:NCGlobal}
If $\MM$ is a noetherian gcd-monoid, then $\RDb\MM$ $($\resp $\RDp\MM$$)$ is $\nn$-semi-convergent if and only if \eqref{E:NC} holds for every $\nn$-multifraction~$\aav$ in~$\FRb\MM$ $($\resp $\FRp\MM$$)$, \ie, if $\aav$ is unital, then it is either trivial or reducible.
\end{lemm}

We now address the cases of small depth. The case of depth one is essentially trivial, in that it follows from a sufficiently strong form of noetherinity and does not really involve the algebraic properties of the monoid:

\begin{prop}
Assume that $\MM$ is a gcd-monoid that admits a length function, namely a map~$\wit : \MM \to \NNNN$ satisfying, for all~$\aa, \bb$ in~$\MM$,
\begin{equation}\label{E:SuperStrWit}
\wit(\aa\bb) = \wit(\aa) + \wit(\bb), \quad \text{and}\quad \wit(\aa) > 0 \text{\ for $\aa \not= 1$}.
\end{equation}
Then the system~$\RDb\MM$ is $1$-semi-convergent.
\end{prop}

Condition~\eqref{E:SuperStrWit} is strong noetherianity~\eqref{E:StrWit} with~$\ge$ replaced by~$=$. It holds in every Artin-Tits monoid and, more generally, in every monoid with a homogeneous presentation.

\begin{proof}
Extend the map~$\wit$ to~$\FRb\MM$ by $\wit(\aav):= \sum_{\text{$\ii$ positive in~$\aav$}} \wit(\aa_\ii) - \sum_{\text{$\ii$ negative in~$\aav$}} \wit(\aa_\ii)$. Then $\wit$ is a homomorphism from the monoid~$\FRb\MM$ to~$(\NNNN, +)$, and, for every~$\aa$ in~$\MM$, we have $\wit(\aa/\aa) = \wit(/\aa/\aa) = \wit(1) = \wit(\ef) = 0$. By Proposition~\ref{P:EnvGroup}, the latter pairs generate~$\simeqb$ as a congruence, hence $\wit$ is invariant under~$\simeqb$. Then $\aa \not= 1$ implies $\wit(\aa) \not= \wit(1)$, whence $\aa \not\simeqb 1$. Hence the only unital $1$-multifraction is~$1$, and $\RDb\MM$ is $1$-semi-convergent.
\end{proof}

The cases of depths~$2$ and~$3$ turn out to be directly connected with the embeddability of the considered monoid in its enveloping group.

\begin{prop}\label{P:Semi2}
If $\MM$ is a gcd-monoid, the following are equivalent:

\ITEM1 The system~$\RDp\MM$ is $2$-semi-convergent.

\ITEM2 The system~$\RDb\MM$ is $2$-semi-convergent.

\ITEM3 The system~$\RDb\MM$ is $3$-semi-convergent.

\ITEM4 The monoid~$\MM$ embeds in~$\EG\MM$.
\end{prop}

\begin{proof}
Assume that $\RDp\MM$ is $2$-semi-convergent. Let $\aa, \bb$ two elements of~$\MM$ satisfying $\can(\aa) = \can(\bb)$. By~\eqref{E:Eval}, we have $\aa / \bb \simeqb 1$. The assumption that $\RDp\MM$ is $2$-semi-convergent implies $\aa / \bb \rds \one$. By definition, this means that there exists~$\xx$ in~$\MM$ satisfying $\aa = \xx = \bb$, whence $\aa = \bb$. So $\MM$ embeds in~$\EG\MM$. Hence \ITEM1 implies~\ITEM4.

Clearly \ITEM2 implies~\ITEM1, and \ITEM3 implies~\ITEM2 by Lemma~\ref{L:SemiDown}. 

Finally, assume that $\MM$ embeds in~$\EG\MM$, and let $\aav$ be a positive nontrivial unital $3$-multifraction. By assumption, we have $\aav \simeqb 1$, whence $\can(\aa_1) \can(\aa_2)\inv \can(\aa_3) = 1$ in~$\EG\MM$ by~\eqref{E:Eval}, and, therefore, $\can(\aa_2) = \can(\aa_3)\can(\aa_1) = \can(\aa_3 \aa_1)$ in~$\EG\MM$. As $\MM$ embeds in~$\EG\MM$, this implies $\aa_2 = \aa_3 \aa_1$ in~$\MM$. Then $\aav$ has the form $\aa_1 / \aa_3 \aa_1 / \aa_3$, implying $\aav \act \Rdiv{1, \aa_1} \Rdiv{2, \aa_3} = 1/1/1$. Thus $\RDp\MM$ is $3$-semi-convergent. The argument is the same for a negative $3$-multifraction~$\aav$, finding now $\can(\aa_1)\inv\can(\aa_2) \can(\aa_3)\inv = 1$, whence $\aa_2 = \aa_1 \aa_3$, and $\aav \act \Rdiv{1, \aa_1} \Rdiv{2, \aa_3} = /1/1/1$. Hence $\RDb\MM$ is $3$-semi-convergent. So \ITEM4 implies~\ITEM3.
\end{proof}

\begin{coro}
For every Artin-Tits monoid~$\MM$, the system~$\RDb\MM$ is $\nn$-semi-convergent for $\nn \le 3$.
\end{coro}

\begin{proof}
By Lemma~\ref{L:SemiDown}, $\RDb\MM$ is $1$-semi-convergent. Next, it is known~\cite{Par} that $\MM$ embeds into~$\EG\MM$. Hence, by Proposition~\ref{P:Semi2}, $\RDb\MM$ is $2$- and $3$-semi-convergent.
\end{proof}

In other words, Conjecture~$\ConjA_\nn$ is true for $\nn \le 3$.

\subsection{Multifractions of depth~$4$}\label{SS:Semi4}

We now address $4$-semi-convergence, which turns out to give rise to interesting phenomena. We begin with preliminary results about unital multifractions that are in some sense the simplest ones.

\noindent\begin{minipage}{\textwidth}
\HS{0}\begin{defi}\label{D:Focal}\rightskip35mm
If $\MM$ is a monoid and $\aav$ is an $\nn$-multifraction on~$\MM$, with $\nn$ even, we say that $(\xx_1 \wdots \xx_\nn)$ is a \emph{central cross} for~$\aav$ if we have 
$$\aa_\ii = \begin{cases}
\ \xx_\ii \xx_{\ii + 1} & \text{for $\ii$ positive in~$\aav$},\\
\ \xx_{\ii + 1} \xx_\ii & \text{for $\ii$ negative in~$\aav$},
\end{cases}\HS{25}$$
with the convention $\xx_{\nn + 1} = \xx_1$. \hfill
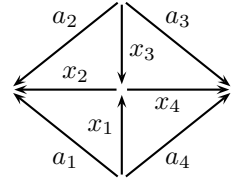
\begin{picture}(0,0)(-5,0)
\psset{nodesep=0.7mm}
\psset{yunit=0.8mm}
\pcline{->}(15,0)(0,15)\tbput{$\aa_1$}
\pcline{->}(15,30)(0,15)\taput{$\aa_2$}
\pcline{->}(15,30)(30,15)\taput{$\aa_3$}
\pcline{->}(15,0)(30,15)\tbput{$\aa_4$}
\pcline{->}(15,30)(15,15)\naput[npos=0.6]{$\xx_3$}
\pcline{->}(15,15)(0,15)\nbput[npos=0.4]{$\xx_2$}
\pcline{->}(15,0)(15,15)\naput[npos=0.6]{$\xx_1$}
\pcline{->}(15,15)(30,15)\nbput[npos=0.4]{$\xx_4$}
\end{picture}
\end{defi}
\end{minipage}

\medskip The diagram of Definition~\ref{D:Focal} shows that a multifraction that admits a central cross is unital: if $(\xx_1 \wdots \xx_\nn)$ is a central cross for a positive multifraction~$\aav$, we find
$$\can(\aav) = \can(\xx_1\xx_2) \, \can(\xx_3 \xx_2)\inv \, \can(\xx_3\xx_4) \, \pdots\, (\xx_1 \xx_\nn)\inv = 1,$$
and similarly when $\aav$ is negative. It follows from the definition that a sequence $(\xx_1 \wdots \xx_\nn)$ is a central cross for a positive multifraction~$\aav$ if and only if $(\xx_2 \wdots \xx_\nn, \xx_1)$ is a central cross for~$/\aa_2 \sdots \aa_\nn / \aa_1$. So, we immediately obtain

\begin{lemm}
For every monoid~$\MM$ and every even~$\nn$, a positive $\nn$-multifraction~$\aav$ admits a central cross if and only if the negative multifraction $/\aa_2 \sdots \aa_\nn / \aa_1$ does. 
\end{lemm}

Multifractions with a central cross always behave nicely in terms of reduction:

\begin{lemm}\label{L:CrossRed}
If $\MM$ is a gcd-monoid and $\aav$ is a multifraction on~$\MM$ that admits a central cross, then $\redtame(\aav) = \one$ holds. 
\end{lemm}

\begin{proof}
We prove the result using induction on~$\nn \ge 2$ even, and assuming $\aav$ positive. Assume that $(\xx_1 \wdots \xx_\nn)$ is a central cross for~$\aav$. For $\nn = 2$, the assumption boils down to $\aa_1 = \aa_2 = \xx_1 \xx_2$, directly implying $\aav \act \Redmax{1} = 1/1$. Assume $\nn \ge 4$ and, say, $\aav$ positive. Let $\xx := \xx_1 \gcdt \xx_3$, with $\xx_1 = \xx'_1 \xx$ and $\xx_3 = \xx'_3 \xx$. Then we have $\aa_1 \gcdt \aa_2 = \xx\xx_2$, whence $\aav \act \Redmax1 = \xx'_1 / \xx'_3 / \aa_3 \sdots \aa_\nn$. As $\aa_3 = \xx_3\xx_4$ expands into $\aa_3 = \xx'_3 \xx \xx_4$, this can be rewritten as $\aav \act \Redmax1 = \xx'_1 / \xx'_3 / \xx'_3 \xx \xx_4 \sdots \aa_\nn$. We deduce 
\begin{equation}\label{E:CrossRed1}
\aav \act \Redmax1 \Redmax2 = \xx'_1 \xx \xx_4 / 1 / 1 / \aa_4 \sdots \aa_\nn = \xx_1 \xx_4 / 1 / 1 / \aa_4 \sdots \aa_\nn
\end{equation}
with $\aa_4 = \xx_1 \xx_4$ for $\nn = 4$, and $\aa_4 = \xx_5 \xx_4$ for $\nn \ge 6$. In every case, the subsequent action of~$\Redmax3 \pdots \Redmax{\nn - 1}$ is to push $\aa_4$, then~$\aa_5$, etc. until~$\aa_\nn$, through~$1/1$, leading to
\begin{equation}\label{E:CrossRed2}
\aav \act \Redmax1 \pdots \Redmax{\nn - 1} = \xx_1 \xx_4 / \aa_4 \sdots \aa_\nn / 1 / 1.
\end{equation}
For $\nn = 4$, \eqref{E:CrossRed2} reads $\aav \act \Redmax1 \Redmax2 \Redmax3 = \xx_1 \xx_4 / \xx_1 \xx_4 / 1 / 1$, and a further application of~$\Redmax1$ yields~$\one$. For $\nn \ge 6$, the assumption that $(\xx_1 \wdots \xx_\nn)$ is a central cross for~$\aav$ implies that $(\xx_1, \xx_4 \wdots \xx_\nn)$ is a central cross for the $(\nn - 2)$-multifraction $\xx_1 \xx_4 / \aa_4 \sdots \aa_\nn$. The induction hypothesis for the latter gives $\redtame(\xx_1 \xx_4 / \aa_4 \sdots \aa_\nn) = \one$, which expands into
\begin{equation}\label{E:CrossRed3}
\xx_1 \xx_4 / \aa_4 \sdots \aa_\nn \act \Redmax{\univ{\nn - 2}} = \one_{\nn - 2}.
\end{equation}
By Lemma~\ref{L:Basics}\ITEM3, \eqref{E:CrossRed3} implies $\xx_1 \xx_4 / \aa_4 \sdots \aa_\nn / 1 / 1 \act \Redmax{\univ{\nn - 2}} = \one_\nn$. Merging with~\eqref{E:CrossRed2}, we deduce $\aav \act \Redmax{1 \pdots \nn - 1}Ê\Redmax{\univ{\nn - 2}} = \one_\nn$, which is $\aav \act \Redmax{\univ{\nn}} = \one_\nn$, \ie, $\redtame(\aav) = \one$. 

The argument is similar when $\aav$ is negative.
\end{proof}

We now concentrate on $4$-multifractions. A sort of transitivity of central crosses holds. 

\begin{lemm}\label{L:TransCross}
Assume that $\MM$ is a gcd-monoid, $\aav, \bbv$ are $4$-multifractions admitting a central cross, and $\cc_1\aa_4 = \cc_4\bb_1$ and $\cc_2\aa_3 = \cc_3\bb_2$ holds. Then $\cc_1\aa_1 / \cc_2\aa_2 / \cc_3\bb_3 / \cc_4\bb_4$ admits a central cross. In particular, $\aa_1 / \aa_2 / \bb_3 / \bb_4$ admits a central cross for $\aa_4 = \bb_1$ and $\aa_3 = \bb_2$.
\end{lemm}

\begin{proof}
(Figure~\ref{F:TransCross}) Let $(\xx_1 \wdots \xx_4)$ and $(\yy_1 \wdots \yy_4)$ be central crosses for~$\aav$ and~$\bbv$, respectively. By assumption, we have $\cc_2 \xx_3 \xx_4 = \cc_2 \aa_3 = \cc_3 \bb_2 = \cc_3 \yy_3 \yy_2$, so $\xx_4$ and~$\yy_2$ admit a common left multiple, say $\xx_4 \lcmt \yy_2 = \xx \xx_4 = \yy \yy_2$, and there exists~$\zz_3$ satisfying $\cc_2 \aa_3 = \zz_3 \xx \xx_4 = \zz_3 \yy \yy_2 = \cc_3 \bb_2$. Then $\aa_3 = \xx_3 \xx_4$ and $\bb_2 = \yy_3 \yy_2$ respectively imply
\begin{equation}
\cc_2 \xx_3 = \zz_3 \xx \qquand \cc_3 \yy_3 = \zz_3 \yy.
\end{equation}
Arguing similarly from $\cc_1 \xx_1 \xx_4 = \cc_1 \aa_4 = \cc_4 \bb_1 = \cc_4 \yy_1 \yy_2$, we deduce the existence of~$\zz_1$ satisfying $\cc_1 \aa_4 = \zz_1 \xx \xx_4 = \zz_1 \yy \yy_2 = \cc_4 \bb_1$, leading to 
\begin{equation}
\cc_1 \xx_1 = \zz_1 \xx \qquand \cc_4 \yy_1 = \zz_1 \yy.
\end{equation}
Then $(\zz_1, \xx\xx_2, \zz_3, \yy\yy_4)$ is a central cross for~$\cc_1\aa_1 / \cc_2\aa_2 / \cc_3\bb_3 / \cc_4\bb_4$.
\end{proof}

\begin{figure}[htb]
\begin{picture}(83,35)(0,1)
\psset{xunit=1.8mm,yunit=1.7mm}
\psset{nodesep=0.7mm}
\def\NPoint(#1,#2,#3){\cnode[style=thin,fillcolor=white,linecolor=white](#1,#2){0}{#3}}
\def\AArrow(#1,#2){\ncline[linewidth=0.8pt, border=2pt]{->}{#1}{#2}}
\def\BArrow(#1,#2){\ncline[linecolor=red, linewidth=1.5pt]{->}{#1}{#2}}
\def\CArrow(#1,#2){\ncline[linecolor=color1, linewidth=3mm]{c-c}{#1}{#2}}
\def\PArrow(#1,#2){\ncline[linestyle=dashed, linewidth=0.8pt]{->}{#1}{#2}}
\NPoint(10,0,v10) \NPoint(20,0,v20) \NPoint(30,0,v30) \NPoint(0,10,v01) \NPoint(10,10,v11) \NPoint(20,10,v21) \NPoint(30,10,v31) \NPoint(40,10,v41) \NPoint(0,20,v02) \NPoint(10,20,v12) \NPoint(20,20,v22) \NPoint(30,20,v32) \NPoint(20,15,x3) \NPoint(20,5,x1) 
\AArrow(v10,v01)\naput{$\aa_1$}
\AArrow(v12,v01)\nbput{$\aa_2$}
\AArrow(v20,v10)\naput{$\cc_1$}
\AArrow(v20,v30)\nbput{$\cc_4$}
\AArrow(v41,v30)\naput{$\bb_4$}
\AArrow(v41,v32)\nbput{$\bb_3$}
\AArrow(v22,v12)\nbput{$\cc_2$}
\AArrow(v22,v32)\naput{$\cc_3$}
\AArrow(v10,v11)\naput{$\xx_1$}
\AArrow(v11,v01)\nbput{$\xx_2$}
\AArrow(v12,v11)\nbput{$\xx_3$}
\AArrow(v11,v21)\naput{$\xx_4$}
\AArrow(v12,v21)\naput{$\aa_3$}
\AArrow(v10,v21)\nbput{$\aa_4$}
\AArrow(v32,v21)\nbput{$\bb_2$}
\AArrow(v30,v21)\naput{$\bb_1$}
\AArrow(v30,v31)\nbput{$\yy_1$}
\AArrow(v31,v21)\nbput{$\yy_2$}
\AArrow(v32,v31)\naput{$\yy_3$}
\AArrow(v31,v41)\naput{$\yy_4$}
\AArrow(v22,x3)\naput{$\zz_3$}
\AArrow(x3,v11)\nbput[npos=0.7]{$\xx$}
\AArrow(x3,v31)\naput[npos=0.7]{$\yy$}
\AArrow(v20,x1)\nbput{$\zz_1$}
\AArrow(x1,v11)\naput[npos=0.7]{$\xx$}
\AArrow(x1,v31)\nbput[npos=0.7]{$\yy$}
\end{picture}
\caption{\small Transitivity of the existence of a central cross.}
\label{F:TransCross}
\end{figure}
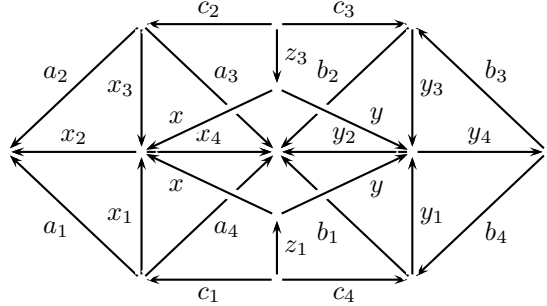

We deduce that reduction preserves the existence of a central cross in both directions.

\begin{lemm}\label{L:RedCross}
Assume that $\MM$ is a gcd-monoid and $\aav, \bbv$ are $4$-multifractions on~$\MM$ satisfying $\aav \rds \bbv$. Then $\aav$ admits a central cross if and only if $\bbv$ does. 
\end{lemm}

\begin{proof}
It is enough to prove the result for $\aav \rd \bbv$, say $\bbv = \aav \act \Red{\ii, \xx}$. Assume that $\aav$ is positive and $(\xx_1 \wdots \xx_4)$ is a central cross for~$\aav$. Our aim is to construct a central cross for~$\bbv$ from that for~$\aav$. Consider the case $\ii = 2$, see Figure~\ref{F:RedCross}. Let~$\aa_2 \xx' = \xx \lcm \aa_2$. By definition, we have 
\begin{equation}\label{E:RedCross}
\bb_1 = \aa_1 \xx', \quad \xx\bb_2 = \aa_2\xx', \quad \xx\bb_3 = \aa_3, \quad \bb_4 = \aa_4.
\end{equation}
By assumption, $\xx$ and~$\xx_3$ admit a common right multiple, namely~$\aa_3$, hence they admit a right lcm, say $\xx \lcm \xx_3 = \xx \yy_3 = \xx_3\yy$. We have $\xx\bb_2 = \aa_2\xx' = \xx_3 (\xx_2 \xx')$, hence the right lcm of~$\xx$ and~$\xx_3$ left divides~$\xx\bb_2$, \ie, we have $\xx\yy_3 \dive \xx \bb_2$, whence $\yy_3 \dive \bb_2$, say $\bb_2 = \yy_3 \yy_2$. 

Next, we have $\xx_3 \xx_2 \xx' = \aa_2 \xx' = \xx \bb_2 = \xx \yy_3 \yy_2 = \xx_3 \yy \yy_2$, whence $\xx_2 \xx' = \yy \yy_2$ by left cancelling~$\xx_3$. Put $\yy_1 = \xx_1 \yy$. We find $\bb_1 = \aa_1 \xx' = \xx_1 \xx_2 \xx' = \xx_1 \yy \yy_2 = \yy_1 \yy_2$.

On the other hand, by assumption, as both~$\xx$ and~$\xx_3$ left divide~$\aa_3$, their right lcm~$\xx_3 \yy$ left divide~$\aa_3$, which is~$\xx_3 \xx_4$. So we have $\xx_3 \yy \dive \xx_3 \xx_4$, whence $\yy \dive \yy_4$, say $\xx_4 = \yy \yy_4$. Then we find $\xx \bb_3 = \aa_3 = \xx_3 \xx_4 = \xx_3 \yy \yy_4 = \xx \yy_3 \yy_4$, whence $\bb_3 = \yy_3 \yy_4$. 

Finally, we have $\bb_4 = \aa_4 = \xx_1 \xx_4 = \xx_1 \yy \yy_4 = \yy_1 \yy_4$. So $(\yy_1 \wdots \yy_4)$ is a central cross for~$\bbv$. The argument for $\ii = 3$ is similar, mutatis mutandis, and so is the one for $\ii = 1$: in this case, the counterpart of~$\xx'$ is trivial, which changes nothing. Finally, the case when $\aav$ is negative is treated symmetrically. So $\bbv$ admits a central cross whenever $\aav$ does.

For the other direction, assume again that $\aav$ and~$\bbv$ are positive, and that $(\yy_1 \wdots \yy_4)$ is a central cross for~$\bbv$. Assume again $\ii = 2$, and \eqref{E:RedCross}. The equality $\xx\bb_2 =\nobreak \aa_2\xx$ implies that $(\aa_1, 1, \aa_2, \xx')$ is a central cross for $\aa_1 / \aa_2 / \xx\bb_2 /\bb_1$. On the other hand, $(\yy_1, \yy_2, \xx\yy_3, \yy_4)$ is a central cross for $\bb_1 / \xx\bb_2 / \aa_3 / \aa_4$. So both $\aa_1 / \aa_2 / \xx\bb_2 /\bb_1$ and $\bb_1 / \xx\bb_2 / \aa_3 / \aa_4$ admit a central cross. By Lemma~\ref{L:TransCross}, this implies that $\aa_1 / \aa_2 / \aa_3 / \aa_4$ admits a central cross. So $\aav$ admits a central cross whenever $\bbv$ does.
\end{proof}

The above proof does not use the assumption that $\xx \bb_2$ is the lcm of~$\xx$ and~$\aa_2$, but only the fact that the equalities of~\eqref{E:RedCross} hold for some~$\xx, \xx'$, which is the case, in particular, for $\bbv = \aav \act \Redt{\ii, \xx'}$. So it also shows that every right reduct of a $4$-multifraction with a central cross admits a central cross, and conversely. 

\begin{figure}[htb]
\begin{picture}(60,46)(0,4)
\psset{nodesep=0.7mm}
\psset{xunit=1mm,yunit=0.80mm}
\setlength{\unitlength}{0.85mm}
\pcline{->}(30,0)(0,30)\tbput{$\aa_1$}
\pcline{->}(30,60)(0,30)\taput{$\aa_2$}
\pcline{->}(30,60)(60,30)\taput{$\aa_3$}
\pcline{->}(30,0)(60,30)\nbput[npos=0.5]{$\aa_4 = \bb_4$}
\pcline{->}(30,0)(30,30)\tlput{$\xx_1$}
\pcline{->}(30,30)(0,30)\taput{$\xx_2$}
\pcline{->}(30,60)(30,30)\tlput{$\xx_3$}
\pcline{->}(30,30)(60,30)\taput{$\xx_4$}
\pcline{->}(30,60)(37,46)\tbput{$\xx$}
\pcline{->}(0,30)(14,23)\put(11,25){$\xx'$}
\pcline[style=exist]{->}(30,0)(37,23)\trput{$\yy_1$}
\pcline[border=1mm,style=exist]{->}(37,23)(14,23)\tbput{$\yy_2$}
\pcline[border=1mm,style=exist]{->}(37,45)(37,23)\trput{$\yy_3$}
\pcline[style=exist]{->}(37,23)(60,30)\tbput{$\yy_4$}
\pcline{->}(30,0)(14,23)\taput{$\bb_1$}
\pcline[border=1mm]{->}(37,46)(14,23)\taput{$\bb_2$}
\pcline{->}(37,46)(60,30)\tlput{$\bb_3$}
\pcline{->}(30,30)(37,23)\put(36,23){$\yy$}
\psarc[style=thin](14,23){3}{45}{150}
\psarc[style=thin](37,23){4}{90}{135}
\end{picture}
\caption{\small Construction of a central cross for~$\aav \act \Red{\ii, \xx}$ starting from one for~$\aav$.}
\label{F:RedCross}
\end{figure}

We deduce a complete description of the $4$-multifractions that reduce to~$\one$:

\begin{prop}\label{P:EquivFoc}
If $\MM$ is a gcd-monoid, then, for every $4$-multifraction~$\aav$ on~$\MM$, the following are equivalent:

\ITEM1 The relation $\aav \rds \one$ holds.

\ITEM2 The relation $\redtame(\aav) = \one$ holds.

\ITEM3 The multifraction~$\aav$ admits a central cross.
\end{prop}

\begin{proof}
The $4$-multifractions $1/1/1/1$ and $/1/1/1/1$ both admit the central cross $(1, 1, 1, 1)$. Hence, by Lemma~\ref{L:RedCross}, every $4$-multifraction satisfying $\aav \rds 1/1/1/1$ or $\aav \rds /1/1/1/1$ admits a central cross as well. Hence \ITEM1 implies~\ITEM3.

Next, Lemma~\ref{L:CrossRed} says that $\redtame(\aav) = \one$ holds for every multifraction~$\aav$ that admits a central cross, so \ITEM3 implies~\ITEM2. 

Finally, \ITEM2 implies~\ITEM1 by definition. 
\end{proof}

We return to the study of $\nn$-semi-convergence for~$\RDb\MM$, here for $\nn = 4, 5$. 

\begin{prop}\label{P:Semi4}
If $\MM$ is a gcd-monoid, the following are equivalent:

\ITEM1 The system~$\RDp\MM$ is $4$-semi-convergent.

\ITEM2 The system~$\RDb\MM$ is $4$-semi-convergent.

\ITEM3 The system~$\RDb\MM$ is $5$-semi-convergent.

\ITEM4 Every unital $4$-multifraction in~$\FRp\MM$ admits a central cross.
\end{prop}

\begin{proof}
Assume that $\RDp\MM$ is $4$-semi-convergent. Let $\aav$ be a unital positive $4$-multifraction. As $\RDp\MM$ is $4$-semi-convergent, we have $\aav \rds \one$. By Proposition~\ref{P:EquivFoc}, we deduce that $\aav$ admits a central cross. Hence \ITEM1 implies~\ITEM4.

By definition, \ITEM2 implies~\ITEM1, and, by Lemma~\ref{L:SemiDown}, \ITEM3 implies~\ITEM2.

Finally, assume~\ITEM4, and let $\aav$ be a unital $5$-multifraction on~$\MM$. Assume first that $\aav$ is positive. Then the positive $4$-multifraction $\aa_5\aa_1 / \aa_2 / \aa_3 / \aa_4$ is unital as well, hence, by assumption, it admits a central cross~$(\xx_1 \wdots \xx_4)$, which expands into $\aav = \aa_1 / \xx_3\xx_2 / \xx_3\xx_4 / \xx_1\xx_4 / \aa_5$ with $\aa_5\aa_1 = \xx_1\xx_2$. Then we obtain
\begin{align*}
\aav 
&\rds \aa_1 / \xx_2 / 1 / \xx_1 / \aa_5
&&\text{via $\Rdiv{2,\xx_3}\Rdiv{3,\xx_4}$}\\
&\rds \aa_1 / \xx_1\xx_2 / \aa_5 / 1 / 1 = \aa_1 / \aa_5\aa_1 / \aa_5 / 1 / 1
&&\text{via $\Red{3,\xx_1} \Red{4,\aa_5}$}\\
&\rds 1 / 1 / 1 / 1 / 1
&&\text{via $\Rdiv{1,\aa_1} \Rdiv{2,\aa_5}$}.
\end{align*}
If $\aav$ is negative, then $\aa_2 / \aa_3 / \aa_4 / \aa_1\aa_5$ is unital and positive, hence admits a central cross by assumption, leading to $\aav = /\aa_1 / \xx_1\xx_2 / \xx_3\xx_2 / \xx_3\xx_4 / \aa_5$ with $\aa_1 \aa_5 = \xx_1 \xx_4$, and to $\aav \rds \one$, this time via $\Rdiv{2,\xx_2} \Rdiv{3,\xx_3} \Red{3,\xx_4} \Red{4,\aa_5} \Rdiv{1,\aa_1} \Rdiv{2,\aa_5}$. Thus, every unital $5$-multifraction on~$\MM$ reduces to~$\one$, and $\RDb\MM$ is $5$-semi-conver\-gent. So \ITEM4 implies~\ITEM3.
\end{proof}

We now establish alternative forms for the point~\ITEM4 in Proposition~\ref{P:Semi4}, connected with the uniqueness of the expression by irreducible fractions.

\begin{prop}\label{P:UniqueFrac}
For every gcd-monoid~$\MM$, the following are equivalent:

\ITEM1 Every unital $4$-multifraction in~$\FRp\MM$ admits a central cross.

\ITEM2 For all $\aa, \bb, \cc, \dd$ in~$\MM$ satisfying $\can(\aa/\bb) = \can(\cc/\dd)$, there exist~$\xx, \yy$ in~$\MM$ satisfying 
\begin{equation}\label{E:4Focal3}
\aa = \xx (\aa \gcdt \bb), \quad \bb = \yy (\aa \gcdt \bb), \quad \cc = \xx (\cc \gcdt \dd), \quad \dd = \yy (\cc \gcdt \dd).
\end{equation}

\ITEM3 All $\aa, \bb, \cc, \dd$ in~$\MM$ satisfying $\can(\aa/\bb) = \can(\cc/\dd)$ and $\aa \gcdt \bb = 1$ satisfy $\aa \dive \cc$ and $\bb \dive \dd$.

\ITEM4 All $\aa, \bb, \cc, \dd$ in~$\MM$ satisfying $\can(\aa/\bb) = \can(\cc/\dd)$ and $\aa \gcdt \bb = \cc \gcdt \dd = 1$ satisfy $\aa = \cc$ and $\bb = \dd$.

\ITEM5 All $\aa, \bb, \cc$ in~$\MM$ satisfying $\aa \gcdt \bb = \bb \gcd \cc = 1$ and $\can(\aa / \bb / \cc) \in \can(\MM)$ satisfy $\bb = 1$.
\end{prop}

Before establishing Proposition~\ref{P:UniqueFrac}, we begin with two characterizations of $4$-multifractions with a central cross:

\begin{lemm}\label{L:4Focal}
If $\MM$ is a gcd-monoid, a positive $4$-multifraction~$\aav$ on~$\MM$ admits a central cross if and only if there exist~$\xx, \yy$ in~$\MM$ satisfying
\begin{equation}\label{E:4Focal1}
\aa_1 = \xx (\aa_1 \gcdt \aa_2), \quad \aa_2 = \yy (\aa_1 \gcdt \aa_2), \quad \aa_3 = \yy (\aa_3 \gcdt \aa_4), \quad \aa_4 = \xx (\aa_3 \gcdt \aa_4),
\end{equation}
if and only if there exist~$\xx', \yy'$ in~$\MM$ satisfying
\begin{equation}\label{E:4Focal2}
\aa_1 = (\aa_1 \gcdt \aa_4) \xx', \quad \aa_2 = (\aa_2 \gcdt \aa_3) \xx', \quad \aa_3 = (\aa_2 \gcdt \aa_3) \yy', \quad \aa_4 = (\aa_1 \gcdt \aa_4) \yy'.
\end{equation}
\end{lemm}

\begin{proof}
Assume that $\aav$ is positive and $(\xx_1 \wdots \xx_4)$ is a central cross for~$\aav$. Write $\xx_1 = \xx (\xx_1 \gcdt \xx_3)$ and $\xx_3 = \yy (\xx_1 \gcdt \xx_3)$. Then we have $\xx \gcdt \yy = 1$, and $\aa_1 = \xx (\xx_1 \gcdt \xx_3) \xx_2$, $\aa_2 = \yy (\xx_1 \gcdt \xx_3)\xx_2$, whence $\aa_1 \gcdt \aa_2 = (\xx_1 \gcdt \xx_3) \xx_2$, and $\aa_1 = \xx (\aa_1 \gcdt \aa_2)$, $\aa_2 = \yy(\aa_1 \gcdt \aa_2)$. On the other hand, we also find $\aa_3 = \xx_3 \xx_4 = \yy (\xx_1 \gcdt \xx_3) \xx_4$, $\aa_4 = \xx_1 \xx_4 = \xx (\xx_1 \gcdt \xx_3) \xx_4$, whence $\aa_3 \gcdt \aa_4 = (\xx_1 \gcdt \xx_3) \xx_4$, and $\aa_3 = \yy (\aa_3 \gcdt \aa_4)$, $\dd = \xx (\aa_3 \gcdt \aa_4)$. So we found~$\xx, \yy$ satisfying~\eqref{E:4Focal1}.

The proof for~\eqref{E:4Focal2} is similar, writing $\xx_2 = (\xx_2 \gcd \xx_4) \xx'$ and $\xx_4 = (\xx_2 \gcd \xx_4) \yy'$ and deducing $\aa_1 \gcd \aa_4 = \xx_1 (\xx_2 \gcd \xx_4)$ and $\aa_2 \gcd \aa_3 = \xx_3 (\xx_2 \gcd \xx_4)$.

In the other direction, if \eqref{E:4Focal1} is satisfied, $(\xx, \aa_1 \gcdt \aa_2, \yy, \aa_3 \gcdt \aa_4)$ is a central cross for~$\aav$, whereas, if \eqref{E:4Focal2} is satisfied, so is $(\aa_1 \gcd \aa_4, \xx', \aa_2 \gcd \aa_3, \yy)$.
\end{proof}

\begin{proof}[Proof of Proposition~\ref{P:UniqueFrac}]
If $\can(\aa/\bb) = \can(\cc/\dd)$ holds, the $4$-multifraction $\aa / \bb / \dd / \cc$ is unital, hence it admits a central cross if \ITEM1 is true. In this case, Lemma~\ref{L:4Focal} gives~\eqref{E:4Focal3}. So \ITEM1 implies~\ITEM2. 

Next, applying \ITEM2 in the case $\aa \gcdt \bb = 1$ gives $\aa = \xx$ and $\bb = \yy$, whence $\aa \dive \cc$ and $\bb \dive \dd$. If, in addition, we have $\cc \gcdt \dd = 1$, we similarly obtain $\cc \dive \aa$ and $\dd \dive \bb$, whence $\aa = \cc$ and $\bb = \dd$. 
So \ITEM2 implies~\ITEM3 and~\ITEM4. On the other hand, \ITEM4 implies~\ITEM3. Indeed, assume $\can(\aa/\bb) = \can(\cc/\dd)$ with $\aa \gcdt \bb = 1$. Let $\ee = \cc \gcdt \dd$, with $\cc = \cc' \ee$ and $\cc = \dd'\ee$. Then we have $\cc' \gcdt \dd' = 1$ and $\can(\aa/\bb) = \can(\cc'/\dd')$. Then \ITEM4 implies $\aa = \cc' \dive \cc$ and~$\bb = \dd' \dive \dd$.

Now, let $\aa, \bb, \cc$ in~$\MM$ satisfy $\aa \gcdt \bb = \bb \gcd \cc = 1$ and $\can(\aa / \bb / \cc) = \can(\dd)$ for some~$\dd$ in~$\MM$. Then $\can(\aa/\bb) = \can(\dd / \cc)$ holds so, if \ITEM2 holds, we have $\aa \dive \dd$ and $\bb \dive \cc$. Then the assumption $\bb \gcd \cc = 1$ implies $\bb = 1$. So \ITEM3 implies~\ITEM5.

Finally, assume that $\aav $ is a unital $4$-multifraction in~$\FRp\MM$. Write $\aa_2 = \xx \aa'_2$ and $\aa_3 = \xx \aa'_3$ with $\xx = \aa_2 \gcd \aa_3$, so that $\aa'_2 \gcd \aa'_3 = 1$ holds. Next, write $\aa_1 = \aa'_1 \yy$ and $\aa'_2 = \aa''_2$ with $\yy = \aa_1 \gcdt \aa'_2$, so that $\aa'_1 \gcdt \aa''_2 = 1$ holds. Then we have $\can (\aa'_1 / \aa''_2 / \aa'_3) = \can (\aa_1 / \aa'_2 / \aa'_3) = \can (\aa_1 / \aa_2 / \aa_3)$, 
and the assumption $\can(\aav) = 1$ implies $\can (\aa_1 / \aa_2 / \aa_3) = \can(\aa_4) \in \can(\MM)$. Hence $\can (\aa'_1 / \aa''_2 / \aa'_3)$ lies in~$\can(\MM)$ with $\aa'_1 \gcdt \aa''_2 = 1$ and $\aa'_2 \gcd \aa'_3 = 1$, whence a fortiori $\aa''_2 \gcd \aa'_3 = 1$. If \ITEM5 is true, we deduce $\aa''_2 = 1$. This means that $(\aa'_1, \yy, \xx, \aa'_3)$ is a central cross for~$\aav$. So \ITEM5 implies~\ITEM1.
\end{proof}

Putting things together, we obtain:

\begin{coro}\label{C:A4B4}
Conjectures~$\ConjA_4$ and~$\ConjB_4$ are equivalent, and they are equivalent to the property that, for every Artin-Tits monoid~$\MM$, every element of~$\EG\MM$ of the form $\aa\bb\inv$ with $\aa, \bb$ in~$\MM$ admits only one such expression with $\aa \gcdt \bb = 1$.
\end{coro}

\begin{proof}
Let $\MM$ be an Artin-Tits monoid and $\aav$ be a unital $4$-multifraction on~$\MM$. By Proposition~\ref{P:EquivFoc}, $\aav \rds \one$ implies $\redtame(\aav) = \one$, so Conjecture~$\ConjA_4$ implies Conjecture~$\ConjB_4$. On the other hand, by Proposition~\ref{P:EquivFoc} again, the property ``$\aav$ unital implies $\aav \rds \one$'' is equivalent to `` $\aav$ unital implies $\aav$ admits a central cross'' and, by Proposition~\ref{P:UniqueFrac}, the latter is equivalent to the uniqueness of fractional decompositions.
\end{proof}

\subsection{An application to partial orderings on~$\EG\MM$}\label{SS:Ordering}

If $\MM$ is a Garside monoid, it is known~\cite[Sec.~II.3.2]{Dir} that the left and right divisibility relations of~$\MM$ can be extended into well-defined partial orders on the enveloping group~$\EG\MM$ by declaring $\gg \dive \hh$ (\resp $\gg \divet \hh$) for $\gg\inv \hh \in \MM$ (\resp $\gg \hh\inv \in \MM$). The construction extends to every gcd-monoid, and we show that lattice properties are preserved whenever $\RDb\MM$ is $4$-semi-convergent.

\begin{prop}\label{P:Lattice}
\ITEM1 If $\MM$ is a gcd-monoid that embeds into~$\EG\MM$, then declaring $\gg\dive \hh$ for $\gg\inv \hh \in \can(\MM)$ provides a partial order on~$\EG\MM$ that extends left divisibility on~$\MM$.

\ITEM2 If $\RDb\MM$ is $4$-semi-conver\-gent, any two elements of~$\EG\MM$ that admit a common $\dive$-lower bound (\resp a common $\dive$-upper bound) admit a greatest one (\resp a lowest one).
\end{prop}

\begin{proof}
\ITEM1 We identify~$\MM$ with its image under~$\can$. As $\MM$ is a semigroup in~$\EG\MM$, the relation~$\dive$ is transitive on~$\EG\MM$, and it is antisymmetric, as $1$ is the only invertible element of~$\MM$. Hence $\dive$ is a partial order on~$\EG\MM$. By definition, it extends left divisibility on~$\MM$.

\ITEM2 Assume that $\gg$ and~$\hh$ admit a common $\dive$-lower bound~$\ff$ (see Figure~\ref{F:Lattice} left). This means that we have $\gg = \ff \xx$ and $\hh = \ff \yy$ for some $\xx, \yy$ in~$\MM$. Let $\ff_0 := \ff(\xx \gcd \yy)$. Write $\xx = (\xx \gcd \yy) \xx_0$, $\yy = (\xx \gcd \yy) \yy_0$. Then we have $\gg = \ff_0 \xx_0$ and $\hh = \ff_0 \yy_0$, whence $\ff_0 \dive \gg$ and $\ff_0 \dive \hh$. Now assume that $\ff_1$ is any common $\dive$-lower bound of~$\gg$ and~$\hh$, say $\gg = \ff_1 \xx_1$ and $\hh = \ff_1 \yy_1$. In the group~$\EG\MM$, we have $\xx_0\xx_1\inv = \yy_0\yy_1\inv$. Hence, the $4$-multifraction $\xx_0 / \xx_1 / \yy_1 / \yy_0$ is unital, so, by assumption, it admits a central cross. Then Lemma~\ref{L:4Focal} provides~$\xx', \yy'$ satisfying
$$\xx_0 = (\xx_0 \gcd \yy_0) \xx', \quad \xx_1 = (\xx_1 \gcd \yy_1) \xx', \quad \yy_1 = (\xx_1 \gcd \yy_1) \yy', \quad \yy_0 = (\xx_0 \gcd \yy_0) \yy'.$$
By definition of~$\ff_0$, we have $\xx_0 \gcd \yy_0 = 1$, whence $\xx' = \xx_0$ and $\yy' = \yy_0$, and, from there, $\ff_1 (\xx_1 \gcd \yy_1) = \ff_0$, hence $\ff_1 \dive \ff_0$, in~$\EG\MM$. So $\ff_0$ is a greatest $\dive$-lower bound for~$\gg$ and~$\hh$.

The argument for lowest $\dive$-upper bound is symmetric. Assume that $\gg$ and~$\hh$ admit a common $\dive$-upper bound~$\ff$ (see Figure~\ref{F:Lattice} right). Write $\ff = \gg \xx = \hh \yy$, then $\xx = \xx_0(\xx \gcdt \yy)$ and $\yy = \yy_0(\xx \gcdt \yy)$. Put $\ff_0 := \gg\xx_0$. We have $\ff_0(\xx \gcdt \yy) = \ff = \hh \yy = \hh \yy_0 (\xx \gcdt \yy)$, whence $\hh \yy_0 = \ff_0$. So $\ff_0$ is a common $\dive$-upper bound of~$\gg$ and~$\hh$. Now assume that $\ff_1$ is any common $\dive$-upper bound of~$\gg$ and~$\hh$, say $\ff_1 = \gg_1 \xx_1 = \hh \yy_1$. In~$\EG\MM$, we have $\xx_0 \yy_0\inv = \xx_1 \yy_1\inv$, \ie, the $4$-multifraction $\xx_0 / \yy_0 / \yy_1 / \xx_1$ is unital, hence it admits a central cross. Then Lemma~\ref{L:4Focal} provides~$\xx, \yy$ satisfying
$$\xx_0 = \xx (\xx_0 \gcdt \yy_0), \quad \yy_0 = \yy (\xx_0 \gcdt \yy_0), \quad \xx_1 = \xx (\xx_1 \gcdt \yy_1), \quad \yy_1 = \yy (\xx_1 \gcdt \yy_1).$$
The definition of~$\ff_0$ gives $\xx_0 \gcdt \yy_0 = 1$, whence $\xx = \xx_0$ and $\yy = \yy_0$, leading in~$\EG\MM$ to $\ff_0 (\xx_1 \gcdt \yy_1) = \ff_1$, hence $\ff_0 \dive \ff_1$. So $\ff_0$ is a lowest $\dive$-upper bound for~$\gg$ and~$\hh$.
\end{proof}

\begin{figure}[htb]\centering
\begin{picture}(45,28)(0,-2)
\psset{nodesep=1mm}
\pcline(0,24)(39,24)\taput{$\yy$}\psarc{<-}(38,17){7}{0}{90}
\pcline(0,24)(0,6)\tlput{$\xx$}\psarc(7,7){7}{180}{270}\pcline{->}(6,0)(15,0)
\pscircle(0,24){0.5}\put(-3,24){$\ff$}
\pcline{->}(0,24)(15,16)\put(6,22){\rotatebox{-25}{\hbox{$\xx \gcd \yy$}}}
\pscircle(15,16){0.5}\put(15,18){$\ff_0$}
\pcline{->}(15,16)(45,16)\taput{$\yy_0$}
\pcline{->}(15,16)(15,0)\tlput{$\xx_0$}
\psarc[style=thin](15,16){3.5}{270}{360}
\pscircle(45,16){0.5}\put(46,16){$\hh$}
\pscircle(15,0){0.5}\put(14,-3){$\gg$}
\pcline{->}(45,0)(45,16)\trput{$\yy_1$}
\pcline{->}(45,0)(15,0)\tbput{$\xx_1$}
\pscircle(45,0){0.5}\put(45,-3){$\ff_1$}
\pcline{->}(45,0)(30,8)\put(33.5,7){\rotatebox{-25}{\hbox{$\xx_1 \gcd \yy_1$}}}
\pcline{->}(15,16)(30,8)\put(20,14.2){\rotatebox{-25}{\hbox{$\xx_0 \gcd \yy_0$}}}
\pcline{->}(30,8)(15,0)\nbput[npos=0.6]{$\xx'$}
\pcline{->}(30,8)(45,16)\nbput[npos=0.6]{$\VR(2,0)\smash{\yy'}$}
\end{picture}
\hspace{15mm}
\begin{picture}(45,28)(0,-2)
\psset{nodesep=1mm}
\pcline(30,24)(39,24)\psarc(38,17){7}{0}{90}\pcline{->}(45,17)(45,0)\trput{$\yy$}
\pcline(0,10)(0,6)\psarc(7,7){7}{180}{270}\pcline{->}(6,0)(45,0)\tbput{$\xx$}
\pscircle(45,0){0.5}\put(45,-3){$\ff$}
\pcline{->}(30,8)(45,0)\put(34.5,6.7){\rotatebox{-26}{\hbox{$\xx \gcdt \yy$}}}
\pscircle(30,8){0.5}\put(31,9){$\ff_0$}
\pcline{->}(30,24)(30,8)\trput{$\yy_0$}
\pcline{->}(0,8)(30,8)\tbput{$\xx_0$}
\psarc[style=thin](30,8){3.5}{90}{180}
\pscircle(30,24){0.5}\put(29,25.5){$\hh$}
\pscircle(0,8){0.5}\put(-3,10){$\gg$}
\pcline{->}(30,24)(0,24)\taput{$\yy_1$}
\pcline{->}(0,8)(0,24)\tlput{$\xx_1$}
\pscircle(0,24){0.5}\put(-4,24){$\ff_1$}
\pcline{->}(15,16)(0,24)\put(5,22.5){\rotatebox{-27}{\hbox{$\xx_1 \gcdt \yy_1$}}}
\pcline{->}(15,16)(30,8)\put(17.5,15.8){\rotatebox{-27}{\hbox{$\xx_0 \gcdt \yy_0$}}}
\pcline{->}(0,8)(15,16)\naput[npos=0.4]{$\xx$}
\pcline{->}(30,24)(15,16)\naput[npos=0.4]{$\yy$}
\end{picture}
\caption{\small Conditional lattice property for the poset $(\EG\MM, \dive)$: greatest lower bound on the left, lowest upper bound on the right.}
\label{F:Lattice}
\end{figure}
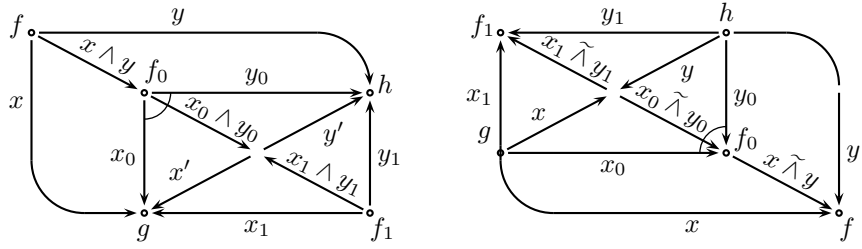

Of course, we have a symmetric extension for the right divisibility relation. 

\begin{ques}
Is the assumption that $\MM$ embeds in~$\EG\MM$ sufficient to ensure that conditional greatest lower bounds and lowest upper bounds for~$\dive$ exist in~$\EG\MM$?
\end{ques}

\subsection{ Depth~$6$ and beyond}\label{SS:Semi6}

The results established in the case of $4$-multifractions do not extend to depth~$6$ and beyond. In particular, unital $6$-multifract\-ions need not admit a central cross: in the Artin-Tits monoid of type~$\Att$, let $\aav := \tta\ttb\tta / \tta\ttc\tta / \ttc\tta\ttc / \ttc\ttb\ttc / \ttb\ttc\ttb / \ttb\tta\ttb$ and $\bbv := \tta\ttb / \tta\ttc / \ttc\tta / \ttc\ttb / \ttb\ttc / \ttb\tta$. We have $\aav \rd \bbv \rds \one$, and $\aav$ admits a central cross, but $\bbv$ does not, as it is prime. So the counterparts of Lemma~\ref{L:RedCross} and Proposition~\ref{P:EquivFoc} fail. However, we shall see now that, for every~$\nn$, Conjecture~$\ConjB_\nn$ implies a geometrical property of van Kampen diagrams that directly extends Prop\,\ref{P:EquivFoc}.

\noindent\begin{minipage}{\textwidth}
\rightskip32mm\VR(3,0)\HS{3} What the latter says is that, if we define $(\UG4, *)$ to be the pointed graph on the right, then, for every $4$-multifraction~$\aav$ that reduces to~$\one$, there exists an $\MM$-labeling of the edges of~$\UG4$ such that the outer labels from~$*$ are~$\aa_1 \wdots \aa_4$ and the labels in each triangle induce equalities in~$\MM$.\hfill
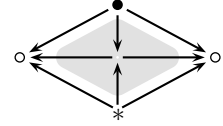
\begin{picture}(0,0)(-1,-1)
\def\NPoint(#1,#2,#3){\cnode[style=thin,fillcolor=white,linecolor=white](#1,#2){0}{#3}}
\def\BPoint(#1,#2,#3){\cnode[fillstyle=solid,fillcolor=black,linecolor=black](#1,#2){0.7}{#3}}
\def\AArrow(#1,#2){\ncline[linewidth=0.8pt]{->}{#1}{#2}}
\psset{nodesep=0.7mm}
\psset{xunit=1.3mm, yunit=0.7mm}
\pspolygon[linearc=3,linewidth=3mm,linecolor=white,fillstyle=solid,fillcolor=color1](10,0)(0,10)(10,20)(20,10)
\NPoint(10, 0, w0) \WPoint(0,10, w1) \BPoint(10,20, w2) \WPoint(20,10, w3) \NPoint(10,10, x1) 
\AArrow(w0,w1)
\AArrow(w2,w1)
\AArrow(w2,w3)
\AArrow(w0,w3)
\AArrow(w0,x1)
\AArrow(w2,x1)
\AArrow(x1,w1)
\AArrow(x1,w3)
\put(12.2, -1.5){$*$}
\end{picture}
\end{minipage}

\noindent If \VR(3.1, 0) $(\Gamma, *)$ is a finite, simply connected pointed graph, let us say that a multifraction~$\aav$ on a monoid~$\MM$ admits a \emph{van Kampen diagram of shape~$\Gamma$} if there is an $\MM$-labeling of~$\Gamma$ such that the outer labels from~$*$ are~$\aa_1 \wdots \aa_\nn$ and the labels in each triangle induce equalities in~$\MM$. This notion is a mild extension of the usual one: if $\SS$ is any generating set for~$\MM$, then replacing the elements of~$\MM$ with words in~$\SS$ and equalities with word equivalence provides a van Kampen diagram in the usual sense for the word in~$\SS \cup \SSb$ then associated with~$\aav$. Then, Proposition~\ref{P:EquivFoc} says that every $4$-multifraction reducing to~$\one$ admits a van Kampen diagram of shape~$\UG4$. Conjecture~$\ConjB$ predicts similar results for every depth.

\begin{defi}
For $\nn \ge 6$ even, let~$(\UG\nn, *)$ be the graph obtained by appending $\nn - 2$ adjacent copies of~$\UG4$ around~$(\UG{\nn - 2}, *)$ starting from~$*$, with alternating orientations, and connecting the last copy of~$(\UG4, *)$ with the first one, see Figure~\ref{F:UnivGr}. 
\end{defi}

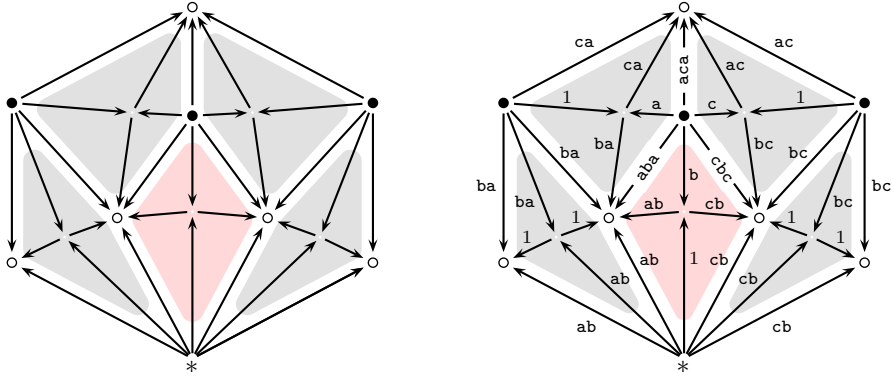
\begin{figure}[htb]
\def\NPoint(#1,#2,#3){\cnode[style=thin,fillcolor=white,linecolor=white](#1,#2){0}{#3}}
\def\BPoint(#1,#2,#3){\cnode[fillstyle=solid,fillcolor=black,linecolor=black](#1,#2){0.7}{#3}}
\def\AArrow(#1,#2){\ncline[linewidth=0.8pt]{->}{#1}{#2}}
\def\BArrow(#1,#2){\ncline[linecolor=color3, linewidth=1.5pt]{->}{#1}{#2}}
\def\CArrow(#1,#2){\ncline[linecolor=color1, linewidth=3mm]{c-c}{#1}{#2}}
\def\DArrow(#1,#2){\ncline[doubleline=true, linewidth=0.4pt]{#1}{#2}}
\def\PArrow(#1,#2){\ncline[linestyle=dashed, linewidth=0.8pt]{->}{#1}{#2}}
\begin{picture}(48,47)(0,1)
\psset{nodesep=0.7mm}
\psset{xunit=1mm, yunit=0.85mm}
\pspolygon[linearc=3,linewidth=3mm,linecolor=white,fillstyle=solid,fillcolor=color1](24,0)(0,15)(0,40)(14,22)
\pspolygon[linearc=3,linewidth=3mm,linecolor=white,fillstyle=solid,fillcolor=color1](14,22)(0,40)(24,55)(24,38)
\pspolygon[linearc=3,linewidth=3mm,linecolor=white,fillstyle=solid,fillcolor=color1](34,22)(48,40)(24,55)(24,38)
\pspolygon[linearc=3,linewidth=3mm,linecolor=white,fillstyle=solid,fillcolor=color1](24,0)(48,15)(48,40)(34,22)
\pspolygon[linearc=3,linewidth=3mm,linecolor=white,fillstyle=solid,fillcolor=color2](24,0)(14,22)(24,38)(34,22)
\NPoint(24, 0, v0) \WPoint(0, 15, v1) \BPoint(0, 40, v2) \WPoint(24, 55, v3) \BPoint(48, 40, v4) \WPoint(48, 15, v5)
\NPoint(24, 0, w0) \WPoint(14, 22, w1) \BPoint(24, 38, w2) \WPoint(34, 22, w3) 
\NPoint(7, 19, x1) \NPoint(16, 38.5, x2) \NPoint(32, 38.5, x3) \NPoint(41, 19, x4) 
\NPoint(24, 23, x0) 
\AArrow(v0,v1)
\AArrow(v2,v1)
\AArrow(v2,v3)
\AArrow(v4,v3)
\AArrow(v4,v5)
\AArrow(v0,v5)
\AArrow(w0,w1)
\AArrow(w2,w1)
\AArrow(w2,w3)
\AArrow(w0,w3)
\AArrow(v0,x1) \AArrow(v2,x1) \AArrow(x1,v1) \AArrow(x1,w1) \AArrow(v2,w1)
\AArrow(x6,v7) \AArrow(x6,w5) 
\AArrow(v6,w5)
\AArrow(v2,x2)\AArrow(w2,x2)\AArrow(x2,v3)\AArrow(x2,w1) 
\AArrow(w2,v3)
\AArrow(w0,x4)\AArrow(w0,v5)
\AArrow(v4,x3) \AArrow(w2,x3)
\AArrow(x3,v3) \AArrow(x3,w3)
\AArrow(v4,w3)
\AArrow(v4,x4) \AArrow(x4,v5)\AArrow(x4,w3)
\AArrow(w0,x0) \AArrow(w2,x0) \AArrow(x0, w1) \AArrow(x0, w3) 
\put(23,-2){$*$}
\end{picture}
\hspace{15mm}
\begin{picture}(48,47)(0,1)
\psset{nodesep=0.7mm}
\psset{xunit=1mm, yunit=0.85mm}
\pspolygon[linearc=3,linewidth=3.5mm,linecolor=white,fillstyle=solid,fillcolor=color1](24,0)(0,15)(0,40)(14,22)
\pspolygon[linearc=3,linewidth=3.5mm,linecolor=white,fillstyle=solid,fillcolor=color1](14,22)(0,40)(24,55)(24,38)
\pspolygon[linearc=3,linewidth=3.5mm,linecolor=white,fillstyle=solid,fillcolor=color1](34,22)(48,40)(24,55)(24,38)
\pspolygon[linearc=3,linewidth=3.5mm,linecolor=white,fillstyle=solid,fillcolor=color1](24,0)(48,15)(48,40)(34,22)
\pspolygon[linearc=3,linewidth=3.5mm,linecolor=white,fillstyle=solid,fillcolor=color2](24,0)(14,22)(24,38)(34,22)
\NPoint(24, 0, v0) \WPoint(0, 15, v1) \BPoint(0, 40, v2) \WPoint(24, 55, v3) \BPoint(48, 40, v4) \WPoint(48, 15, v5)
\NPoint(24, 0, w0) \WPoint(14, 22, w1) \BPoint(24, 38, w2) \WPoint(34, 22, w3) 
\NPoint(7, 19, x1) \NPoint(16, 38.5, x2) \NPoint(32, 38.5, x3) \NPoint(41, 19, x4) 
\NPoint(24, 23, x0) 
\AArrow(w2,v3)\ncput*[labelsep=1pt,npos=0.4,nrot=90]{$\scriptstyle\fr{aca}$}
\AArrow(w2,w1)\ncput*[labelsep=1pt,npos=0.5,nrot=55 ]{$\scriptstyle\fr{aba}$}
\AArrow(w2,w3)\ncput*[labelsep=1pt,npos=0.5,nrot=-55]{$\scriptstyle\fr{cbc}\!$}
\AArrow(v0,v1)\naput[npos=0.5]{$\scriptstyle\fr{ab}$}
\AArrow(v2,v1)\tlput{$\scriptstyle\fr{ba}$}
\AArrow(v2,v3)\naput[npos=0.5]{$\scriptstyle\fr{ca}$}
\AArrow(v4,v3)\nbput[npos=0.5]{$\scriptstyle\fr{ac}$}
\AArrow(v4,v5)\trput{$\scriptstyle\fr{bc}$}
\AArrow(v0,v5)\nbput[npos=0.5]{$\scriptstyle\fr{cb}$}
\AArrow(w0,w1)\nbput[labelsep=1pt,npos=0.7]{$\scriptstyle\fr{ab}$}
\AArrow(w0,w3)\naput[labelsep=1pt,npos=0.7]{$\scriptstyle\fr{cb}$}
\AArrow(v0,x1)\nbput[labelsep=0.5pt,npos=0.6]{$\scriptstyle\fr{ab}$}
\AArrow(v2,x1)\nbput[labelsep=1pt,npos=0.7]{$\scriptstyle\fr{ba}$} 
\AArrow(x1,v1)\nbput[labelsep=1pt,npos=0.5]{$\scriptstyle1$} 
\AArrow(x1,w1)\naput[labelsep=1pt,npos=0.5]{$\scriptstyle1$} 
\AArrow(v2,w1)\naput[labelsep=1pt,npos=0.5]{$\scriptstyle\fr{ba}$}
\AArrow(v2,x2)\naput[labelsep=1pt,npos=0.5]{$\scriptstyle1$}
\AArrow(w2,x2)\nbput[labelsep=1pt,npos=0.4]{$\scriptstyle\fr{a}$}
\AArrow(w2,x3)\naput[labelsep=1pt,npos=0.4]{$\scriptstyle\fr{c}$}
\AArrow(x2,v3)\naput[labelsep=1pt,npos=0.4]{$\scriptstyle\fr{ca}$} 
\AArrow(x2,w1)\nbput[labelsep=1pt,npos=0.3]{$\scriptstyle\fr{ba}$} 
\AArrow(w0,x4)\naput[labelsep=0.5pt,npos=0.6]{$\scriptstyle\fr{cb}$} 
\AArrow(v4,x3)\nbput[labelsep=1pt,npos=0.5]{$\scriptstyle1$} 
\AArrow(x3,v3)\nbput[labelsep=1pt,npos=0.4]{$\scriptstyle\fr{ac}$} 
\AArrow(x3,w3)\naput[labelsep=1pt,npos=0.3]{$\scriptstyle\fr{bc}$}
\AArrow(v4,w3)\nbput[labelsep=1pt,npos=0.5]{$\scriptstyle\fr{bc}$}
\AArrow(v4,x4)\naput[labelsep=1pt,npos=0.7]{$\scriptstyle\fr{bc}$} 
\AArrow(x4,v5) \naput[labelsep=1pt,npos=0.5]{$\scriptstyle1$}
\AArrow(x4,w3)\nbput[labelsep=1pt,npos=0.5]{$\scriptstyle1$}
\AArrow(w0,x0)\nbput[labelsep=1pt,npos=0.7]{$\scriptstyle1$} 
\AArrow(w2,x0) \naput[labelsep=1pt,npos=0.6]{$\scriptstyle\fr{b}$}
\AArrow(x0, w1)\nbput[labelsep=1pt,npos=0.4]{$\scriptstyle\fr{ab}$} 
\AArrow(x0, w3)\naput[labelsep=1pt,npos=0.4]{$\scriptstyle\fr{cb}$} 
\put(23,-2){$*$}
\end{picture}
\caption{\small The graph~$\UG6$: on the left, the naked graph, with four juxtaposed copies of~$\UG4$ around one (colored) copy of~$\UG4$; it has fourteen vertices, namely four springs (one inner), five wells (two inner), and five $4$-prongs; on the right, for $\MM$ the Artin-Tits monoid of type~$\Att$, the $\MM$-labeling of~$\UG6$ that results from the first reduction in Example~\ref{X:Att2}; it provides a van Kampen diagram for the $6$-multifraction $\fr{ab/ba/ca/ac/bc/ab}$. Conjecture~$\ConjB$ predicts that \emph{every} $6$-multifraction representing~$1$ in an Artin-Tits monoid admits a van Kampen diagram of shape~$\UG6$.}
\label{F:UnivGr}
\end{figure}

\enlargethispage{2mm}

\begin{prop}\label{P:Tiling}
Let $\MM$ be an Artin-Tits monoid. If Conjecture~$\ConjB$ is true, then every unital $\nn$-multifraction on~$\MM$ $($with $\nn$ even$)$ admits a van Kampen diagram of shape~$\UG\nn$, see Figure~\ref{F:UnivGr}.
\end{prop}

\begin{proof}
Let $\aav$ be a unital $\nn$-multifraction on~$\MM$. Conjecture~$\ConjB$ predicts the equality $\aav \act \Redmax{\univ\nn} = \one$. We shall see that the latter (and, more generally, any equality of the form $\aav \act \Red{\univ\nn, \xxv} = \one$, with $\xxv$ maximal or not) implies that $\aav$ admits a van Kampen diagram of shape~$\UG\nn$. In view on an induction, assume $\aav \act \Red{1, \xx_1} \pdots \Red{\nn - 1, \xx_{\nn - 1}} = \bbv \opp \One2$, for some (necessarily unital) $(\nn - 2)$-multifraction~$\bbv$. Let $\aav^0:= \aav$ and, inductively, $\aav^\ii:= \aav^{\ii - 1} \act \Red{\ii, \xx_\ii}$. We start from a loop of edges with alternating orientations labeled~$\aav^0$. Then we inductively complete the graph using $\nn - 1$~steps of the type 
$$
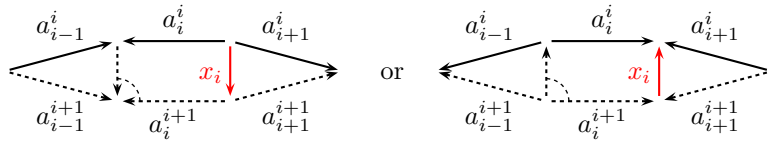
\begin{picture}(45,15)(0,-4) \psset{nodesep=0.7mm}
\pcline{->}(0,4)(15,8)\taput{$\aa^\ii_{\ii - 1}$} 
\pcline{<-}(15,8)(30,8)\taput{$\aa^\ii_\ii$} 
\pcline{->}(30,8)(45,4)\taput{$\aa^\ii_{\ii + 1}$}
\pcline[style=exist]{->}(0,4)(15,0)\tbput{$\aa^{\ii + 1}_{\ii - 1}$} 
\pcline[style=exist]{<-}(15,0)(30,0)\tbput{$\aa^{\ii + 1}_{\ii}$} 
\pcline[style=exist]{->}(30,0)(45,4)\tbput{$\aa^{\ii + 1}_{\ii + 1}$} 
\pcline[style=exist]{->}(15,8)(15,0) 
\pcline[linecolor=color3]{->}(30,8)(30,0)\tlput{\color{color3}$\xx_\ii$}
\psarc[style=thinexist](15,0){3}{0}{90}
\put(50,3){or}
\end{picture}
\hspace{12mm}
\begin{picture}(45,15)(0,-4) \psset{nodesep=0.7mm}
\pcline{<-}(0,4)(15,8)\taput{$\aa^\ii_{\ii - 1}$} 
\pcline{->}(15,8)(30,8)\taput{$\aa^\ii_\ii$} 
\pcline{<-}(30,8)(45,4)\taput{$\aa^\ii_{\ii + 1}$}
\pcline[style=exist]{<-}(0,4)(15,0)\tbput{$\aa^{\ii + 1}_{\ii - 1}$} 
\pcline[style=exist]{->}(15,0)(30,0)\tbput{$\aa^{\ii + 1}_{\ii}$} 
\pcline[style=exist]{<-}(30,0)(45,4)\tbput{$\aa^{\ii + 1}_{\ii + 1}$} 
\pcline[style=exist]{<-}(15,8)(15,0) 
\pcline[linecolor=color3]{<-}(30,8)(30,0)\tlput{\color{color3}$\xx_\ii$}\psarc[style=thinexist](15,0){3}{0}{90}
\end{picture}
$$
according to the sign of~$\ii$ in~$\aav$. Because the final two entries of~$\aav^{\nn - 1}$ are trivial, the last two steps take a simpler form: $\aa^{\nn - 1}_\nn = 1$ is equivalent to~$\aa_\nn = \xx_{\nn - 1}$, whereas $\aa^{\nn - 1}_{\nn - 1} = 1$ is equivalent to~$\xx_{\nn - 2}\inv\aa_{\nn - 1} \divet \aa_\nn$. 

\noindent\begin{minipage}{\textwidth}
\rightskip80mm We \VR(3,0) thus obtain an $\MM$-labeling of an annular graph made of $\nn - 2$ copies of~$\UG4$, whose outer boundary is labeled~$\aav$, and whose inner boundary is labeled~$\bbv$, see the diagram on the right, here in the case going from~$8$ to~Ê$6$. Then \VR(3,0) we repeat the process with~$\bbv$, etc., until a $4$-multifraction is reached, and we conclude using Proposition~\ref{P:EquivFoc}. This \VR(3,0) construction exactly corresponds to the inductive definition of~$\UG\nn$.\ \hfill$\square$%
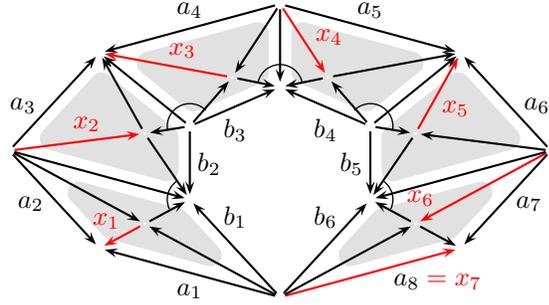
\begin{picture}(0,0)(-2,-2)
\def\NPoint(#1,#2,#3){\cnode[style=thin,fillcolor=white,linecolor=white](#1,#2){0}{#3}}
\def\AArrow(#1,#2){\ncline[linewidth=0.8pt]{->}{#1}{#2}}
\def\BArrow(#1,#2){\ncline[linecolor=color3, linewidth=0.8pt]{->}{#1}{#2}}
\psset{nodesep=0.7mm}
\psset{xunit=1.2mm, yunit=0.65mm}
\pspolygon[linearc=3,linewidth=3mm,linecolor=white,fillstyle=solid,fillcolor=color1](30,0)(10,10)(0,30)(20,20)
\pspolygon[linearc=3,linewidth=3mm,linecolor=white,fillstyle=solid,fillcolor=color1](20,20)(0,30)(10,50)(20,35)
\pspolygon[linearc=3,linewidth=3mm,linecolor=white,fillstyle=solid,fillcolor=color1](20,35)(10,50)(30,60)(30,43)
\pspolygon[linearc=3,linewidth=3mm,linecolor=white,fillstyle=solid,fillcolor=color1](40,35)(50,50)(30,60)(30,43)
\pspolygon[linearc=3,linewidth=3mm,linecolor=white,fillstyle=solid,fillcolor=color1](40,20)(60,30)(50,50)(40,35)
\pspolygon[linearc=3,linewidth=3mm,linecolor=white,fillstyle=solid,fillcolor=color1](30,0)(50,10)(60,30)(40,20)
\NPoint(30, 0, v0) \NPoint(10, 10, v1) \NPoint(0, 30, v2) \NPoint(10, 50, v3) \NPoint(30, 60, v4) \NPoint(50, 50, v5) \NPoint(60, 30, v6) \NPoint(50, 10, v7)
\NPoint(30, 0, w0) \NPoint(20, 20, w1) \NPoint(20, 35, w2) \NPoint(30, 43, w3) 
\NPoint(40, 35, w4) \NPoint(40, 20, w5) 
\NPoint(15, 15, x1) \NPoint(15, 33.5, x2) \NPoint(24.7, 45, x3) \NPoint(35.3, 45, x4) 
\NPoint(45, 33.5, x5) \NPoint(45, 15, x6) 
\psarc[style=thin](20, 20){3}{125}{205}
\psarc[style=thin](20, 35){3}{45}{190}
\psarc[style=thin](30, 43){3}{10}{170}
\psarc[style=thin](40, 35){3}{-10}{135}
\psarc[style=thin](40, 20){3}{-25}{55}
\AArrow(v0,v1)\tbput{$\aa_1$}
\AArrow(v2,v1)\tlput{$\aa_2$}
\AArrow(v2,v3)\tlput{$\aa_3$}
\AArrow(v4,v3)\taput{$\aa_4$}
\AArrow(v4,v5)\taput{$\aa_5$}
\AArrow(v6,v5)\trput{$\aa_6$}
\AArrow(v6,v7)\trput{$\aa_7$}
\BArrow(v0,v7)\nbput[npos=0.8]{$\aa_8 \color{color3} = \xx_7$}
\AArrow(w0,w1)\taput{$\bb_1$}
\AArrow(w2,w1)\trput{$\bb_2$}
\AArrow(w2,w3)\tbput{$\bb_3$}
\AArrow(w4,w3)\tbput{$\bb_4$}
\AArrow(w4,w5)\tlput{$\bb_5$}
\AArrow(w0,w5)\taput{$\bb_6$}
\AArrow(v0,x1) \AArrow(v2,x1) 
\BArrow(x1,v1)\nbput[npos=0.6]{\color{color3}$\xx_1$} \AArrow(x1,w1) 
\AArrow(v2,w1)
\AArrow(v0,x6) \BArrow(v6,x6)\nbput[npos=0.9]{\color{color3}$\xx_6$} 
\AArrow(x6,v7) \AArrow(x6,w5) 
\AArrow(v6,w5)
\BArrow(v2,x2)\naput[npos=0.6]{\color{color3}$\xx_2$}\AArrow(w2,x2)\AArrow(x2,v3)\AArrow(x2,w1)
\AArrow(w2,v3)
\AArrow(v6,x5)\AArrow(w4,x5) \BArrow(x5,v5)\nbput[npos=0.4]{\color{color3}$\xx_5$} \AArrow(x5,w5)
\AArrow(w4,v5)
\AArrow(v4,x3) \AArrow(w2,x3)
\BArrow(x3,v3)\nbput[npos=0.4]{\color{color3}$\xx_3$} \AArrow(x3,w3)
\AArrow(v4,w3)
\BArrow(v4,x4)\naput[npos=0.6]{\color{color3}$\xx_4$} \AArrow(w4,x4)\AArrow(x4,v5)\AArrow(x4,w3)
\end{picture}
\end{minipage}
\def\qed{\relax}\end{proof}

It is well-known that, if a multifraction represents~$1$ in a group, then it admits a van Kampen diagram in the sense defined above. What is remarkable here is the existence of one single universal shape, with prescribed springs, wells, and $4$-prongs, that works for every unital $\nn$-multifraction at the same time. 

Finally, one may wonder whether some counterparts of Lemma~\ref{L:CrossRed} and~\ref{L:RedCross} might hold with~$\UG\nn$ replacing~$\UG4$: maybe they do, but the natural argument for proving them requires that the ground monoid satisfies the $3$-Ore condition, in which case it is known that every unital $\nn$-multifraction admits a van Kampen diagram of shape~$\UG\nn$, thus making the results trivial.

\section{Miscellanea}\label{S:Misc}

The main three properties addressed in this paper are Conjectures~$\ConjA$,~$\ConjB$, and~$\ConjC$ (together with the uniform version~$\ConjCunif$ of the latter), which involve arbitrary Artin-Tits monoids, and are known to be true for those of FC~type. Testing these statements with a computer is easy, and we report about experiments that, alltogether, support the conjectures and provide some experimental evidence. The involved program is available at~\cite{Diw}, and the experiments are easy to repeat and confirm. 

We begin with a few remarks about implementation options (Subsection~\ref{SS:Options}), then report about the obtained data (Subsection~\ref{SS:Data}). Finally, we conclude with a few hints about further properties of reduction, including several counter-examples (Subsection~\ref{DeadEnds}).

\subsection{Implementation options}\label{SS:Options}

\subsubsection*{Choice of the monoid}

We are interested in Artin-Tits monoids~$\MM$ such that $\RDb\MM$ is not convergent, hence not of FC~type. It is natural to look for monoids with a maximal convergence defect, meaning that the proportion of multifractions with more than one irreducible reduct is maximal. As can be expected, the ratio is maximal for the Artin-Tits monoid of type~$\Att$, and, more generally, those with all relations of length~$3$ exactly: as divergent reducts may arise only with counter-examples to the $3$-Ore property, it is natural that this happens more frequently when all atoms give rise to such counter-examples. Another advantage of~$\Att$ and, more generally, $K_{\nn, 3}$, the Artin-Tits monoid whose Dynkin diagram is the complete graph with $\nn$ vertices and all edges labeled~$3$, is the existence of an explicit description of basic elements (namely $1$, atoms, and products of two distinct atoms) providing a better efficiency (and 100\% correctness with no termination problem) for the implementation of the monoid operations (equality, lcms, gcds, etc). Therefore, we mostly concentrated on~$\Att$, considered as the critical type (but any other choice is possible with~\cite{Diw}).


\subsubsection*{Generation of random multifractions}

Exhaustively enumerating multifractions up to a given length (sum of the lengths of the entries) is difficult, as, even in the case of $3$~atoms, there are more than $2.6 \times 10^9$ multifractions of length up to~$12$. Therefore it is more realistic to use samples of random multifractions. Generating random elements of the monoid and, from there, random multifractions, is easy via random words in the atom alphabet (with a bias due to the relations). 

Generating random unital multifractions is more delicate. As the density of unital multifractions is negligible, generating random multifractions and selecting those that are unital is not a good option (in addition, it requires a prior solution to the word problem, which exists for~$\Att$ but not in general). Two methods have been used. The first one (``brownian motion'') is to follow the definition of~$\simeq$, thus starting with an empty word and randomly adding or deleting pairs $\ss/\ss$ and applying the Artin-Tits relations. Inserting right and left reversing steps (the special transformations of Property~$\PropH$, see Subsection~\ref{SS:SemiAppli}) improves the efficiency.

The second method (``lcm-expansions'') consists in starting from a multifraction that admits a random central cross, hence is unital of a very special type, and deriving new, more generic, unital multifractions as follows:

\begin{defi}\label{D:DivExp}
(Figure~\ref{F:DivExp}) If $\MM$ is a gcd-monoid and $\aav$ is a unital $\nn$-multifraction on~$\MM$, with $\nn$ even, we say that $\bbv$ is an \emph{lcm-expansion} of~$\aav$ if $\dh\bbv = \dh\aav$ holds and, for each~$\ii$, there exist decompositions $\aa_\ii = \aa'_\ii \aa''_\ii$, $\bb_\ii = \bb'_\ii \bb''_\ii$ satisfying $\aa'_\ii \bb''_{\ii - 1} = \aa'_{\ii + 1} \bb''_ \ii = \aa'_\ii \lcm \aa'_{\ii + 1}$ for $\ii$ negative in~$\aav$, and $\bb'_{\ii - 1} \aa''_\ii = \bb'_ \ii \aa''_{\ii + 1} = \aa''_\ii \lcmt \aa''_{\ii + 1}$ for $\ii$ positive in~$\aav$, with indices modulo~$\nn$, \ie, $\nn + 1$ means~$1$.
\end{defi}

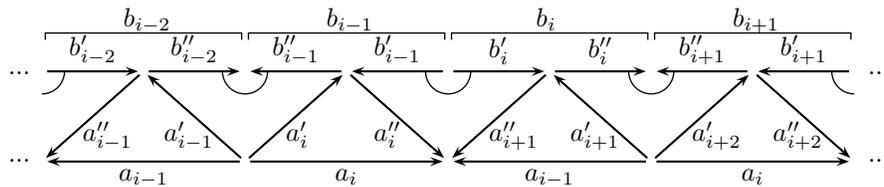
\begin{figure}[htb]
\begin{picture}(108,19)(0,-1)
\psset{nodesep=0.7mm}
\setlength{\unitlength}{0.9mm}
\psset{xunit=0.9mm, yunit=0.6mm}
\put(-5,0){...}\put(-5,13){...}
\put(122,0){...}\put(122,13){...}
\pcline{<-}(0,0)(30,0)\tbput{$\aa_{\ii - 1}$}
\pcline{->}(30,0)(60,0)\tbput{$\aa_{\ii}$}
\pcline{<-}(60,0)(90,0)\tbput{$\aa_{\ii - 1}$}
\pcline{->}(90,0)(120,0)\tbput{$\aa_{\ii}$}
\pcline{<-}(0,0)(15,20)\put(6,3){$\aa''_{\ii - 1}$}
\pcline{<-}(15,20)(30,0)\put(18,3){$\aa'_{\ii - 1}$}
\pcline{->}(30,0)(45,20)\put(36,3){$\aa'_{\ii}$}
\pcline{->}(45,20)(60,0)\put(49,3){$\aa''_{\ii}$}
\pcline{<-}(60,0)(75,20)\put(66,3){$\aa''_{\ii + 1}$}
\pcline{<-}(75,20)(90,0)\put(78,3){$\aa'_{\ii + 1}$}
\pcline{->}(90,0)(105,20)\put(96,3){$\aa'_{\ii + 2}$}
\pcline{->}(105,20)(120,0)\put(108,3){$\aa''_{\ii + 2}$}
\pcline{->}(0,20)(15,20)\taput{$\bb'_{\ii - 2}$}
\pcline{->}(15,20)(30,20)\taput{$\bb''_{\ii - 2}$}
\pcline{<-}(30,20)(45,20)\taput{$\bb''_{\ii - 1}$}
\pcline{<-}(45,20)(60,20)\taput{$\bb'_{\ii - 1}$}
\pcline{->}(60,20)(75,20)\taput{$\bb'_{\ii}$}
\pcline{->}(75,20)(90,20)\taput{$\bb''_{\ii}$}
\pcline{<-}(90,20)(105,20)\taput{$\bb''_{\ii + 1}$}
\pcline{<-}(105,20)(120,20)\taput{$\bb'_{\ii + 1}$}
\psarc[style=thin](0,20){3}{270}{360}
\psarc[style=thin](30,20){3}{180}{360}
\psarc[style=thin](60,20){3}{180}{360}
\psarc[style=thin](90,20){3}{180}{360}
\psarc[style=thin](120,20){3}{180}{270}
\psline[style=thin](0.5,27)(0.5,28.5)(29.5,28.5)(29.5,27)\put(12,20){$\bb_{\ii - 2}$}
\psline[style=thin](30.5,27)(30.5,28.5)(59.5,28.5)(59.5,27)\put(42,20){$\bb_{\ii - 1}$}
\psline[style=thin](60.5,27)(60.5,28.5)(89.5,28.5)(89.5,27)\put(73,20){$\bb_{\ii}$}
\psline[style=thin](90.5,27)(90.5,28.5)(119.5,28.5)(119.5,27)\put(102,20){$\bb_{\ii + 1}$}
\end{picture}
\caption{\small Lcm-expansion, here for $\ii$ positive in~$\aav$; the shift of the indices for~$\bbv$ ensures that $\bbv$ has the same sign as~$\aav$. Note the symmetry: starting with right divisors leads to the same notion.}
\label{F:DivExp}
\end{figure}

For $\bbv$ an lcm-expansion of~$\aav$, one reads on Figure~\ref{F:DivExp} the equality $\can(\bbv) = \can(\aa'_1 \bb''_\nn)\inv \can(\aav) \can(\aa'_1 \bb''_\nn)$, hence lcm-expansion preserves unitality. Constructing lcm-expansions of~$\aav$ is easy: with the notation of Definition~\ref{D:DivExp}, once a left divisor~$\aa'_\ii$ of~$\aa_\ii$ is chosen for each~$\ii$, all the remaining elements~$\aa''_\ii, \bb'_\ii, \bb''_\ii$ are determined, and then so is~$\bbv$ (but the choice leads to an lcm-expansion only if the lcms exist). The advantage is that the depth is controlled (which is more difficult with brownian motion), the inconvenience is that there is no guarantee that generic multifractions are obtained (but the expansion procedure can be iterated).

\subsubsection*{Maximal \vs atomic reduction steps}

Implementing reduction is straightforward, once the lattice operations of the monoid are available. As a composition of $\ii$-reductions is again an $\ii$-reduction, one might think of restricting to maximal reduction steps, \ie, considering reducts $\aav \act \Red{\ii, \xx}$ where $\xx$ is a maximal $\ii$-reducer for~$\aav$. This is not a good option, as some reducts may be missed:

\begin{exam}
In the Artin-Tits monoid of type~$\Att$, let $\aav = \fr{ab/ba/ca/bcbc}$. The only maximal reductions from~$\aav$ are $\Red{3,\fr{cc}}$ followed by~$\Rdiv{1,\fr{ab}}$, leading to the unique irreducible~$\fr{\1/ab/ca/cb}$. However, one also finds $\aav \act \Red{3,\fr{c}} \Rdiv{1,\fr{ab}} \Red{2,\fr{c}} \Red{3,\fr{b}} = \fr{cb/abbc/ba/bc}$, a second irreducible reduct of~$\aav$ not reachable by maximal steps.
\end{exam}

This however does not contradict the following (surprising?) result:

\begin{prop}\label{P:MaxSteps}
If $\MM$ is a strongly noetherian gcd-monoid, and $\bbv, \ccv$ are irreducible reducts of some multifraction, then there exists a finite sequence of maximal reductions and inverses of maximal reductions connecting~$\bbv$ to~$\ccv$.
\end{prop}

\begin{proof}[Proof (sketch)]
Write $\bbv \bowtie \ccv$ when there exists a sequence as in the statement. Using induction on the ordinal~$\wit(\aav)$, where $\wit$ satisfies~\eqref{E:StrWit}, one proves the following general criterion: If $\bowtie$ is an equivalence relation on~$\FRb\MM$ such that, for every multifraction~$\aav$ that is $\ii$-prime and $\jj$-irreducible for~$\jj \not= \ii$ and every $1$-reduct~$\aav'$ of~$\aav$, there exists a irreducible reduct~$\bbv$ of~$\aav'$ satisfying $\aav \equ \bbv$, then, for every~$\aav$ in~$\FRb\MM$, any two irreducible reducts of~$\aav$ are $\bowtie$-equivalent.
\end{proof}

\subsection{Experimental data}\label{SS:Data}

\subsubsection*{Conjectures~$\ConjA$ and~$\ConjB$}

Testing the two of them is essentially the same thing: one generates a random unital multifraction~$\aav$, and one checks $\aav \rds \one$ in the former case, $\redtame(\aav) = \one$ is the latter. Because semi-convergence implies $1$-confluence, one can fix any reduction strategy for checking~$\aav \rds \one$, for instance looking at each step for the smallest level~$\ii$ and the first atom~$\xx$ such that $\Red{\ii, \xx}$ applies. Although computing $\redtame$ is slightly slower, as, at each step, all $\ii$-reducers have to be determined in order to take their gcds, both computations are fast. Precise numbers are not really significant here; in type~$\Att$ or $K_{4, 3}$, the typical order of magnitude is $5 \times 10^4$ (\resp $1.5 \times 10^4$) random unital multifractions of length~$20$ and depth~$4$ (\resp length $40$ and depth~$6$) per hour of computation. In other types (\eg, $\Aft{A}3$, or $\Aft{C}2$), efficiency is diminished by a factor~$10$ approximately.

No counter-example to Conjecture~$\ConjA$ or~$\ConjB$ was ever found. As the density of visited (unital) multifractions becomes negligible when the length grows, the significance of such data is questionable. However, it may be noticed that, for the many properties considered in this paper and discarded by counter-examples, the length of the latter (all found by random search) is never more than~$12$ or so: this does not say anything about a possible counter-example to Conjecture~$\ConjA$ or~$\ConjB$ but, at the least, this shows that the considered lengths are not ridiculous.

\subsubsection*{Conjecture~$\ConjA_4$}

For the special case of Conjecture~$\ConjA_4$ (and of Conjecture~$\ConjB_4$, which, by Corollary~\ref{C:A4B4}, is equivalent), an exhaustive search makes sense, by systematically considering all possibilities for the central cross (up to a certain length).

\begin{fact}
For the Artin-Tits monoid of type~$\Att$, Conjecture~$\ConjA_4$ is true for all lcm-expansions of all multifractions that admits a central cross with entries of length at most~$2$.
\end{fact}

By contrast, the procedure applied to the monoid~$\MM_{C, 4}$ of~\cite[Proposition~6.9]{Div} duly finds a counter-example, namely an irreducible lcm-expansion of a $4$-multifraction with a central cross (with rays of length~$1$). What seems to discard a similar counter-example in an Artin-Tits monoid is the fact that Artin-Tits relations preserve the atoms occurring in an element (the ``support''), but, even for this weak form of Conjecture~$\ConjA_4$, we have no proof so far.

\subsubsection*{Conjectures~$\ConjC$ and~$\ConjCunif$}

Testing these conjectures is easier in that it involves arbitrary, not necessarily unital multifractions, but it is more difficult in that it requires to determine all right reducts of a multifraction~$\aav$ and then, for each pair of them, to determine all their left reducts. The complexity of constructing the tree~$T_{\aav}$ (as used in the proof of Proposition~\ref{P:FinitelyMany}) and its right counterpart~$\widetilde{T}_{\aav}$ increases quickly with the depth. Beyond depth~$4$, the total number of reducts often becomes large (usually a few ones, but possibly several thousands), resulting in a huge computation time. Typically, in the current version of the program, one can test Conjecture~$\ConjCunif$ for about $5 \times 10^3$ (\resp $4 \times 10^3$) random multifractions of length~$20$ (\resp length~$30$) and depth~$3$ per hour of computation. Going to depth~$4$ diminishes the speed by a factor~$50$.

To overcome these bounds, we also tested (without size limitation) the following instance of Conjecture~$\ConjC$: starting with a multifraction~$\aav$, the determine the (not necessarily distinct) right reducts~$\bbv_1 \wdots \bbv_4$ and left reducts $\cc_1 \wdots \cc_4$ of~$\aav$ obtained using the four natural strategies (levels from bottom or top, atoms in lexicographical or antilexicographical order), and check the property $\exists\kk\ \forall\jj\,(\bbv_\jj \rds \ccv_\kk)$, thus checking Conjecture~$\ConjC$ for $\{\bbv_1 \wdots \bbv_4\}$. The cost is then comparable as the one for Conjecture~$\ConjB$, with about $5 \times 10^4$ (\resp $5 \times 10^3$) random tries for length~$20$ (\resp $30)$ per hour of computation. The stronger conclusion $\forall\kk\ \forall\jj\,(\bbv_\jj \rds \ccv_\kk)$ is almost always valid but, as in the case of Figure~\ref{F:CrossConf}, exceptions occur.

\subsection{Further questions}\label{DeadEnds}

We point to a few natural questions involving reduction. Most of them remain open, or gave rise to counter-examples.

\subsubsection*{Normal forms for reductions}

Distinguished expressions for sequences of reductions could be obtained by identifying skew commutation relations, typically of the form $\Red{\ii, \xx} \Red{\jj, \yy} \Rrightarrow \Red{\jj, \yy'} \Red{\ii, \xx'}$, meaning that, if $\aav \act \Red{\ii, \xx} \Red{\jj, \yy}$ is defined, then so is $\aav \act \Red{\jj, \yy'} \Red{\ii, \xx'}$ and the results are equal. Typically, one could try to push divisions to one side, so that all remaining steps are invertible. This approach does not work well, as exceptions always appear. In the same vein, in view of a possible induction and building on the universal scheme~$\univ\nn$ that works in the $3$-Ore case, one could conjecture that every reduction sequence is equivalent to one where the highest level occurs only once, or that, if an $\nn$-multifraction~$\aav$ is $\ii$-irreducible for~$\ii < \nn - 1$, then reducing~$\aav$ can be done by a sequence of the form $\Red{\nn - 1, \xx_1}\Red{\nn - 3, \xx_3} \Red{\nn - 5, \xx_5} \pdots$. This need not be the case.

\begin{exam}
In the Artin-Tits of type~$\Att$, let $\aav = \fr{\1/ba/cb/ca/ab}$. Then $\aav$ is $\ii$-irreducible for $\ii \le 3$, but the only sequence from~$\aav$ to an irreducible reduct is~$\Red{4,\fr{a}} \Red{2, \fr{bc}} \Red{3, \fr{a}} \Red{4,\fr{b}}$, discarding the above two conjectures. Of course, Conjecture~$\ConjB$ is not contradicted, because $\aav$ is not unital.
\end{exam}

In the same direction, one can study local confluence between left reductions~$\Red{\ii, \xx}$ and right reductions~$\Redt{\jj, \yy}$, with the hope of obtaining normal forms useful for cross-conflence. In almost all cases, there exists indeed local confluence solutions. However, the case $\jj = \ii + 2$ remains problematic in general. Moreover, using local confluence for an induction is unclear, because there is no common well-founded relation underlying both left and right reduction and, except in type~FC, a multifraction may admit infinitely many left-right reducts.

\subsubsection*{Homomorphisms}

As reduction is constructed using multiplication and lcm operations, it is preserved by mor\-phisms preserving these operations, namely lcm-morph\-isms~\cite{CriLCM}: if $\phi$ is an lcm-morphism from a gcd-monoid~$\MM$ to a gcd-monoid~$\MM'$, then $\aav \rds \bbv$ implies $\phi(\aav) \rds \phi(\bbv)$ for all $\aav, \bbv$ in~$\FRb\MM$. This however is not easy to use for, say, Conjecture~$\ConjA$, because the implications go in the wrong direction. If if we study the semi-convergence of~$\RDb\MM$, mapping~$\MM$ to a gcd-monoid~$\MM'$ that satisfies the $3$-Ore condition does not help: if $\aav$ is unital in~$\FRb\MM$, then $\phi(\aav)$ is unital in~$\FRb{\MM'}$, so $\phi(\aav) \rds \one$ holds, but deducing~$\aav \rds \one$ is problematic. In the other direction, if $\MM'$ satisfies the $3$-Ore condition and $\phi'$ is an lcm-morphism from~$\MM'$ to~$\MM$, then $\mathrm{Im}\phi'$ is included in some part of~$\MM$ where the $3$-Ore condition is satisfied, and knowing that $\RDb{\MM'}$ is (semi)-convergent will not help for multifractions on~$\MM$ outside~$\mathrm{Im}\phi'$.

Morphisms (lcm-morphisms or not) might be useful for establishing particular properties, like Conjecture~$\ConjA_4$, \ie, the fact that every unital $4$-multifraction admits a central cross. If $\MM$ is the Artin-Tits monoid of type~$\Att$, one may think of using the classical embedding from~$\MM$ to the Artin-Tits group of type~$\mathrm{B}_3$, namely $\GR{\fr{a, b, c}}{\fr{abab=baba, bcb=cbc, ac=ca}}$, that maps~$\fr{a}$ to~$\fr{a\inv b\inv cba}$, and $\fr{b}$ and~$\fr{c}$ to themselves, but the image is not included in the monoid.
A probably better choice is to map~$\MM$ to the Artin-Tits monoid~$\MM'$ of type~$\mathrm{D}_4$, namely $\MON{\fr{a, b, c, d}}{\fr{ada=dad, bdb=dbd, cdc =dcd}}$, by $\phi(\ss) = \fr{d}\ss$ for $\ss = \fr{a, b, c}$. If $\aav$ is a unital $4$-multifraction on~$\MM$, then $\phi(\aav)$ admits a central cross in~$\MM'$: to deduce that $\aav$ admits a central cross in~$\MM$, it suffices to show that at least one of the central crosses for~$\phi(\aav)$ lies in~$\mathrm{Im}\phi$. The problem is that $\mathrm{Im}\phi$ is not closed under right divisor in~$\MM'$: for instance, $\phi(\fr{a}) \dive \phi(\fr{bac})$ holds in~$\MM'$, but $\fr{a} \dive \fr{bac}$ fails in~$\MM$. 

By the way, the following natural question seems to be open: 

\begin{ques}[F.\,Wehrung]
Does the above morphism~$\phi$ induce an embedding from the Artin-Tits group of type~$\Att$ into the Artin-Tits group of type~$\mathrm{D}_4$?
\end{ques}

\subsubsection*{Reduction graphs}

Almost nothing is known about the structural properties of the graph formed by the reducts of a multifraction, for instance their possible lattice properties: if we have $\aav \rds \bbv$ and $\aav \rds \ccv$ and if $\bbv$ and $\ccv$ admit a common reduct, does there exist a common reduct~$\ddv$ of~$\bbv$ and~$\ccv$ such that every common reduct of~$\bbv$ and~$\ccv$ is a reduct of~$\ddv$? This is frequently true, but not always:

\begin{exam}
In the Artin-Tits monoid of type~$\Att$, consider $\aav = \fr{\1/a/bcb/bcb/a}$ (which is unital). One finds $\bbv = \aav \act \Red{2,\fr{b}} = \fr{ba/ab/cb/bcb/a}$ and $\ccv = \aav \act \Red{2,\fr{c}} = \fr{ca/ac/bc/bcb/a}$. Then $\bbv$ and $\ccv$ admit the two symmetric maximal common reducts, namely
\begin{gather*}
\ddv' = \bbv \act \Rdiv{3,\fr{b}} = \ccv \act \Red{3,\fr{b}} \Rdiv{1, \fr{ca}} \Red{2, \fr{b}} = \fr{ba/ab/c/bc/a},\\
\ddv'' = \bbv \act \Red{3,\fr{c}} \Rdiv{1, \fr{ba}} \Red{2, \fr{c}} = \ccv \act \Rdiv{3,\fr{c}} = \fr{ca/ac/b/cb/a},
\end{gather*}
and there is no common reduct~$\ddv$ of~$\bbv$ and~$\ccv$ of which $\ddv'$ and $\ddv''$ are reducts.
\end{exam}

A possible conclusion in view of the long list of counter-examples described in this paper could be that there is no hope for many further general properties of reduction, implying that a possible proof of semi-convergence has to involve the specific properties of the ground monoid in a deep way. In particular, a proof of Conjecture~$\ConjA$, $\ConjB$, or~$\ConjC$ should require developing new specific tools for Artin-Tits groups. The results of~\cite{Div} may suggest approaches.

\bibliographystyle{plain}
\newcommand{\noopsort}[1]{}

\end{document}